\documentclass[preprint,12pt]{elsarticle}

\usepackage{amssymb}
\usepackage{graphicx}
\usepackage{euscript,color}
\usepackage{amssymb,amsfonts,amsmath,amscd,dsfont,mathrsfs,amsthm}
\usepackage{multirow}
\usepackage{subfig}
\usepackage{epstopdf}
\usepackage{algorithm}
\usepackage{algorithmic}[5]
\biboptions{sort&compress} 
\usepackage{url}

\newtheorem{thm}{Theorem}
\newtheorem{cor}{Corollary} %[subsection]
\newtheorem{lem}{Lemma} %[subsection]
\newtheorem{dfn}{Definition}
\newtheorem{prop}{Proposition}

\def\x{{\mathbf x}}

 \def\E{\mathds E}
 \def\P{\mathds P}

\journal{Applied and Computational Harmonic Analysis}

\begin{document}

\begin{frontmatter}



\title{{\bf Anisotropic Nonlocal Means Denoising}}


%\author{Arian Maleki\corref{cor1}\fnref{fn1}} \ead{arian.maleki@rice.edu} \ead[url]{http://dsp.rice.edu/}
%\author{Manjari Narayan} \ead{manjari@rice.edu} 
%\author{Richard G.\ Baraniuk\corref{cor2}} \ead{richb@rice.edu} 
%
%\address{Department of Electrical and Computer Engineering \\ Rice University}
%\cortext[cor2]{Principal corresponding author}
%
%
%\fntext[fn1]{Phone: +1 713.348.3579; Fax: +1 713.348.5685}

\author{Arian Maleki\corref{cor1}} \ead{arian.maleki@rice.edu} \ead[url]{http://www.ece.rice.edu/~mam15/}
\author{Manjari Narayan} \ead{manjari@rice.edu} \ead[url]{http://www.ece.rice.edu/~mn4/}
\author{Richard G.\ Baraniuk\corref{cor2}\fnref{fn3}} \ead{richb@rice.edu} \ead[url]{http://web.ece.rice.edu/richb/}

\address{Dept.~of Computer and Electrical Engineering, Rice University, MS-380 \\ 6100 Main Street, Houston, TX 77005, USA \\ }

\cortext[cor1]{Corresponding author}
\cortext[cor2]{Principal corresponding author}

%\fntext[fn1]{Phone: +1 713.348.3579; Fax: +1 713.348.5685}
%\fntext[fn2]{Phone: +1 713.348.2371; Fax: +1 713.348.5685}
\fntext[fn3]{Phone: +1 713.348.5132; Fax: +1 713.348.5685}

\begin{abstract}
It has recently been proved that the popular nonlocal means (NLM) denoising algorithm does not optimally denoise images with sharp edges.  Its weakness lies in the isotropic nature of the neighborhoods it uses to set its smoothing weights.  In response, in this paper we introduce several theoretical and practical {\em anisotropic nonlocal means} (ANLM) algorithms and prove that they are near minimax optimal for edge-dominated images from the Horizon class.  On real-world test images, an ANLM algorithm that adapts to the underlying image gradients outperforms NLM by a significant margin.
\end{abstract}


\end{frontmatter}

\section{Introduction}\label{sec:intro}

Image denoising is a fundamental primitive in image processing and computer vision. Denoising algorithms have evolved from the classical linear and median filters to more modern schemes like total variation denoising \cite{RuOsFa92}, wavelet thresholding \cite{JohnstoneMono}, and bilateral filters \cite{SmBr97,Yaroslavsky85,Tomasi:1998fk, Elad:2002kx}.  

A particularly successful denoising scheme is the \textit{nonlocal means} (NLM) algorithm \cite{Buades:2005p4221}, which estimates each pixel value as a weighted average of other, similar noisy pixels. However, instead of using spatial adjacency or noisy pixel value as the similarity measure to adjust the estimate weights, NLM uses a more reliable notion of similarity based on the resemblance of the pixels' neighborhoods in high-dimensional space.  This unique feature benefits NLM in two ways. First, it provides more accurate weight estimates. Second, it enables NLM to exploit the contribution of all pixels in the image. In concert, these features enable NLM to provide remarkable performance for a large class of image denoising problems.

\begin{figure}
\centering{
  \includegraphics[width=6.2cm]{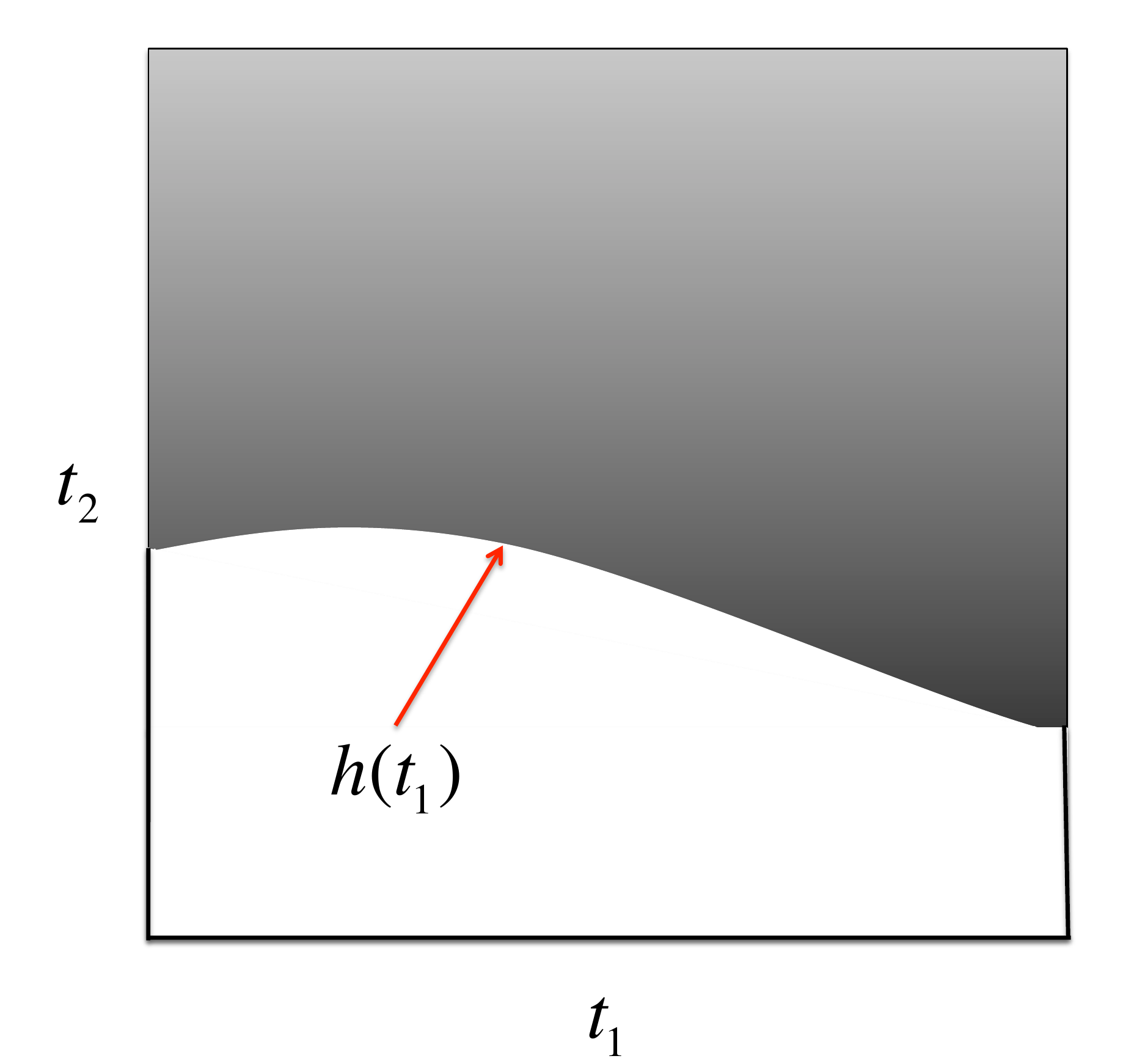} 
  \caption{An example of a Horizon class image that features a smooth edge contour that separates the white region from the black region.}
   \label{fig:horizonfunction}
  }
\end{figure}

Nevertheless, in a recent paper, we have proved that NLM does not attain optimal performance on images with sharp edges from the so-called {\em Horizon class} (see Figure \ref{fig:horizonfunction}) \cite{MaNaBa11}.\footnote{Recently \cite{CaSaWi12} has extended our results to more complicated image$\slash$edge models.} Indeed, NLM's theoretical performance is more or less equivalent to wavelet thresholding, which was shown to be suboptimal in \cite{Donoho:1999p1950}. Empirical results have also confirmed the suboptimality of NLM in estimating sharp edges \cite{DuAjGo11, SaSt10, DeDuSa11, DaFoKatEgi08}.  The core problem is that NLM (and wavelet thresholding) cannot exploit the smoothness of the edge contour that separates the white and black regions.

In this paper, we introduce and study a new denoising framework and prove that it is near-optimal for Horizon class images with sharp edges.  {\em Anisotropic nonlocal means} (ANLM) outperforms NLM, wavelet thresholding, and more classical techniques by using anisotropic neighborhoods that are elongated along and matched to the local edge orientation.  Figure \ref{fig:compareanisotropic} compares the neighborhoods used in ANLM with those
used in NLM.  Anisotropic neighborhoods enable ANLM to distinguish between similar and dissimilar pixels more accurately.

 \begin{figure}
\begin{center}
%\vspace{-1cm}
  \includegraphics[width=.8\textwidth]{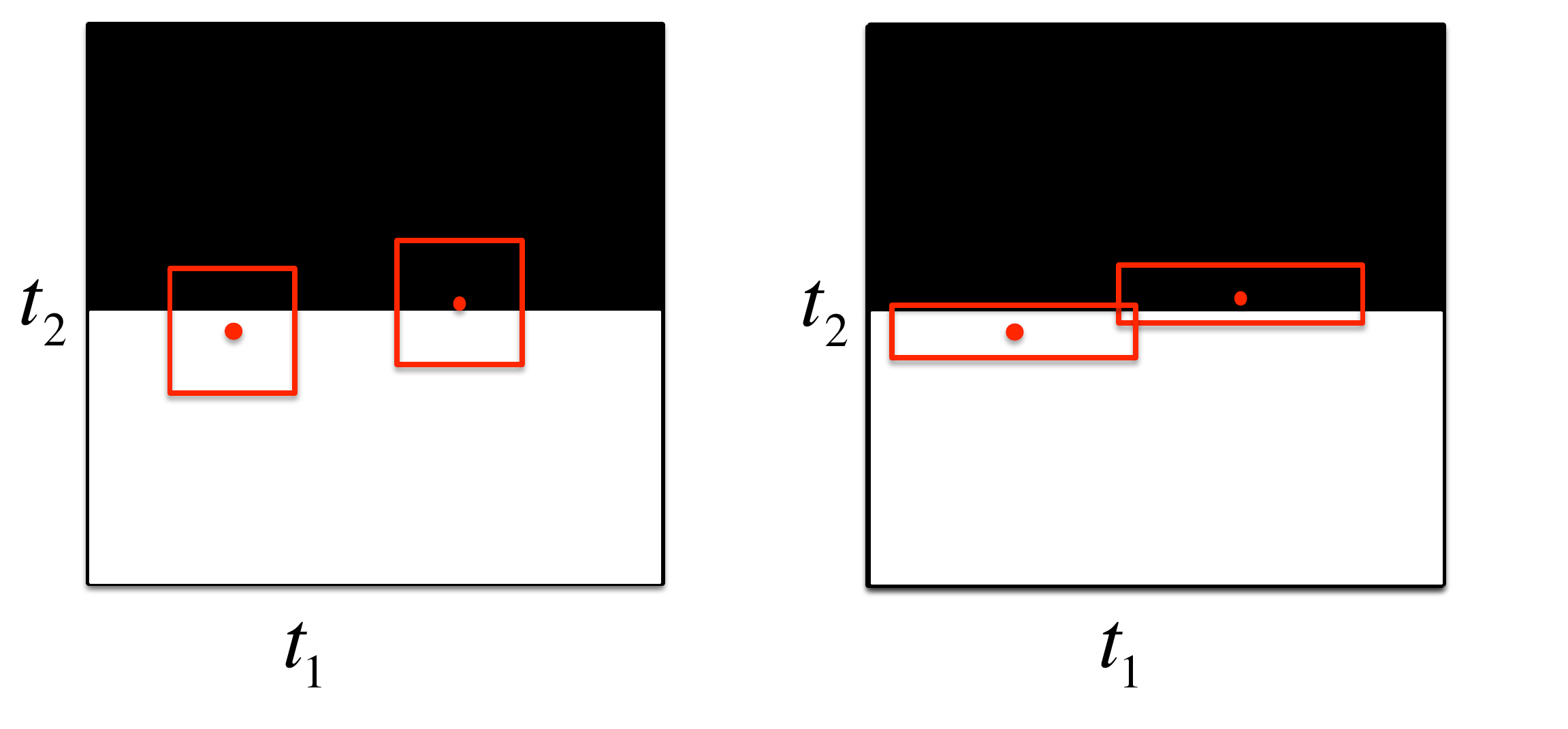} 
  %\vspace{-1.5cm}
  \caption{Comparison of (left) the isotropic neighborhoods employed by Non Local Means (NLM) versus (right) anisotropic neighborhoods employed by Anisotropic NLM (ANLM).}
   \label{fig:compareanisotropic}
  \end{center}
\end{figure}

We develop three different ANLM algorithms of increasing levels of practicality. \textit{Oracle Anisotropic Nonlocal Means} (OANLM) assumes perfect knowledge of the local orientation of the edge contour and is used primarily for our theoretical optimality analysis.  \textit{Discrete-angle Anisotropic Nonlocal Means} (DANLM) optimizes the choice of the anisotropic neighborhood around each pixel in order to achieve near-optimal performance without any oracle information.   Since it is more computationally demanding than NLM, we introduce an algorithmic simplification.  \textit{Gradient based Anisotropic Nonlocal Means} (GANLM) uses image gradient information to estimate the edge orientation; using simulations, we demonstrate that GANLM significantly outperforms NLM in practice on both Horizon class and real-world images.

Anisotropy is a fundamental feature of real-world images. Hence, many denoising algorithms have been proposed to exploit this feature \cite{NaMa79, Takeda06, TaFaMi07, PeMa90, Can99-1, Donoho:2000p1676, Wil03, Wil07, KuLa07, DoVe05, 1407972, Vel06, DeDuSa11, SaSt10, DaFoKatEgi08}. We will review all these algorithms and their similarities and differences with ANLM in Section \ref{sec:related}. 

The paper is organized as follows. Section \ref{sec:minimax} explains the minimax framework we use to analyze the denoising algorithms and reviews the necessary background. Section \ref{sec:anlm} introduces the OANLM and DANLM algorithms and presents the main theorems. Section \ref{sec:simulation} addresses some of the practical ANLM issues by introducing GANLM and summarizes the results of a range of simulations using synthetic and real-world imagery.  
Section \ref{sec:related} reviews the related work in the literature. Section \ref{sec:conc} discusses our current and potential future results.
Section \ref{sec:proofs} contains the proofs of the main theorems.

\section{Minimax analysis framework}\label{sec:minimax}

In this section we introduce the minimax framework \cite{JohnstoneMono, LeCa98} and the Horizon class image model considered in this paper.  Note that, in order to streamline the proofs, we take a continuous-variable analysis approach in this paper, in contrast to our approach in \cite{MaNaBa11}.  The moral of the story is compatible with \cite{MaNaBa11}, however.

\subsection{Risk}

We are interested in estimating an image described by the function $f: [0,1]^2 \rightarrow [0,1]$ ($f \in L^2([0,1]^2)$) from its noisy observation 
\begin{equation}
\label{eq:noisemodel}
dY(t_1, t_2) = f(t_1,t_2)dt_1 dt_2 + \sigma dW(t_1,t_2).
\end{equation}
Without loss of generality, we consider only square images.
Here $W(t_1,t_2)$ is the Wiener sheet,\footnote{The Wiener sheet is the primitive of  white noise.} and $\sigma$ is a constant that scales the noise. For a given function $f$ and a given estimator $\hat{f}$, define the \emph{risk} as
\[
R(f, \hat{f}) = \E (\|f- \hat{f} \|_2^2),
\] 
where the expected value is over $W$. The risk can be decomposed into bias squared and variance terms
\[
R(f, \hat{f}) = \|f- \E \hat{f}\|_2^2+ \E\|\hat{f}-\E\hat{f} \|_2^2.
\]

Let  $f$ belong to a class of functions $\mathcal{F}$. The risk of the estimator $\hat{f}$ on $\mathcal{F}$ is defined as the risk of the least-favorable function, i.e., 
\[
R(\mathcal{F}, \hat{f}) = \sup_{f \in \mathcal{F}} R(f, \hat{f}).
\]
The \textit{minimax} risk over the class of functions $\mathcal{F}$ is then defined as
\[
R^{*}(\mathcal{F}) = \inf_{\hat{f}}   \sup_{f \in \mathcal{F}} R(f, \hat{f}).
\]
$R^{*}(\mathcal{F})$ is a lower bound for the performance of any estimator on $\mathcal{F}$. 

In this paper, we are interested in the asymptotic setting
as $\sigma \rightarrow 0$. For all the estimators we consider, $R(\mathcal{F}, \hat{f}) \rightarrow 0$ as $\sigma \rightarrow 0$. Therefore, following \cite{PePeDoMa07,CaDo02Illposed} we
will consider the {\em decay rate of the minimax risk} as our measure of performance. For this purpose, we will use the following asymptotic notation.

\begin{dfn}
$f(\sigma) = O(g(\sigma))$ as $\sigma \rightarrow 0$, if and only if there exist $\sigma_0$ and $c$ such that for any $\sigma<\sigma_0$, $|f(\sigma)| \leq c |g(\sigma)|$. Likewise, $f(\sigma) = \Omega(g(\sigma))$ as $\sigma \rightarrow 0$, if and only if there exist $\sigma_0$ and $c$ such that for any $\sigma< \sigma_0$, $|f(\sigma)| \geq c |g(\sigma)|$. Finally, $f(\sigma) = \Theta(g(\sigma))$, if and only if $f(\sigma) = O(g(\sigma))$ and $f(\sigma)= \Omega(g(\sigma))$. We may interchangeably use $f(\sigma) \asymp g(\sigma)$ for $f(\sigma) = \Theta(g(\sigma))$.
\end{dfn}

\begin{dfn}
$f(\sigma) = o(g(\sigma))$ if and only if $\lim_{\sigma \rightarrow 0} \frac{f(\sigma)}{g(\sigma)} = 0$.
\end{dfn}

\subsection{Horizon edge model}\label{ssec:edgemodel}
In our analysis, we consider the {\em Horizon} model that contains piecewise constant images with sharp step edges lying along a smooth contour \cite{Donoho:1999p1950,PePeDoMa07,Tsybakov:1993uq} (our analysis extends easily to piecewise smooth edges).  Let ${\sl H\ddot{o}lder}^{\alpha}(C)$ be the class of H$\rm \ddot{o}$lder functions on $\mathds{R}$, defined in the following way: $h \in {\sl H\ddot{o}lder}^{\alpha}(C)$ if and only if
%For $0 < \alpha <1$, $f \in Holder^{\alpha}(C)$ if and only if
%\[
%|h(x)-h(y)| \leq C|x-y|^{\alpha}. 
% \]
%For $\alpha>1$, the function will be $k= \lfloor \alpha \rfloor$ times differentiable and,
\[
|h^{(k)}(t_1)-h^{(k)}(t'_1)| \leq C|t_1-t'_1|^{\alpha-k},
\]
where $k = \lfloor \alpha \rfloor $. Consider a transformation that maps each one-dimensional edge contour function $h$ to a two-dimensional image $f_h: [0,1]^2 \rightarrow [0,1]$ via
\begin{eqnarray*}
 f_h(t_1,t_2) = \textbf{1}_{\{t_2 < h(t_1)\}}.
 \end{eqnarray*}
The Horizon class of images is then defined as 
\begin{align}
H^{\alpha} (C) = \{f_h(t_1,t_2) : h \in  {\sl H\ddot{o}lder}^{\alpha}(C) \cap {\sl H\ddot{o}lder}^{1}(1)  \},
\end{align} 
where $\alpha$ is the smoothness of the edge contour.  Figure \ref{fig:horizonfunction} illustrates a sample function of this class. The following theorem characterizes the minimax rate of $H^{\alpha} (C)$  \cite{Donoho:1999p1950,PePeDoMa07,Tsybakov:1993uq}.\footnote{The models considered in  \cite{Donoho:1999p1950,PePeDoMa07,Tsybakov:1993uq} are slightly different from the continuous framework of this paper. Therefore, for the sake of completeness we prove Theorem \ref{thm:minimax} in \ref{app:proofminimax}. } 

\bigskip\begin{thm}\label{thm:minimax}
{\rm \cite{Donoho:1999p1950,PePeDoMa07,Tsybakov:1993uq}} For $\alpha \geq1$, the minimax risk of the class $H^{\alpha} (C)$ is
\[
R^{*}(H^{\alpha} (C)) = \Theta\!\left(\sigma^{\frac{2\alpha}{\alpha+1}}\right). 
\]
\end{thm}

Achieving the minimax rate is a laudable goal that any well-respecting denoising algorithm should aspire to.  In this paper, we will focus primarily on $\alpha =2$ edge contours, for which the optimal minimax decay rate is $\sigma^{4/3}$. However, it is straightforward to draw similar conclusions for the other values of $\alpha$. 

\subsection{Image denoising algorithms}\label{sec:imden}

Minimax risk analysis of the classical denoising algorithms has revealed their suboptimal performance on images with sharp edges.  Table \ref{tabl:denoise} summarizes several of these algorithms and their minimax decay rates (up to a log factor).\footnote{The analysis framework used in \cite{CaDo09} and  \cite{MaNaBa11} is discrete rather than continuous.  Therefore, for the sake of completeness we establish the results for the mean filter and NLM in \ref{app:meanfilter} and \ref{app:nlm}, respectively.} The suboptimality of wavelet thresholding has led to the development of ridgelets \cite{Can99-1}, curvelets \cite{Donoho:2000p1676}, wedgelets \cite{Donoho:1999p1950}, platelets \cite{Wil03}, shearlets \cite{KuLa07}, contourlets \cite{DoVe05}, bandelets \cite{1407972}, directionlets \cite{Vel06} and other types of directional transforms.  See \cite{RaMaPo11, JaDuCuPe11} and the references therein for more information.  In the framework of this paper, wedgelet denoising \cite{Donoho:1999p1950} provably achieves the optimal minimax rate. However, it performs poorly on textures, which has limited its application in image processing.

% Appendices references \ref{app:meanfilter} and \ref{app:nlm}

\begin{table}[t] \footnotesize
\caption{Minimax risk decay rates of several classical image denoising algorithms; recall that the according to Theorem \ref{thm:minimax}, for $\alpha=2$ the optimal minimax rate is $\sigma^{4/3}$.}
\begin{center}
\scalebox{1.2}{
\begin{tabular}{| l | c |}
	\hline
	 Algorithm & Minimax Rate  \\ \hline \hline
	 Mean filter \cite{CaDo09}& $\sigma^{2/3}$ \\ \hline
	 Wavelet thresholding \cite{Donoho:1999p1950,PePeDoMa07} & $\sigma^{1}$ \\ \hline
	 Nonlocal means \cite{MaNaBa11}& $\sigma^{1}$ \\ \hline
	% Minimax rate \cite{PePeDoMa07,Tsybakov:1993uq} & $\sigma^{-4/3} $ \\
\end{tabular}
}
\end{center}
\label{tabl:denoise}
\end{table}

\subsection{Nonlocal means denoising}

In 2005, Buades, Coll, and Morel significantly improved the performance of the bilateral filter \cite{Tomasi:1998fk} by incorporating a new notion of pixel similarity.  As in the mean filter, NLM estimates each pixel value using a weighted average of other pixel values in the image. However, NLM sets the weights according to the similarity between the pixel neighborhoods rather than the pixel values.  Furthermore, in contrast to the bilateral filter, in which only the vicinity of each pixel contributes to the estimate, in NLM all pixels may contribute.  

Specifically, the NLM algorithm is defined as follows. Let $S \triangleq [0,1]^2$ represent the domain of the image. Define the $\delta$-neighborhood of a $(t_1, t_2)$ as
\begin{eqnarray}
I_{\delta}(t_1,t_2) \triangleq  \left[t_1- \frac{\delta}{2}, t_1+\frac{\delta}{2} \right] \times \left[t_2- \frac{\delta}{2}, t_2+\frac{\delta}{2}\right].
\end{eqnarray}
Following the definition of NLM in the discrete setting, we pixelate this neighborhood by partitioning $I_{\delta}(t_1,t_2)$ into $n^2$ subregions $I^{j_1,j_2}_{\delta} \triangleq  [t_1- \frac{j_1\delta}{2n}, t_1+\frac{j_1\delta}{2n}] \times [t_2- \frac{j_2\delta}{2n}, t_2+\frac{j_2\delta}{2n}]$. The pixelated neighborhood of $(t_1,t_2)$ is defined as $\mathbf{y}^{\delta}_{t_1,t_2} \in \mathds{R}^{n \times n}$ and satisfies
\[
\mathbf{y}_{t_1,t_2}^{\delta}(j_1,j_2) \triangleq \frac{n^2}{\delta^2}\int_{(s_1,s_2) \in I^{j_1,j_2}_{\delta} } dY(s_1,s_2).
\]
We further define the pixelated process as 
\[
X(t_1,t_2) \triangleq \frac{n^2}{\delta^2} \int_{(s_1,s_2) \in I_{\delta/n}(t_1,t_2) } dY(s_1,s_2).
\]
Define the $\delta$-neighborhood distance between two points in the image as
\begin{eqnarray}
 \lefteqn{d^2_{\delta}(dY(t_1,t_2), dY(s_1,s_2))} \nonumber \\
  \! \!&\triangleq& \! \! \! \frac{1}{n^2-1} \left( \| \mathbf{y}_{t_1,t_2}^{\delta}- \mathbf{y}_{s_1,s_2}^{\delta}  \|_2^2 
 - |\mathbf{y}_{t_1,t_2}^{\delta}(0,0) - \mathbf{y}_{s_1,s_2}^{\delta}(0,0)|^2\right),
\end{eqnarray}
where $\|\mathbf{y}^{\delta}_{t_1,t_2}\|_2^2 \triangleq \sum_{i,j} (\mathbf{y}^{\delta}_{t_1,t_2}(i,j))^2$.
Note that in contrast to the definition in \cite{Buades:2005p4221}, we have removed the center element $|\mathbf{y}_{t_1,t_2}^{\delta}(0,0) - \mathbf{y}_{s_1,s_2}^{\delta}(0,0)|^2$ from the summation. As we will see in Section \ref{sec:anlm}, $n^2 \rightarrow \infty$ as $\sigma \rightarrow 0$, and hence the effect of removing the center pixel is negligible on the asymptotic performance. But, as we will see below,
removing the center element simplifies the calculations considerably. NLM uses the neighborhood distances to estimate
\[
\hat{f}^{N}(t_1,t_2) \triangleq \frac{\int_{(s_1,s_2) \in S} w^N_{t_1,t_2} (s_1,s_2)X(s_1,s_2) ds_1 ds_2}{\int_{(s_1,s_2) \in S} w^N_{t_1,t_2}(s_1,s_2) ds_2 ds_2 },
\]
where $w^N_{t_1,t_2} (s_1,s_2)$ is set according to the $\delta$-neighborhood distance between $dY(t_1,t_2)$ and $dY(s_1,s_2)$.\footnote{We assume that both $w^N_{t_1,t_2} (s_1,s_2)X(s_1,s_2)$ and $w^N_{t_1,t_2}(s_1,s_2)$ are Lebesgue integrable with high probability.}  It is straightforward to
verify that $\E d^2_{\delta}(dY(t_1,t_2),$ $dY(s_1,s_2)) =  d^2_{\delta}(f(t_1,t_2), f(s_1,s_2)) +\frac{2n^2\sigma^2}{\delta^2}$, which suggests the following strategy 
for setting the weights:
\begin{eqnarray}
w^{N}_{t_1,t_2}(s_1,s_2) \triangleq \! \! \left\{\begin{array}{rl}
 1 &  \mbox{ if $d^2_{\delta} (dY(t_1,t_2),dY(s_1,s_2)) \leq \frac{2n^2\sigma^2}{\delta^2}+ \tau_{\sigma} $,}\\
 0 &   \mbox{ otherwise,}
\end{array}\right.
\end{eqnarray}
where $\tau_\sigma$ is the \textit{threshold parameter}. Soft/tapered weights have been explored and are often used in practice \cite{Buades:2005p4221}. However, the above 
untapered weights capture the essence of the algorithm while simplifying the analysis.  

The distinguishing feature of NLM --- the weighted averaging of pixels based on the neighborhoods --- produces a decay rate that is superior to that of linear filters as shown in  \cite{MaNaBa11}. However, compared to the optimal rate of $\sigma^{4/3}$, NLM remains suboptimal. We introduce the anisotropic NLM algorithm in the next section to address this gap in performance.

\section{Anisotropic nonlocal means denoising}\label{sec:anlm}

In this section, we introduce the Anisotropic Nonlocal Means (ANLM) concept that exploits the smoothness of edge contours via
anisotropic neighborhoods. We then present OANLM and DANLM algorithms and explain their performance guarantees.  

\subsection{Directional neighborhoods}\label{ssec:dirnei}

We now formally introduce the notion of directional neighborhoods. Then we extend NLM to exploit such neighborhoods.  This will lay the foundation for the OANLM and DANLM algorithms. 

Let $R^{\theta}_{\nu,\mu}(\cdot)$ represent the rotation operator on a neighborhood; when applied to a generic point $(u,v) \in {S}$ in the neighborhood around $(\nu,\mu)$, $R^{\theta}_{\nu,\mu}$ rotates $(u,v)$ by $\theta^{\circ}$ counter-clockwise around the point $(\nu,\mu)$. For a set $Q \subset {S}$ we define $Q_{\theta} \triangleq R^{\theta}_{\nu,\mu}(Q)$ as
\[
(s,t) \in Q_{\theta} \Leftrightarrow \exists (u,v) \in Q \  {\rm such \ that}  \ (s,t) = R^{\theta}_{\nu,\mu}((u,v)).
\]

\noindent The \textit {$(\theta, \delta_s, \delta_{\ell})$-anisotropic neighborhood} of the point $(t_1,t_2)$ is defined as
\begin{eqnarray*}
{I}_{\theta, \delta_s, \delta_{\ell}}(t_1,t_2) \triangleq R^{\theta}_{t_1,t_2}\left(\left[t_1- \frac{\delta_{\ell}}{2}, t_1+ \frac{\delta_{\ell}}{2}\right] \times\left[t_2- \frac{\delta_{s}}{2}, t_2+ \frac{\delta_{s}}{2}\right] \right) \cap S,
\end{eqnarray*}
where $\theta$, $\delta_{\ell}$, and $\delta_{s}$ ($\delta_s \leq \delta_{\ell}$) represent the orientation angle, length, and width of the neighborhood, respectively. Figure \ref{fig:anisotropicneighborhood} displays such neighborhoods for two different pixels.  We partition ${I}_{\theta, \delta_s, \delta_{\ell}}(t_1,t_2)$ into $n_s \times n_{\ell}$ subregions ${I}^{j_1,j_2}_{\theta, \delta_s, \delta_{\ell}}(t_1,t_2) \triangleq  R^{\theta}_{t_1,t_2} ([t_1- \frac{j_1\delta_{\ell}}{2n_{\ell}}, t_1+ \frac{j_1\delta_{\ell}}{2n_{\ell}} ]\times [t_2- \frac{j_2\delta_{s}}{2n_{s}}, t_2+ \frac{j_2\delta_{s}}{2n_{s}} ] )$. 

\begin{figure}
\begin{center}
%  % Requires \usepackage{graphicx}
  \includegraphics[width=6cm]{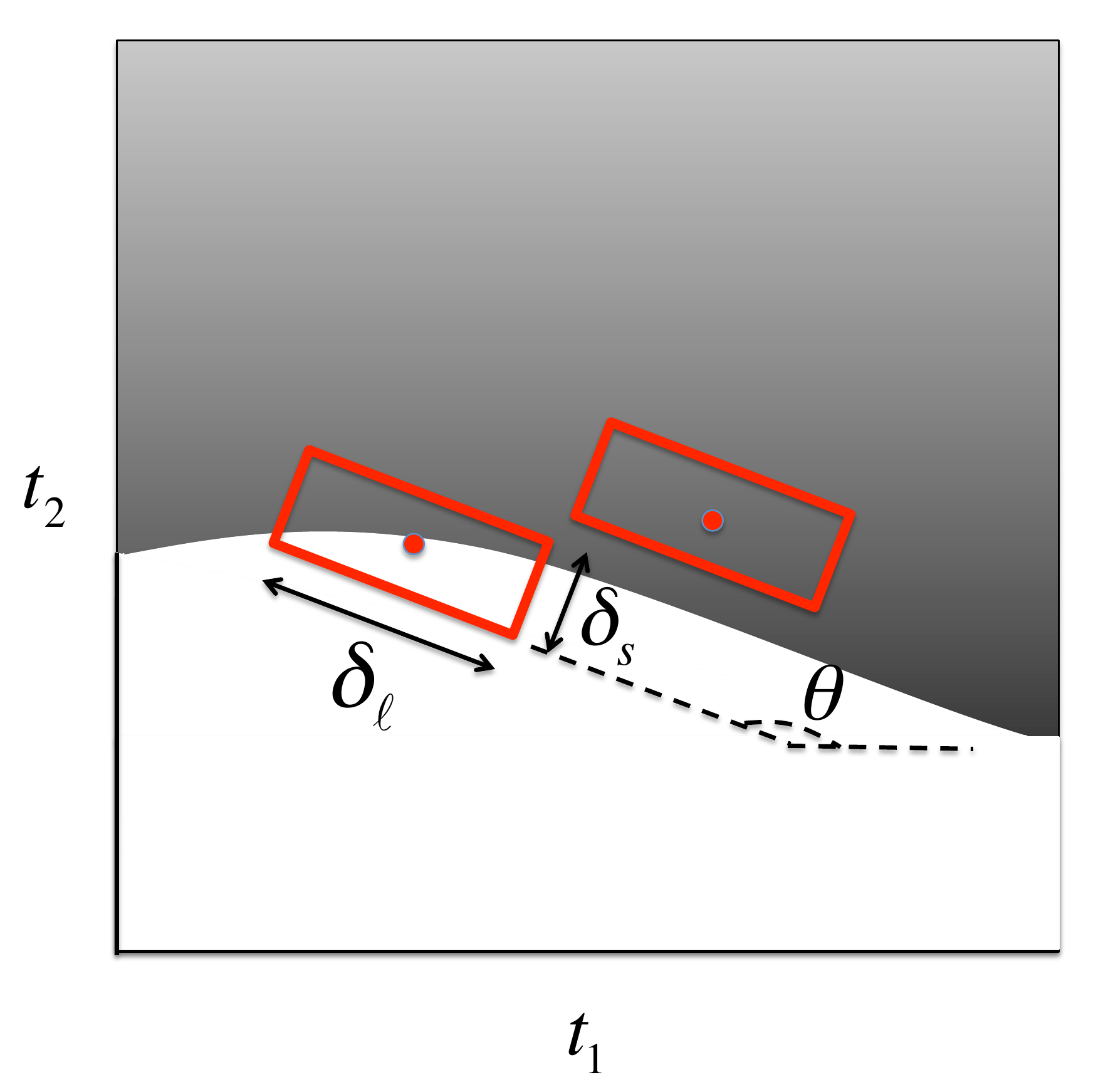}\\
   \caption{The anisotropic neighborhood ${I}_{\theta, \delta_s, \delta_{\ell}}(t_1,t_2)$ in the discrete setting for two different pixels of a Horizon class image.}
    \label{fig:anisotropicneighborhood}
 \end{center}
\end{figure}

The pixelated neighborhood of $(t_1,t_2)$ is defined as the $\mathbf{y} \in \mathds{R}^{n_s \times n_{\ell}}$ satisfying
\[
\mathbf{y}^{\theta, \delta_{\ell}, \delta_{s}}_{t_1,t_2}(j_1,j_2) \triangleq \frac{n_s n_{\ell}}{\delta_s \delta_{\ell}} \int \int_{(s_1,s_2) \in I^{j_1,j_2}_{\theta,\delta_s,\delta_{\ell} }(t_1, t_2) } dY(s_1,s_2).
\]
Define the neighborhood process as
\[
X(t_1,t_2) \triangleq \frac{n_s n_{\ell}}{\delta_s \delta_{\ell}} \int \int_{(s_1,s_2) \in I_{\theta,\frac{\delta_s}{n_s},\frac{\delta_{\ell}}{n_\ell} }(t_1, t_2)} dY(s_1,s_2).
\]
The anisotropic $(\delta_{\ell}, \delta_s, \theta)$-neighborhood distance between two given points in the image is then defined as 
\begin{eqnarray}\label{eq:direcdistance}
\lefteqn{d^{2}_{\theta, \delta_s,\delta_{\ell}} (dY(t_1,t_2), dY(s_1,s_2))} \nonumber \\
 & \triangleq & \! \!\! \frac{1}{n_s n_{\ell}-1} \left( \|\mathbf{y}^{\theta, \delta_{\ell}, \delta_{s}}_{t_1,t_2}-\mathbf{y}^{\theta, \delta_{\ell}, \delta_{s}}_{s_1,s_2} \|_2^2 - |\mathbf{y}^{\theta, \delta_{\ell}, \delta_{s}}_{t_1,t_2}(0,0)-\mathbf{y}^{\theta, \delta_{\ell}, \delta_{s}}_{s_1,s_2}(0,0) |^2 \right).
\end{eqnarray}
The ANLM estimate at point $(t_1,t_2)$ is given by
\begin{eqnarray*}
\hat{f}^{\theta, \delta_s, \delta_{\ell}}(t_1,t_2) \triangleq \frac{ \int_{(s_1, s_2) \in S} w^{\theta, \delta_s, \delta_{\ell}}_{t_1,t_2}(s_1,s_2)X(s_1,s_2) ds_1 ds_2 }{ \int \int _{(s_1, s_2) \in S} w^{\theta, \delta_s, \delta_{\ell}}_{t_1,t_2}(s_1,s_2) ds_1 ds_2 }.
\end{eqnarray*}
where the weights are obtained from 
\begin{eqnarray}\label{eq:weiasscont}
w^{\theta,\delta_s, \delta_{\ell}}_{t_1,t_2}(s_1,s_2) =\! \! \left\{\begin{array}{rl}
 1 &  \mbox{ if $d^2_{\theta,\delta_s,\delta_{\ell}} (dY(t_1,t_2),dY(s_1,s_2)) \leq \frac{2n_sn_{\ell}\sigma^2}{\delta_s \delta_{\ell}}+ \tau_{\sigma} $,}\\
 0 &   \mbox{ otherwise.}
\end{array}\right.
\end{eqnarray}
Here $\tau_\sigma$ is the \textit{threshold parameter} and $\frac{2n_sn_{\ell}\sigma^2}{\delta_s \delta_{\ell}}+ \tau_{\sigma} $ is the \textit{threshold value}. 

ANLM extends NLM in two significant ways. First, in ANLM $\theta, \delta_s, \delta_{\ell}$ are free parameters, while in NLM $\theta=0$
and $\delta_s = \delta_{\ell}$. As Figures \ref{fig:tuning} and \ref{fig:ONLM_diffdir} demonstrate, these free parameters significantly affect the performance of ANLM. Figure \ref{fig:tuning} illustrates the visual denoising results as we vary the angle of the anisotropic neighborhood.  Figure \ref{fig:ONLM_diffdir} does the same for the empirical peak signal-to-noise ratio (PSNR).\footnote{In practice images are measured on a  discrete lattice, i.e., the pixelated values are known. Let the pixelated values of the original noise-free image and the estimator be  $f_{ij}$ and $\hat{f}_{i,j}$, respectively. The empirical mean-squared-error (MSE) of an estimator $\hat{f}$ is defined as ${\rm MSE} = \frac{1}{n^2}  \sum_{i,j} (\hat{f}_{i,j} - f_{i,j})^2$. If $f_{i,j}  \in [0, A]$, then the empirical peak signal-to-noise ratio (PSNR) is defined as $10 \log_{10} \frac{A^2}{\rm MSE}$. }

\begin{figure}
	\captionsetup[subfigure]{labelformat=empty}
	\centering
	\subfloat[(a)~$0^{\circ}$, ${\rm PSNR}=20.6$dB]{\label{fig:hor_ganlm}
	\includegraphics*[width = 0.40\textwidth, height = 0.40\textwidth]{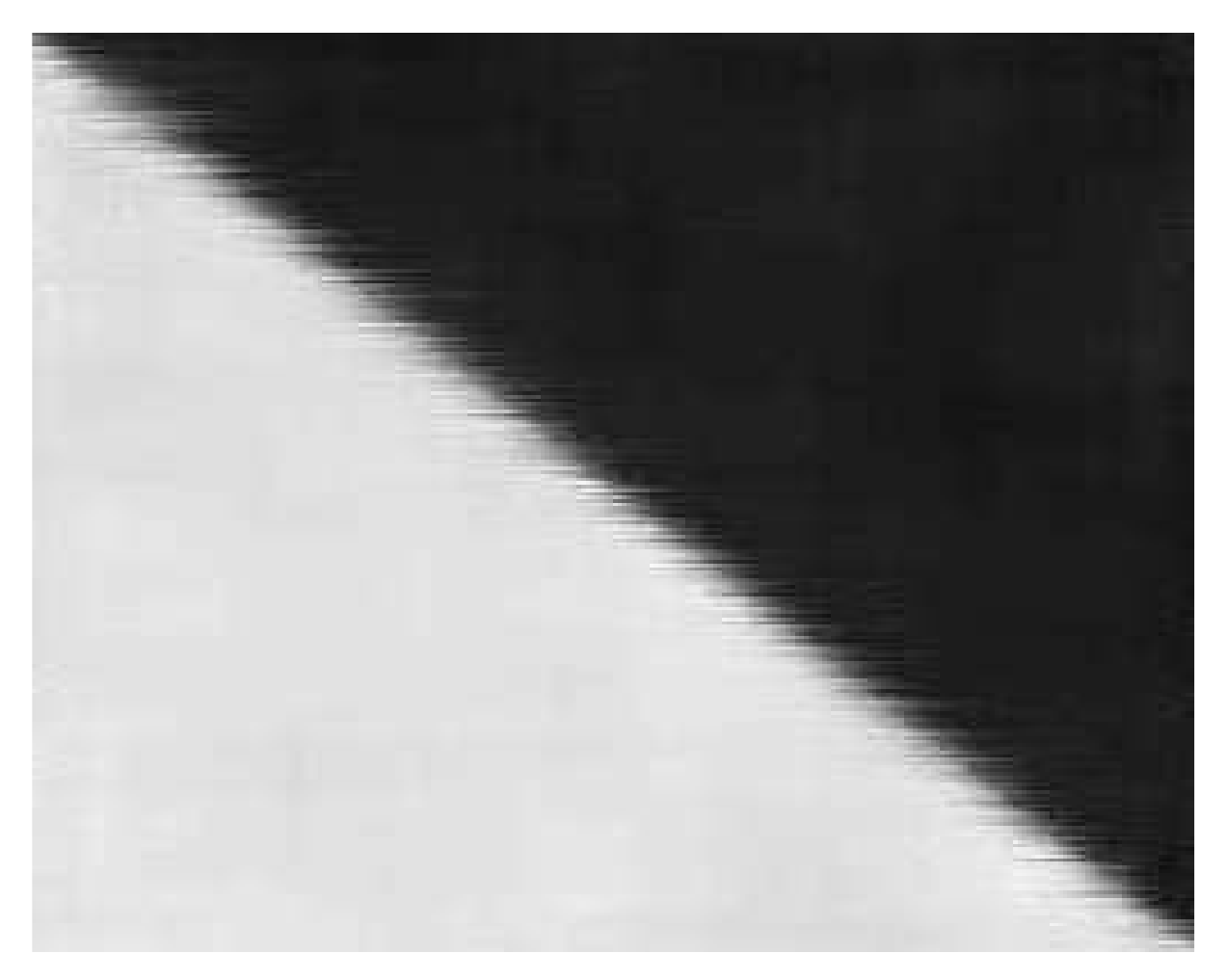}}
	\subfloat[(b)~$45^{\circ}$, ${\rm PSNR}=19.2$dB]{\label{fig:hor_ganlm}
	\includegraphics*[width = 0.40\textwidth, height = 0.40\textwidth]{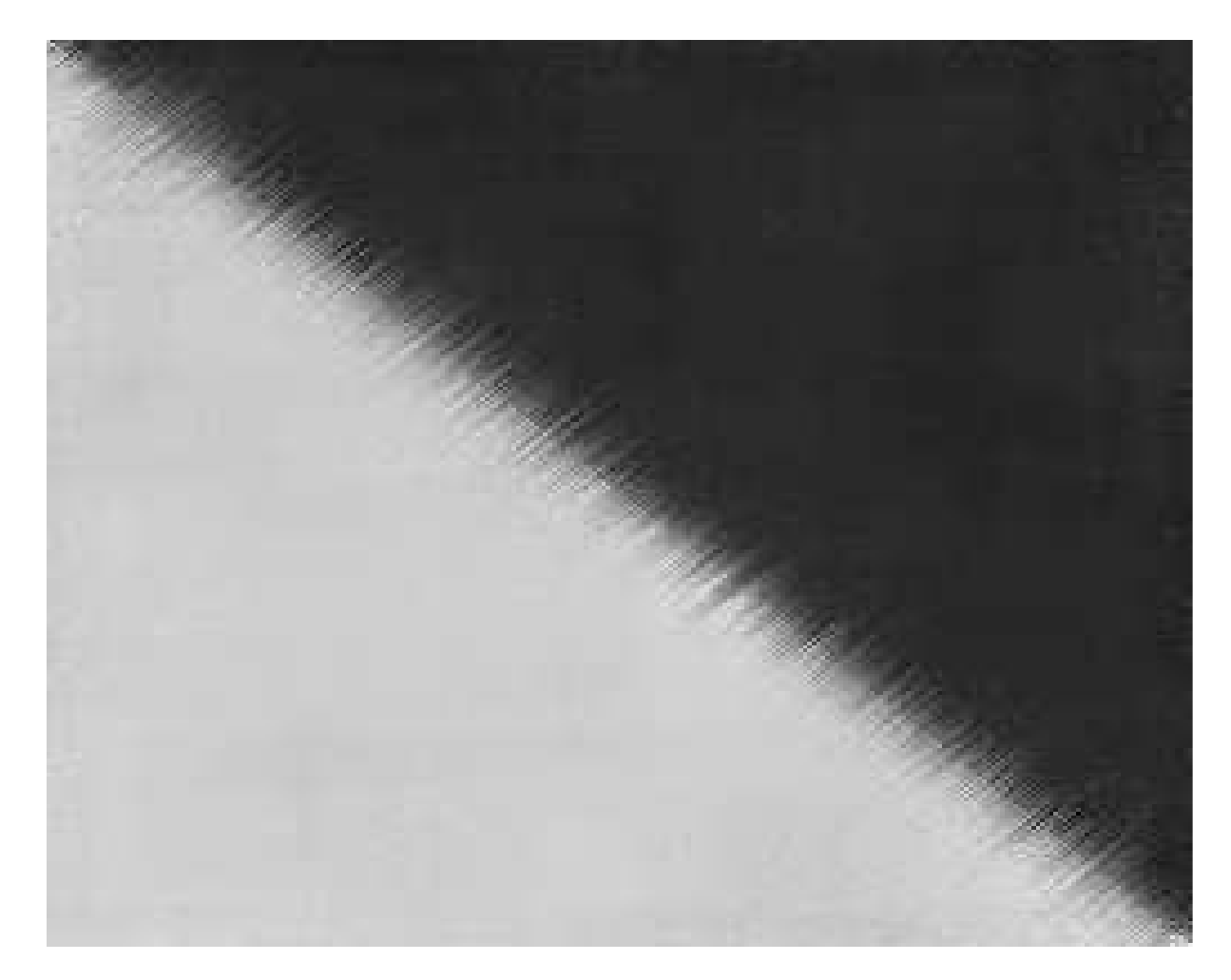}} \\
	\subfloat[(a)~$90^{\circ}$, ${\rm PSNR}=20.6$dB]{\label{fig:hor2} 
	\includegraphics*[width = 0.42\textwidth, height = 0.42\textwidth]{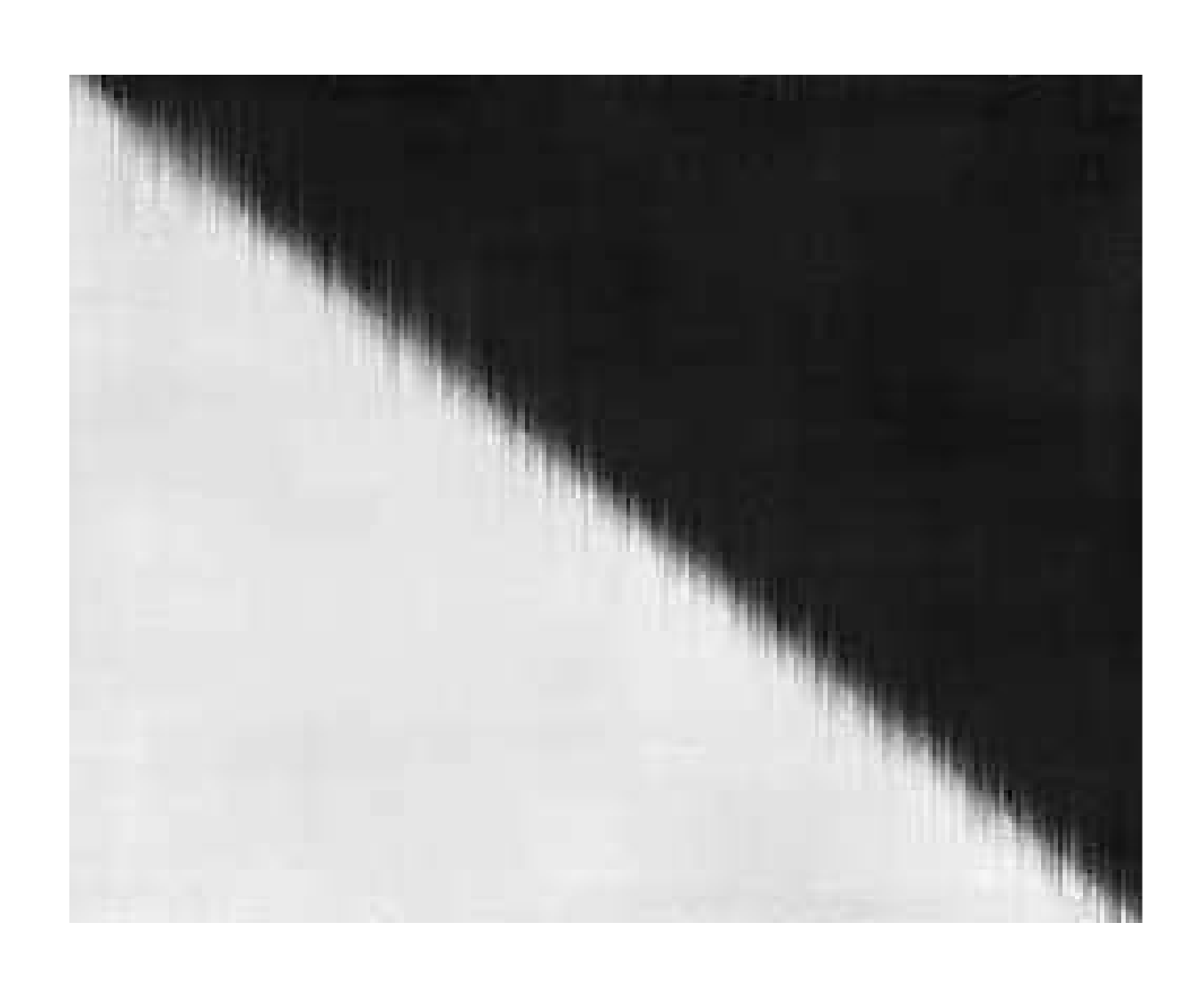}} 
	\subfloat[(d)~$135^{\circ}$, ${\rm PSNR}=28.9$dB]{\label{fig:hor_nlm}
	\includegraphics*[width = 0.40\textwidth, height = 0.40\textwidth]{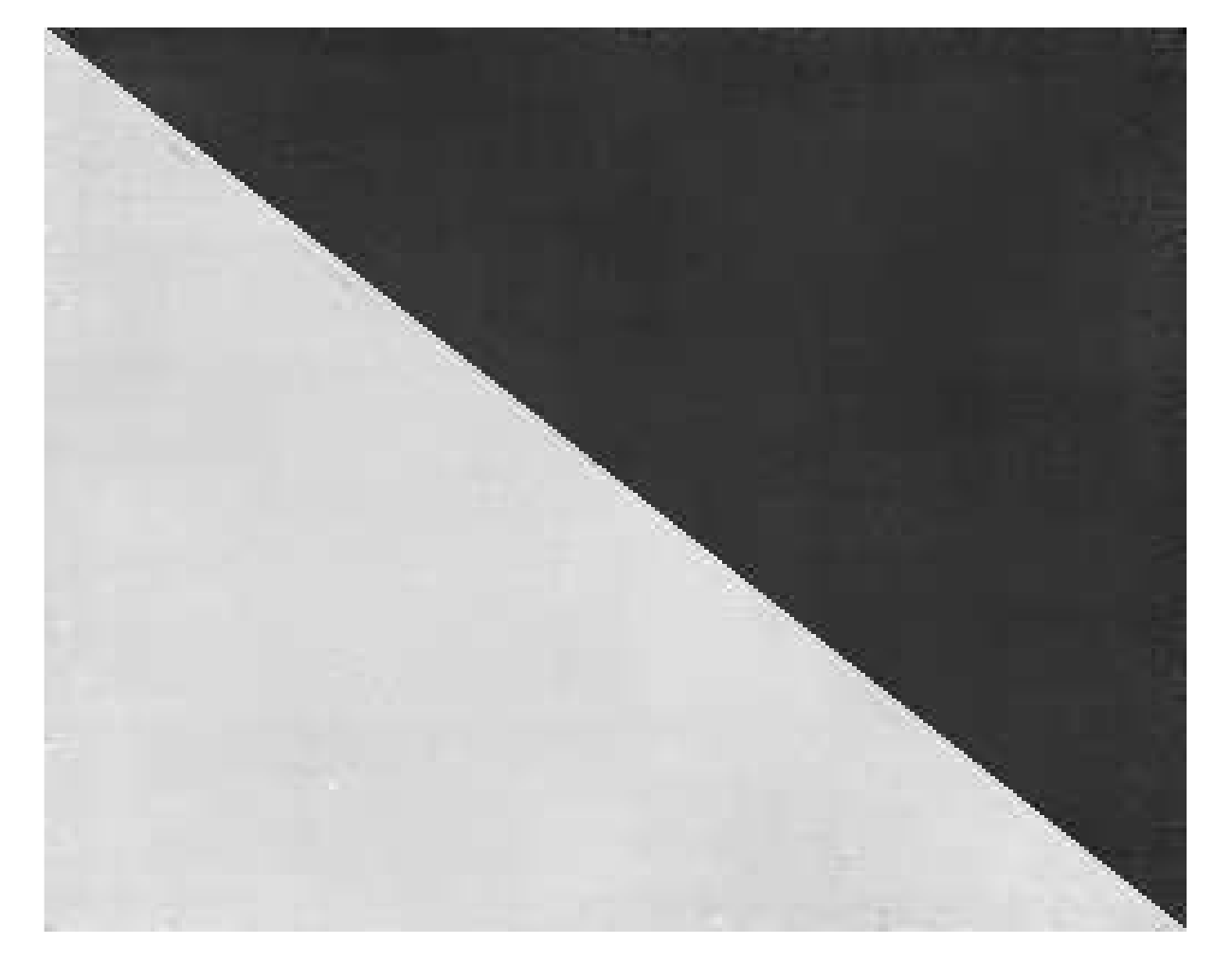}} 
  \caption{Effect of the anisotropic neighborhood orientation angle $\theta$ on the denoising performance of ANLM. The images are $256 \times 256$ pixels with a true edge orientation of $135^{\circ}$. The noise standard deviation is $\xi = 0.5$, which is equivalent to ${\rm PSNR} =6.1$dB (see Section \ref{ssec:discset}).  The neighborhood size and threshold parameter are $\delta_s = 1/256$, $\delta_{\ell}=20/256$ and $\tau=2.7\xi^2$, respectively.  The angles of the neighborhoods ($\theta$) and the resulting PSNR of the estimates are noted below the images.  
%(a) $90^{\circ}$, $22.37$dB; (b) $135^{\circ}$, $22.84$dB; (c) $0^{\circ}$, $22.29$dB; and (d) $45^{\circ}$, $20.80$dB. 
There is a substantial subjective and objective improvement in the performance of ANLM when the neighborhood is aligned with the edge.}
    \label{fig:tuning}
\end{figure}

\begin{figure}[!t] 
\begin{center}
  \includegraphics[width=8cm]{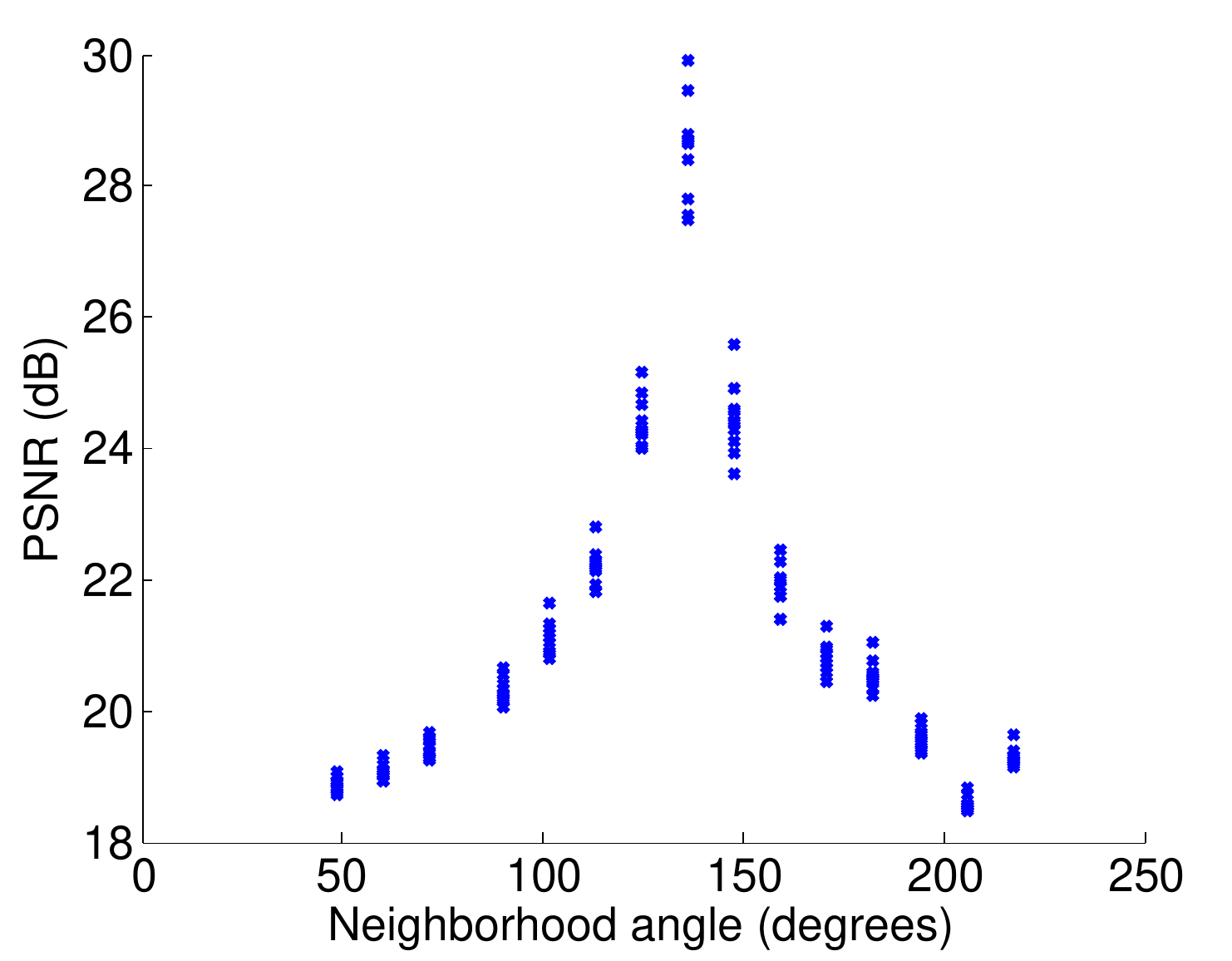}
\caption{Effect of the anisotropic neighborhood orientation angle $\theta$ on the denoising performance of ANLM.  The experimental parameters are the same as in Figure \ref{fig:tuning}.  For each value of $\theta$ we plot the result of 10 Monte Carlo simulations. PSNR of the original noisy image is $6.1$dB}
  \label{fig:ONLM_diffdir}
  \end{center}
\end{figure}

\subsection{Oracle ANLM} \label{ssec:oracle}

It is clear from Figures \ref{fig:tuning} and \ref{fig:ONLM_diffdir} that we should align the ANLM neighborhood locally with the edge to maximize denoising performance.  In the Oracle ANLM (OANLM) algorithm, we assume that we have access to an oracle that knows the image's edge orientation at each pixel.
Consider $f_h(t_1, t_2) \in H^{\alpha}(C)$.  Let $h^{\prime}(t_1)$ denote the derivative of the edge contour $h(t_1)$, and let ${\Gamma}_{\gamma} \triangleq \{(t_1,t_2) \ : \ t_1 = \gamma  \}$ segment the region $S$.  Define $\theta_{\gamma}$ as the angle between the tangent to the edge contour $h$ and the horizontal line at $t_1 = \gamma$.  OANLM is defined as ANLM with the following parameter settings:
\begin{itemize}

\item {\em Quadratic scaling:} Instead of using isotropic (square) neighborhoods, we use anisotropic (rectangular) neighborhoods of size $\delta_s \times \delta_{\ell}$. The scaling of $\delta_s$ and $\delta_{\ell}$ depends on the smoothness of the edge. Since we are primarily interested in ${H}^2(C)$, we will emphasize the quadratic scaling\footnote{If the smoothness of an edge is $\alpha>1$, then the optimal neighborhood size is given by $\delta_s = \Theta(\sigma^{2\alpha/(\alpha+1)} |\log \sigma|^{2\alpha/(\alpha+1))}$ and $\delta_{\ell} = \Theta(\sigma^{2/(\alpha+1)} |\log \sigma|^{2/(\alpha+1)})$. Note that for the special case $\alpha =2$ they satisfy the quadratic scaling. }
 $\delta_s = 4\sigma^{4/3} |\log{\sigma}|^{4/3}$ and $\delta_{\ell} = 2\sigma^{2/3} |\log{\sigma}|^{2/3}$; that is $\delta_s = \delta_{\ell}^2$.
 
\item {\em Aligned neighborhoods:} %In order to obtain the estimate $\hat{f}_{i,j}$, 
We align the neighborhood at point $(t_1,t_2)$ to the orientation $\theta_{t_1}$ of the edge contour at $(t_1, h(t_1))$.  Thus all pixels in a column have the same orientation. Figure \ref{fig:ONLM} illustrates such a neighborhood selection.\footnote{In smooth regions, where the pixel neighborhoods do not intersect with the edge, neither the shape nor the orientation of the neighborhood affect the denoising performance.}

\item {\em Logarithmic pixelation:} We assume that $n_s \times n_{\ell} = \lceil \log^2 \sigma \rceil$. In other words, as $\sigma$ increases the resolution of the pixelated image increases as well. 
\end{itemize}
\begin{figure}
\begin{center}
  \includegraphics[width=6.1cm]{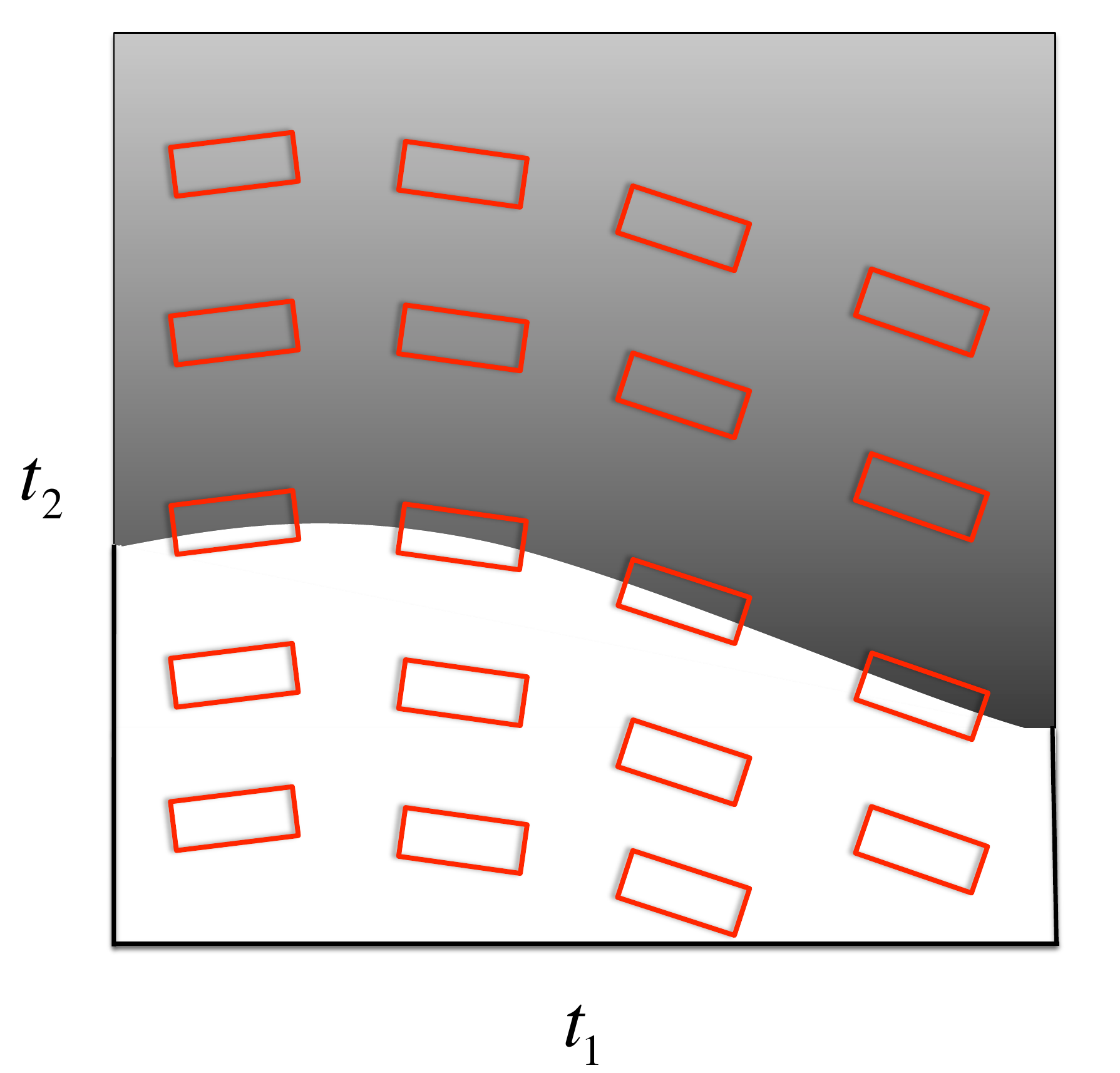}\\
  \caption{The OANLM neighborhoods satisfy the quadratic scaling $\delta_s = \delta_{\ell}^2$ and are aligned with the edge. In regions far from the edge, isotropic neighborhoods perform just as well.}
  \label{fig:ONLM}
  \end{center}
\end{figure}

%\begin{figure} \label{fig:anisotropicneighborhood}
%\begin{center}
%  % Requires \usepackage{graphicx}
%  \includegraphics[width=10cm]{AnisotropicNeighborhood.pdf}\\
%  \caption{An example of a linear edge with an anisotropic neighborhood aligned with it.}
%  \end{center}
%\end{figure}

\begin{thm}\label{thm:oracleANLM}
If $\hat{f}^{O}$ is the OANLM estimator with the threshold $\tau_{\sigma} = \frac{2}{\sqrt{|\log \sigma|}}$, then
\begin{eqnarray*}
R({H}^{\alpha}(C), \hat{f}^{O}) = O(\sigma^{4/3} |\log \sigma|^{4/3} ).
\end{eqnarray*}
\end{thm}

%\begin{figure} 
%\centering{
%  % Requires \usepackage{graphicx}
%  \includegraphics[width=7.2cm]{ONLMFig.pdf}
%  \caption{The OANLM neighborhoods satisfy the quadratic scaling $\delta_s = \delta_{\ell}^2$ and are aligned with the edge. In regions far from the edge, isotropic neighborhoods perform just as well.}
%  \label{fig:ONLM}
%  }
%\end{figure}
\noindent The proof of Theorem \ref{thm:oracleANLM} can be simply modified to provide a bound for the risk of NLM as well. Let $\delta = \sqrt{\delta_s \delta_{\ell}}$ and $n= \sqrt{n_s n_{\ell}} = |\log (\sigma)|$, where $\delta_s$, $\delta_{\ell}$, $n_s$, $n_{\ell}$ are as given in OANLM algorithm. 
\medskip

\begin{cor}\label{cor:nlm}
If $\hat{f}^N$ is the NLM estimator with threshold $\tau_{\sigma} = 2/\sqrt{|\log(\sigma)|}$, then
\begin{eqnarray*}
R({H}^{\alpha}(C), \hat{f}^{N}) = O(\sigma |\log \sigma|).
\end{eqnarray*}
\end{cor}

See Section \ref{sec:proofupnlm} for the proof of this corollary. Also, under certain mild assumptions, the risk of NLM can be lower bounded by $\Omega(\sigma)$. The proof is similar to the proof of Theorem 5 in \cite{MaNaBa11}. However, for the sake of completeness and since our framework differs from \cite{MaNaBa11}, we overview the main steps of the proof in \ref{app:nlm}. 

Theorem \ref{thm:oracleANLM} is based on the strong oracle assumption that the edge direction is known exactly.  Needless to say, such information is rarely available in practical applications. Consider a weaker notion of OANLM that has access to an estimate $\hat{\theta}_{\gamma}$ of $\theta_{\gamma}$ that satisfies
\begin{equation}\label{eqn:thetamismatch}
|\hat{\theta}_{\gamma} - \theta_{\gamma}| \leq \Theta(\sigma^{\beta}). 
\end{equation} 
 OANLM with exact edge orientation information corresponds to $\beta = \infty$ in this model. 
When $\beta < \infty$, the choices $\delta_\ell = 2\sigma^{2/3} |\log\sigma|^{2/3}$ and $\delta_{s} = 4\sigma^{4/3} |\log\sigma|^{4/3}$ are not necessarily optimal. 
%Consider the extreme case where the edge direction estimator is just a random guess and the error could be as large as $\pi/2$. According to the heuristic arguments provided in Section \ref{ssec:whyaniso}, in this case isotropic neighborhoods outperform anisotropic neighborhoods. 
The following theorem characterizes the performance of OANLM using $\hat{\theta}$.
\bigskip

\begin{thm}\label{thm:thetamismatch}
\sloppy Let the OANLM estimator use the inaccurate edge orientation $\hat{\theta}_{\gamma}$ satisfying \eqref{eqn:thetamismatch}. Set the neighborhood sizes to $\delta_{\ell} = \min(\sigma^{2/3}| \log^{2/3}\sigma|, \sigma^{1-\beta/2} |\log \sigma|)$ and $\delta_s =\sigma^2\log^2\sigma / \delta_{\ell} $. Then the risk of the estimator with  $\tau_{\sigma} = 2/\sqrt{|\log(\sigma)|}$ satisfies 
\begin{eqnarray*}
R({H}^{\alpha}(C), \hat{f}^{O}) = O(\delta_s {\rm Poly}(|\log \sigma|)),
\end{eqnarray*}
where ${\rm Poly}(| \log \sigma|)$ is a polynomial of degree at most $2$ in terms of $|\log \sigma|$.
\end{thm}

If the edge estimate is exact, then the result simplifies to the result of Theorem \ref{thm:oracleANLM}. However, this theorem confirms that OANLM algorithm achieves the optimal rate even if there is an error in $\hat{\theta}$ of order $O(\sigma^{\beta})$ with $\beta \geq 2/3$. 
%Additionally,Theorem \ref{thm:thetamismatch} provides us a rate for $\beta<2/3$ as well.
\bigskip

\begin{cor}\label{cor:oracANLM}
Let the OANLM estimator use the inaccurate edge orientation $\hat{\theta}_{\gamma}$ satisfying \eqref{eqn:thetamismatch} with $\beta \geq 2/3$. Set the neighborhood sizes to $\delta_{\ell} =  2\sigma^{2/3} |\log^{2/3} \sigma|$ and $\delta_s =4\sigma^{4/3} |\log \sigma| $. Then, the risk of the estimator with  $\tau_{\sigma} = 2/\sqrt{|\log(\sigma)|}$ satisfies 
\begin{eqnarray*}
R({H}^{\alpha}(C), \hat{f}^{O}) = O( \sigma^{4/3} {\rm Poly}( |\log \sigma |)),
\end{eqnarray*}
where ${\rm Poly}( |\log \sigma |)$ is a polynomial of degree at most 4 in terms of $|\log( \sigma)|$. 
\end{cor}
\begin{proof}
If we plug in $\beta \geq 2/3$ in Theorem \ref{thm:thetamismatch} we obtain
\[
\delta_{\ell} = \min(2\sigma^{2/3}| \log^{2/3}\sigma|,2 \sigma^{1-\beta/2} |\log \sigma|) =2 \sigma^{2/3} |\log \sigma|^{2/3}.
\]
Therefore, $\delta_s = 4 \sigma^{4/3} |\log \sigma|^{4/3}$. This establishes the result. 
\end{proof}

We can also simplify the result of Theorem \ref{thm:thetamismatch} for the case of $\beta< 2/3$.

\begin{cor}\label{cor:oracANLM}
Let the OANLM estimator use the inaccurate edge orientation $\hat{\theta}_{\gamma}$ satisfying \eqref{eqn:thetamismatch} with $\beta \leq 2/3$. Set the neighborhood sizes to $\delta_{\ell} = 2 \sigma^{1-\beta/2} |\log \sigma|$ and $\delta_s =4\sigma^{1+ \beta/2} |\log \sigma| $. Then, the risk of the estimator with  $\tau_{\sigma} = 2/\sqrt{|\log(\sigma)|}$ satisfies 
\begin{eqnarray*}
R({H}^{\alpha}(C), \hat{f}^{O}) = O( \sigma^{1+\beta/2} {\rm Poly}( |\log \sigma |)).
\end{eqnarray*}
where ${\rm Poly}( |\log \sigma |)$ is a polynomial of degree at most 3 in terms of $|\log( \sigma)|$. 
\end{cor}
\begin{proof}
Note that this corollary is exactly the same as Theorem \ref{thm:thetamismatch}. In fact, since $\beta\leq 2/3$, 
\[
\delta_{\ell} = \min(\sigma^{2/3}| \log^{2/3}\sigma|, \sigma^{1-\beta/2} |\log \sigma|) = \sigma^{1-\beta/2} |\log \sigma|.
\]
Therefore, $\delta_s = \sigma^{1+ \beta/2} |\log \sigma|$ and therefore, according to Theorem \ref{thm:thetamismatch} 
\begin{eqnarray*}
R({H}^{\alpha}(C), \hat{f}^{O})  = O(\sigma_s {\rm Poly}( |\log \sigma |))=O( \sigma^{1+\beta/2} {\rm Poly}( |\log \sigma |)).
\end{eqnarray*}
\end{proof}

This corollary suggests that, as long as we have an edge orientation estimate that improves as $\sigma \rightarrow 0$ (or the number of pixels increases), OANLM outperforms NLM.  Also note that, as $\beta$ decreases, the neighborhoods become more isotropic.

\subsection{Discrete angle ANLM}

In this section, we introduce the Discrete-angle ANLM (DANLM) algorithm that achieves the minimax rate without oracle information. The key idea is to calculate the neighborhood distance over several directional neighborhoods and fuse them to obtain a similarity measure that works well for all directions. As for OANLM, set $\delta_s =4 \sigma^{4/3} |\log(\sigma)|^{4/3}$ and $\delta_{\ell} = 2\sigma^{2/3} |\log(\sigma)|^{2/3}$, and let $q \triangleq \pi \sigma^{-2/3}$. Define the angles $\theta_0 \triangleq 0$, $\theta_1 \triangleq \sigma^{2/3}$, $\theta_2 \triangleq 2\sigma^{2/3}, \ldots, \theta_q \triangleq \pi - \sigma^{2/3}$ and 
\[
\mathcal{O} \triangleq \{\theta_0, \ldots, \theta_q \}.
\]
 For a point $(u,v)$ we consider all of the anisotropic, directional neighborhoods for $\theta \in \mathcal{O}$. Figure 7 displays four of these neighborhoods.

\begin{figure} \label{fig:ANLMneighborhood}
\begin{center}
  \includegraphics[width=6.05cm]{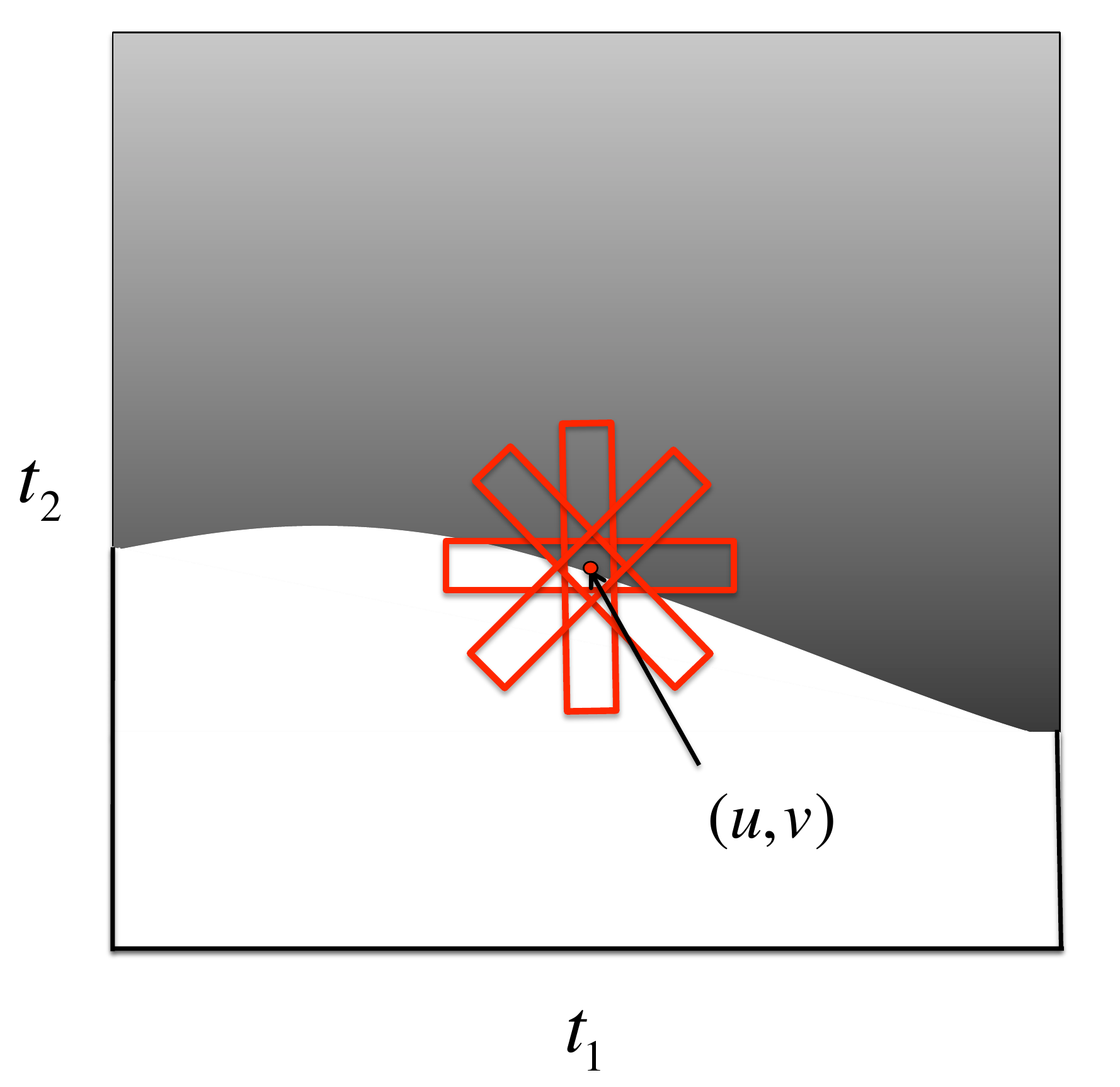}\\
  \caption{DANLM neighborhoods around one pixel location for $q = 4$.}
  \end{center}
\end{figure} 

\sloppy
Define the discrete angle anisotropic distance between $dY(t_1,t_2)$ and $dY(s_1,s_2)$ as
\begin{equation*}
d_{\mathcal{A}}^2 (dY(t_1,t_2), dY(s_1,s_2)) \triangleq \min_{\theta \in \mathcal{O}} d_{\theta, \delta_s, \delta_{\ell}}^2 (dY(t_1,t_2), dY(s_1,s_2)),
\end{equation*}
where $ d_{\theta, \delta_s, \delta_{\ell}}^2 (dY(t_1,t_2), dY(s_1,s_2))$ is defined in \eqref{eq:direcdistance}. The DANLM estimate is then given by 
\begin{eqnarray*}
\hat{f}^{D}(t_1,t_2) = \frac{\int_{(s_1,s_2) \in S} w^{AN}_{t_1,t_2}(s_1,s_2)X(s_1,s_2) ds_1 ds_2 }{ \int_{(s_1,s_2) \in S} w^{AN}_{t_1,t_2}(s_1,s_2) ds_1 ds_2 }.
\end{eqnarray*}
where 
\begin{eqnarray}
w^{AN}_{t_1,t_2}(s_1,s_2) =\! \! \left\{\begin{array}{rl}
 1 &  \mbox{ if $d^2_{\mathcal{A}} (dY(t_1,t_2),dY(s_1,s_2)) \leq \frac{2n_sn_{\ell}\sigma^2}{\delta_s \delta_{\ell}}+ \tau_{\sigma} $,}\\
 0 &   \mbox{ otherwise.}
\end{array}\right.
\end{eqnarray}

In summary, we note the following features of the DANLM algorithm.  First, it uses the quadratic scaling $\delta_s = \delta_{\ell}^2$.  Second, the optimal neighborhood direction can change from pixel to pixel. The following theorem, proved in Section \ref{sec:proofs}, shows that the risk of this algorithm is within the logarithmic factor of the optimal minimax risk.\bigskip

\begin{thm} \label{thm:ANLM}
The risk of DANLM satisfies
\[
R(H^{\alpha} (C_{\alpha}), \hat{f}^{D}) \leq O( \sigma^{4/3} |\log(\sigma)|^{4/3}).
\]
\end{thm}

Figure \ref{fig:ONLMvsNLM}  reprises the simulation experiment of Figure \ref{fig:tuning} using the four algorithms described thus far:  NLM, OANLM with perfect knowledge of the edge orientation, OANLM with imperfect knowledge of the edge orientation, and DANLM with four angles.  All three of the latter algorithms outperform the isotropic NLM.

\begin{figure}[!t]
	\captionsetup[subfigure]{justification=centering,labelformat=parens}
	\centering
%	NLM 
	\subfloat[NLM, ${\rm PSNR}= 26$ dB]{\label{fig:hor_ganlm}
	\includegraphics*[width = 0.41\textwidth, height = 0.41\textwidth]{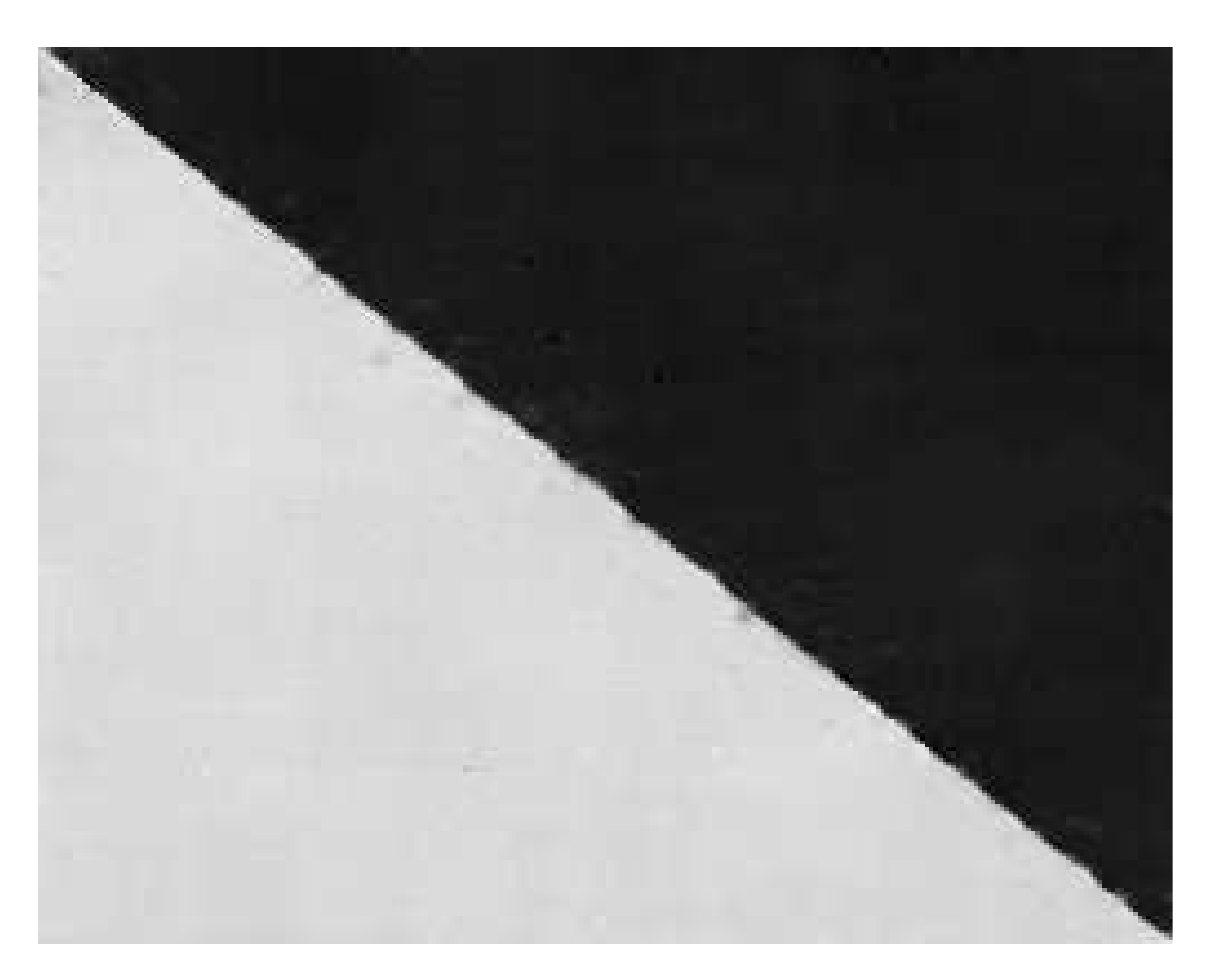}}
%
%	OANLM perfect 
	\subfloat[OANLM with no error, \protect\\ ${\rm PSNR}= 28.9$ dB]{\label{fig:hor2} 
	\includegraphics*[width = 0.40\textwidth, height = 0.40\textwidth]{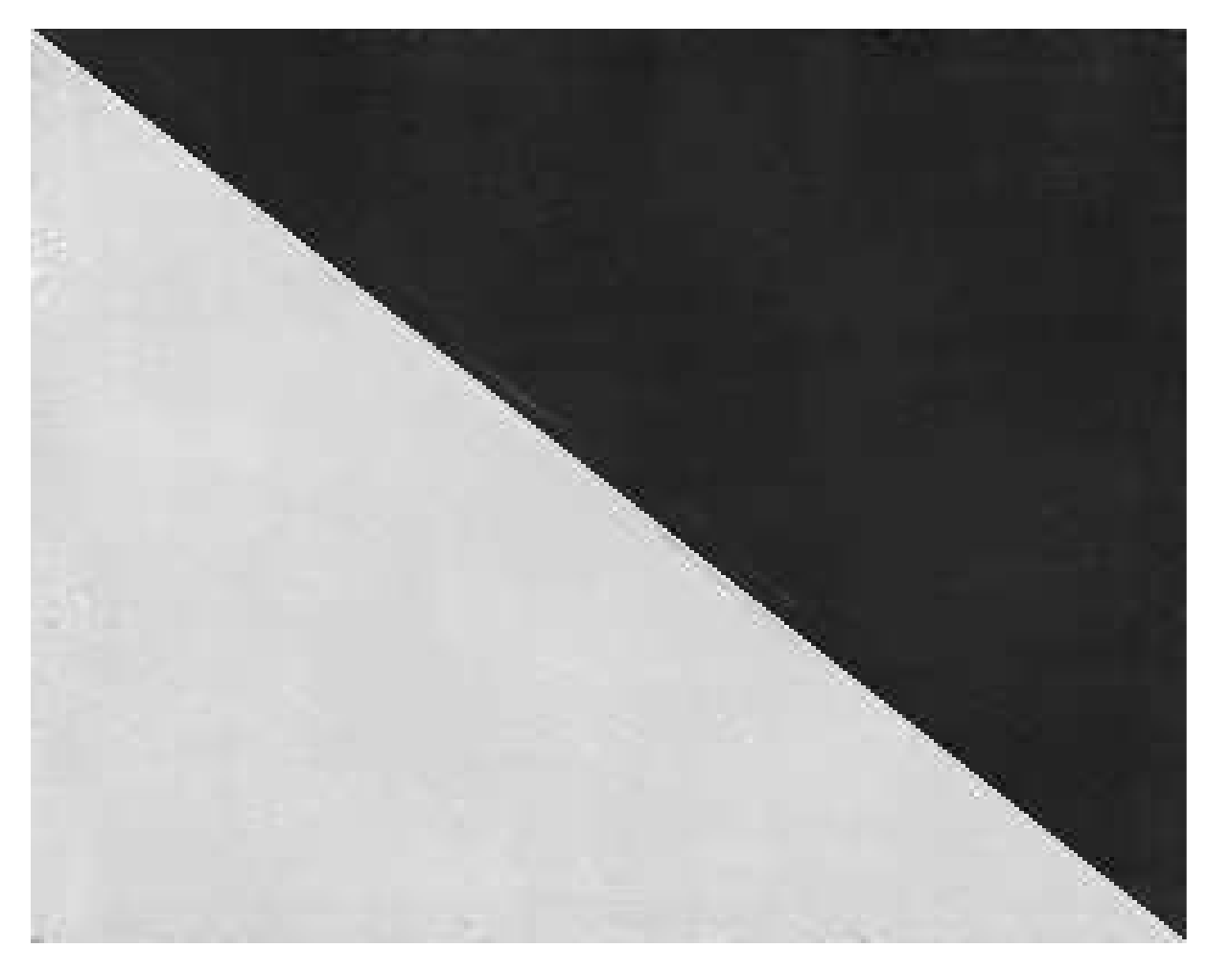}}
	
%	OANLM 10 percent 
	\subfloat[ONALM with error, \protect\\ ${\rm PSNR}= 28.6$ dB]{\label{fig:hor_nlm}
	\includegraphics*[width = 0.40\textwidth, height = 0.40\textwidth]{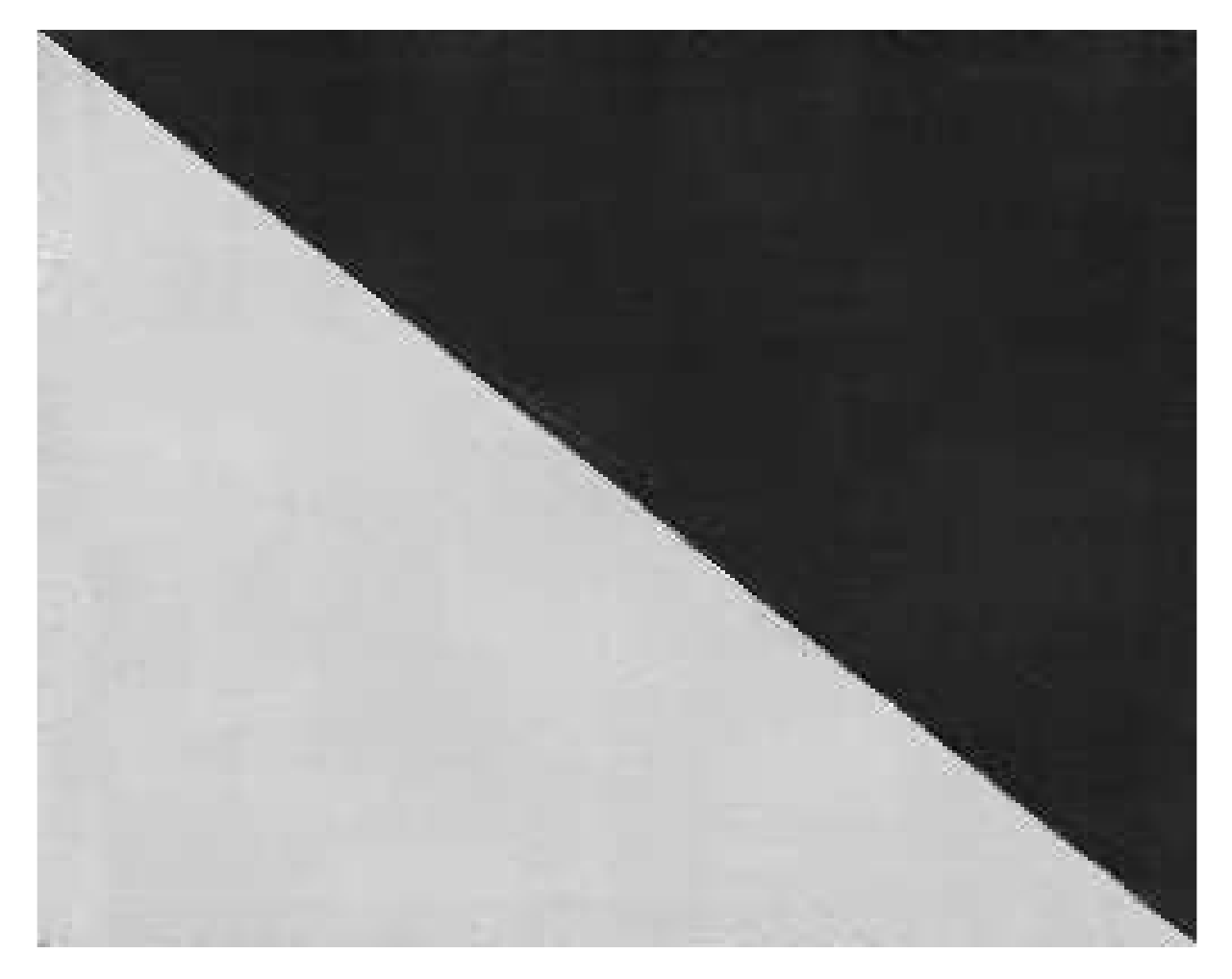}}
%
%	DANLM 
	\subfloat[DANLM, ${\rm PSNR} = 27.5$ dB]{\label{fig:hor_ganlm}
	\includegraphics*[width = 0.40\textwidth, height = 0.40\textwidth]{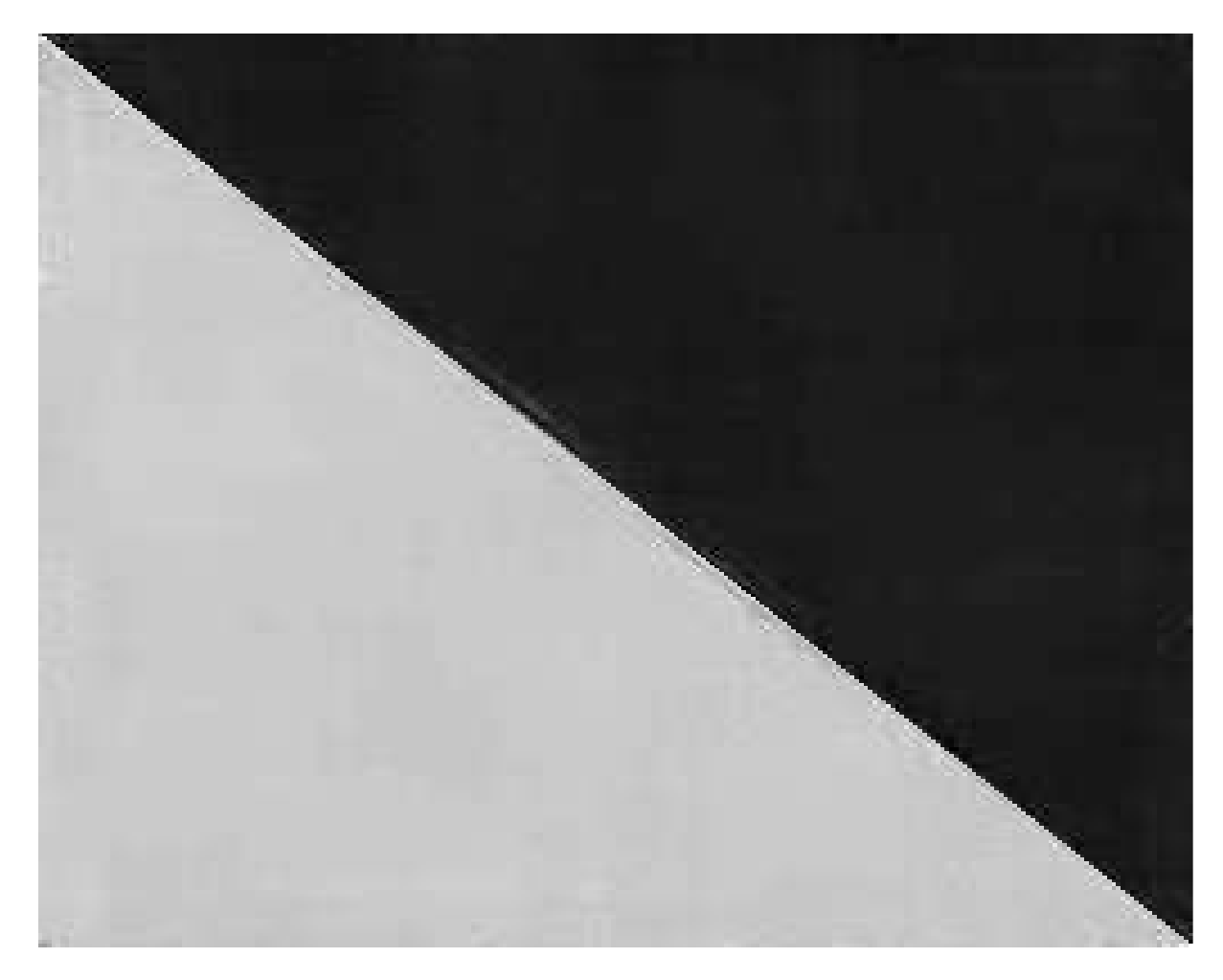}}
\caption{Continuation of the experiment of Figure \ref{fig:tuning} that compares 
(a)~isotropic NLM, 
%${\rm MSE} = 26.02$dB, 29.37dB
(b)~OANLM with perfect knowledge of the edge direction, %${\rm MSE}= 28.86$dB, 33.05dB
(c)~OANLM with 10\% error in the knowledge of the edge direction, %${\rm MSE}= 28.58$dB, 30.45dB
(d)~DANLM with $q=4$ angles. %${\rm MSE} = 27.45$dB. 32.75dB
The images are of size $256\times 256$ pixels. In the ANLM algorithms, $\delta_s$ and $\delta_{\ell}$ are the same as in Figure \ref{fig:tuning}. The neighborhood size of NLM is $\sqrt{\delta_s \times \delta_{\ell}}$. The edge orientation is $135^{\circ}$. The standard deviation of the noise is $\xi = 0.5$ which implies $\rm PSNR = 6.1 dB$.}
\label{fig:ONLMvsNLM}
\end{figure}

\section{Practical ANLM algorithms and experiments}\label{sec:simulation}

In this section we introduce a practical gradient-based ANLM algorithm and then complement the above theoretical arguments with additional simulations and experiments with synthetic and real-world imagery. Note that the algorithms we introduce in this section do not immediately outperform state-of-the-art denoising algorithms such as BM3D \cite{Dabov:2007p4629} but are rather intended to highlight that anisotropic neighborhoods are more suitable for edges than isotropic ones. Since NLM is a key building block in several top-performing algorithms \cite{Dabov:2007p4629, Kervrann:2006p33}, we expect that anisotropy will pay off for such algorithms as it does for NLM. But addressing this important question is left for future research.

\subsection{Extension to discrete images}\label{ssec:discset}

In practice the observations are noisy pixelated values of an image and the objective is only to estimate the pixelated values. In this section we explain how the ideas of directional neighborhood and ANLM can be extended to the discrete settings. Suppose we are interested in estimating a $n \times n$ image $f(\frac{i}{n}, \frac{j}{n})$ with noisy observations $o_{i,j} = f(\frac{i}{n}, \frac{j}{n})+\zeta_{i,j}$, where $\zeta_{i,j} \overset{iid}{\sim} N(0, \xi^2)$ and $\xi$ is the standard deviation of the noise. The extension of the anisotropic neighborhood to the discrete setting is straightforward. Let $\bar{S} \triangleq \left\{\frac{1}{n}, \frac{2}{n}, \ldots, \frac{n-1}{n},1 \right\} \times \left\{\frac{1}{n}, \frac{2}{n}, \ldots, \frac{n-1}{n},1 \right\}$ and
$\tilde{S} \triangleq \{1,2,\ldots, n \} \times \{1,2, \ldots, n\}$. For a given set $B \subset S= [0,1]^2$ we define $\bar{B} \triangleq B \cap \bar{S}$.
 
\noindent The discrete $(\theta, \delta_s, \delta_{\ell})$-distance between two pixels $o_{i,j}$ and $o_{m,\ell}$ is defined as 
\begin{eqnarray}\label{eqn:ddist}
\bar{d}_{\theta, \delta_s, \delta_{\ell}}^2(o_{i,j}, o_{m,\ell}) \triangleq \frac{1}{|\mathcal{P}|}  \sum \limits_{(r,q) \in \mathcal{P} } (o_{i+r,j+q}- o_{m+r,\ell+q})^2,
\end{eqnarray}
where $\mathcal{P} \triangleq \{(r,q) \in \mathds{Z}^2 \ | \ (\frac{i+r}{n}, \frac{j+q}{n}) \in \bar{I}_{\theta, \delta_s, \delta_{\ell}}(\frac{i}{n},\frac{j}{n}) \}$. See Figure \ref{fig:ANLMneighborhooddiscrete2}. Note that for the pixels that are close to the image boundaries the neighborhoods are not rectangular any more and they include the part of the rectangle that is inside the image boundaries. Such non-rectangular neighborhoods will be compared with non-rectangular neighborhoods of the other pixels for the calculation of the distance. The ANLM estimate at pixel $f(\frac{i}{n},\frac{j}{n})$ is given by
\begin{figure}[!t]
\centering{
  \includegraphics[width=6.1cm]{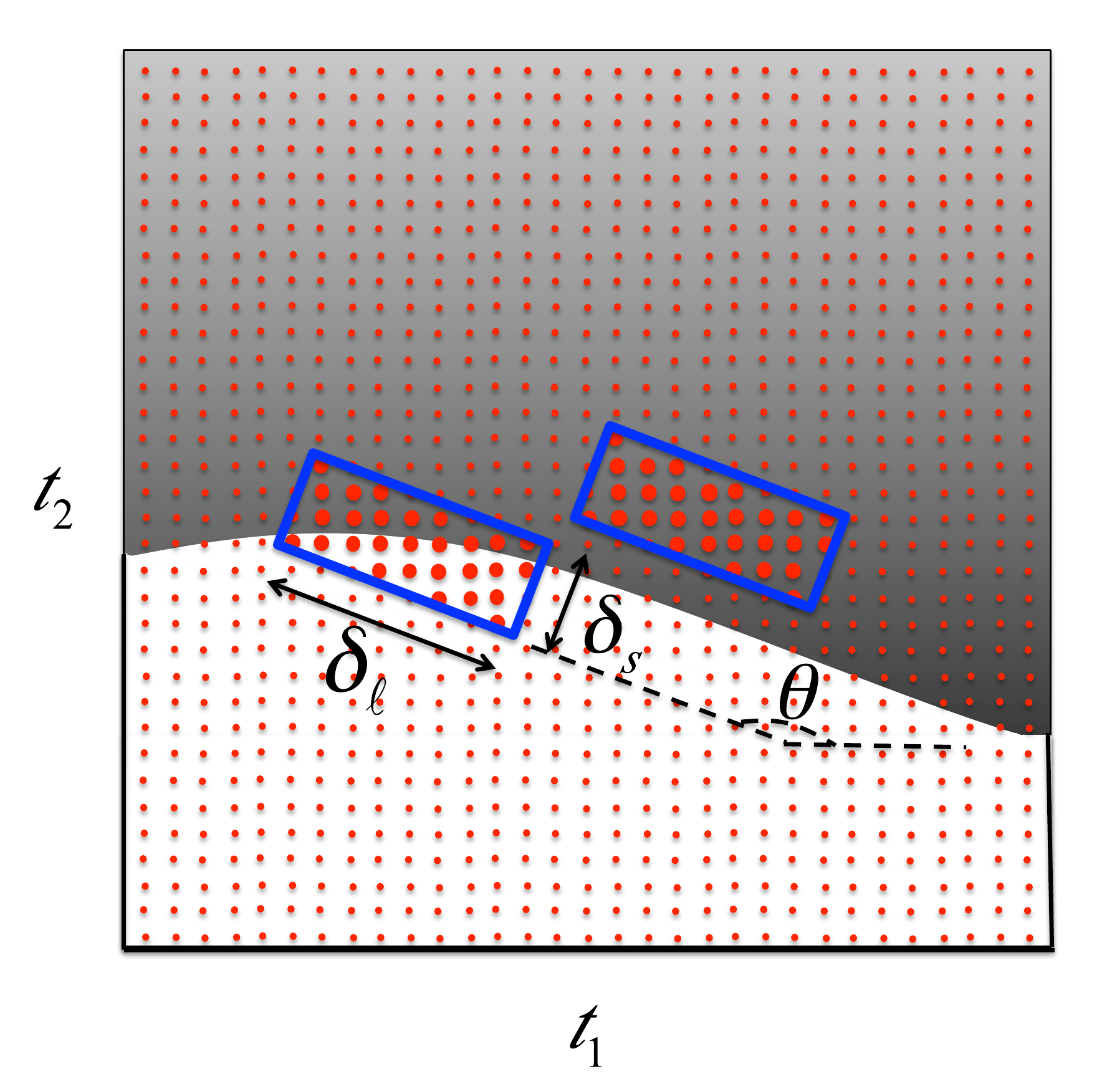}
  \caption{The anisotropic neighborhood $\bar{I}_{\theta, \delta_s, \delta_{\ell}}(\frac{i}{n}, \frac{j}{n})$ in the discrete setting for two different pixels of a Horizon class image. The bold red dots represent the anisotropic neighborhoods. }
   \label{fig:ANLMneighborhooddiscrete2}
 }
\end{figure}
\begin{eqnarray*}
\hat{f}_{i,j}^{\theta, \delta_s, \delta_{\ell}} \triangleq \frac{\sum_m \sum_{\ell}{\bar{w}}^{\theta, \delta_{\ell}, \delta_s}_{i,j}(m,\ell) o_{m, \ell}}{\sum_m \sum_{\ell}  {\bar{w}}^{\theta, \delta_{\ell}, \delta_s}_{i,j}(m,\ell)}, 
\end{eqnarray*}
where the weights are obtained from the distances $\bar{d}_{\theta, \delta_s, \delta_{\ell}}^2(o_{i,j}, o_{m,\ell})$. For instance, the simple policy introduced in \eqref{eq:weiasscont} corresponds to  
\begin{eqnarray}\label{eqn:weightvalue}
{\bar{w}}^{\theta, \delta_{\ell},\delta_s}_{i,j}(m,\ell) \triangleq \! \! \left\{\begin{array}{rl}
 1 &  \mbox{ if ${\bar{d}}_{\theta, \delta_s, \delta_{\ell}}^2(o_{i,j}, o_{m,\ell})  \leq 2 \xi^2+ \tau_{n}  $,}\\
 0 &   \mbox{ otherwise.}
\end{array}\right.
\end{eqnarray}
Using the discrete anisotropic neighborhood and distance,  the extensions of OANLM and DANLM to the discrete setting is straightforward.

\subsection{Gradient-based ANLM}\label{ssec:simugrad}

While DANLM is somewhat practical and also theoretically optimal for the Horizon class of images, its computational complexity is higher than NLM and grows linearly in the number of directions $q$, where $q \sim o(n^{2/3})$.  As a more practical alternative, we propose Gradient based ANLM (GANLM), which adjusts the orientation of an anisotropic ANLM neighborhood using an estimate of the edge orientation provided by the image gradients.  Pseudocode is given in Algorithm \ref{alg:ganlm}.  Note that GANLM reverts back to NLM in regions with low image gradients, since they will not be ``edgy'' enough to benefit from the special treatment.

\floatname{algorithm}{Algorithm}
\begin{algorithm}[!t]
\caption{\emph{Gradient-based ANLM (GANLM)}}
\label{alg:gradient}
\begin{algorithmic}%[Anisotropic Nonlocal Means]
\STATE \textbf{Inputs:} 
\STATE $\hat{f}_{i,j}$ : Estimate of the image pixel
\STATE $\delta_s \times \delta_\ell$: Size of the neighborhood 
%\STATE $g_{x}$: Image gradient in x direction 
%\STATE $g_{y}$: Image gradient in y direction 
%\STATE $\theta_{i,j}$: Edge orientation at pixel $(i,j)$ with respect to x direction 
\STATE $\lambda$: Threshold that determines isotropic/anisotropic neighborhood selection
\\[3mm]
\STATE{Estimate image gradient $(g_{h}(i,j), g_v(i,j))$ at each pixel $(i,j)$ }
\FOR {every pixel $(i,j) \in I $}
\STATE $g(i,j) = \sqrt{g_{h}^{2}(i,j)+g_{v}^{2}(i,j)}$
\STATE $\theta_{i,j} = \displaystyle \arctan \left(\frac{g_{v}(i,j)}{g_{h}(i,j)}\right)$
\STATE
%\STATE $\delta_{\ell_{1}} \! \!= \displaystyle \frac{(2\delta+1)}{1+g_{i}}$
%\STATE
%\STATE $\delta_{\ell_{2}} \! \!= \displaystyle \frac{(2\delta+1)^{2}}{\delta_{\ell_{1}}}$
\IF{$g_{i} \ge \lambda$}
\STATE Perform ANLM at pixel $y_{i,j}$ with $d_{\delta_{\ell},\delta_{s},\theta_{i,j}}$
\ELSE
\STATE Perform NLM at pixel $y_{i,j}$ with  $\sqrt{\delta_s \delta_{\ell}}$.
\ENDIF
\ENDFOR
\end{algorithmic}
\label{alg:ganlm}
\end{algorithm}

There is a rich literature on robust image gradient estimation \cite{Feng:2002fk,Li:2001uq,Bigun:1991kx}.  In our implementations we use the simplest method for estimating the image gradient, i.e.,
\begin{eqnarray*}
g_h(i,j) = \hat{f}_{i+1,j} - \hat{f}_{i,j}, \nonumber \\
g_v(i,j) = \hat{f}_{i,j+1} - \hat{f}_{i,j},
\end{eqnarray*}
where $\hat{f}(i,j)$ is an estimate of the image (for instance from isotropic NLM).  If $g_h(i,j)$ and $g_v(i,j)$ are the estimated image derivatives at pixel $(i,j)$, then we can estimate the local orientation of an edge by $\hat{\theta}(i,j) = \arctan\!\left(\frac{g_v(i,j)}{g_h(i,j)}\right)$. To allay any concerns that gradient-based adaptivity is not robust to noise and errors, we recall Theorem \ref{thm:thetamismatch}, which establishes the robustness of OANLM to edge angle estimation error.  For extremely noisy images, numerous heuristics are possible, including estimating the image gradients for GANLM from an isotropic NLM pilot estimate.  Figure \ref{fig:hor_denoised} builds on Figure \ref{fig:ONLMvsNLM} with a slightly more realistic, curved edge and demonstrates that the pilot estimate approach to GANLM performs almost as well as an oracle GANLM that has access to the edge gradient. 

%{\bf RICHB: PLEASE TELL ME WHY:  DANLM looks so clean in Figure \ref{fig:hor_denoised} - why is its PSNR lower than GANLM?} Lower PSNR % answered in  Figure 9. DANLM no longer looks cleaner than GANLM

\begin{figure}
	\captionsetup[subfigure]{justification=centering,labelformat=parens} %{labelformat=empty}
	\centering
	\subfloat[Noisy image, ${\rm PSNR}= 16.5$dB]{\label{fig:hornoise} 
	%\framebox{
	\includegraphics*[width = .32\textwidth,height=.325\textwidth]{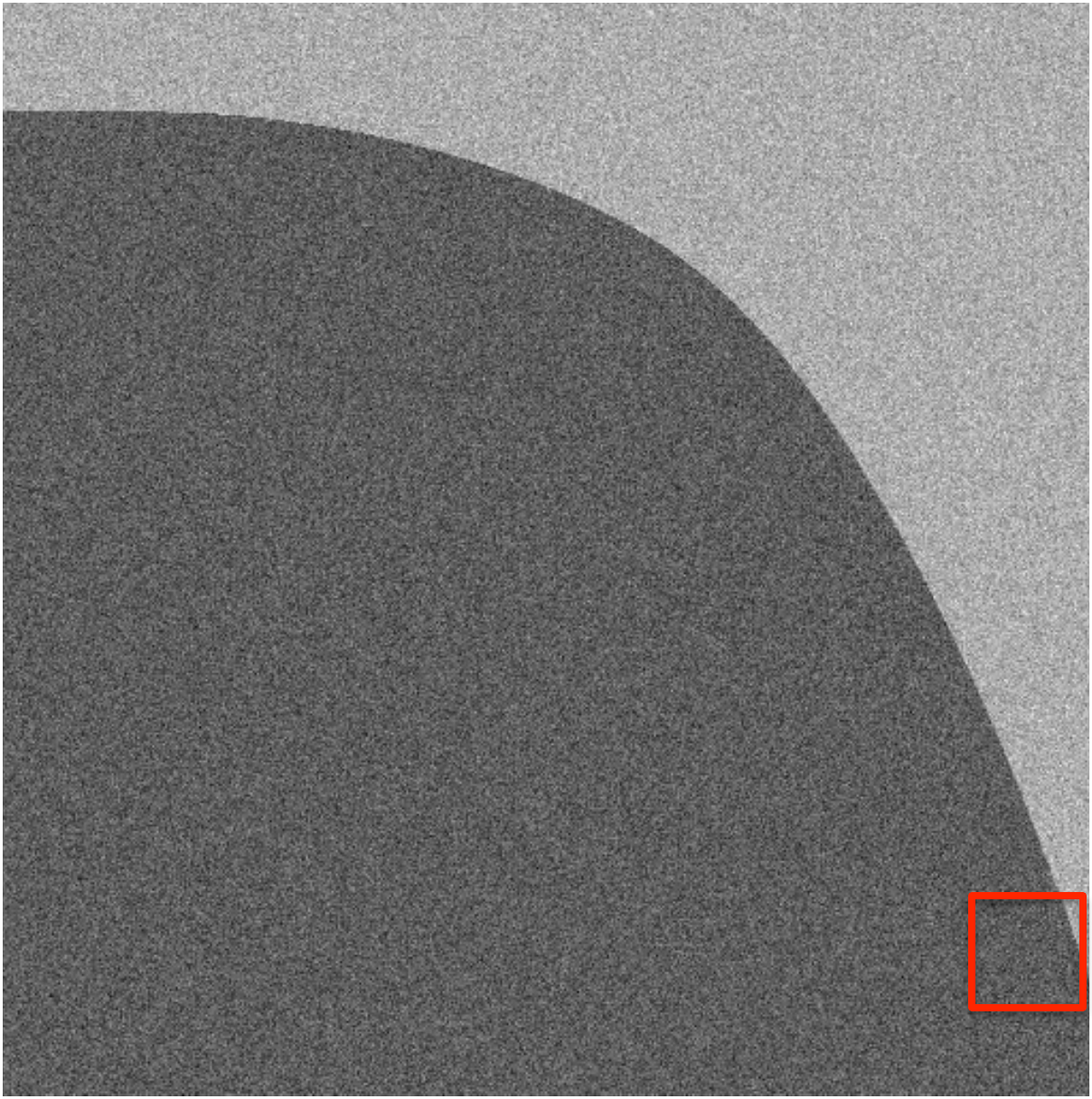}}	
	\subfloat[Horizon image]{\label{fig:hor}	
	%\framebox{ 
	\includegraphics*[width = .32\textwidth]{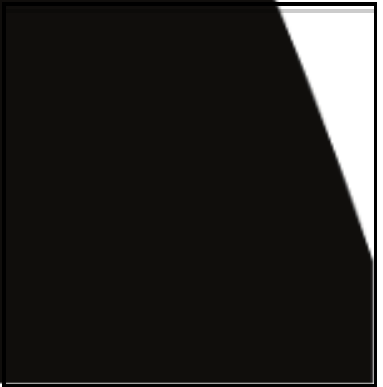}}
	\subfloat[NLM, ${\rm PSNR}= 34.4$dB]{\label{fig:nlm} 
	%\framebox{
	\includegraphics*[width = .32\textwidth,height=.32\textwidth]{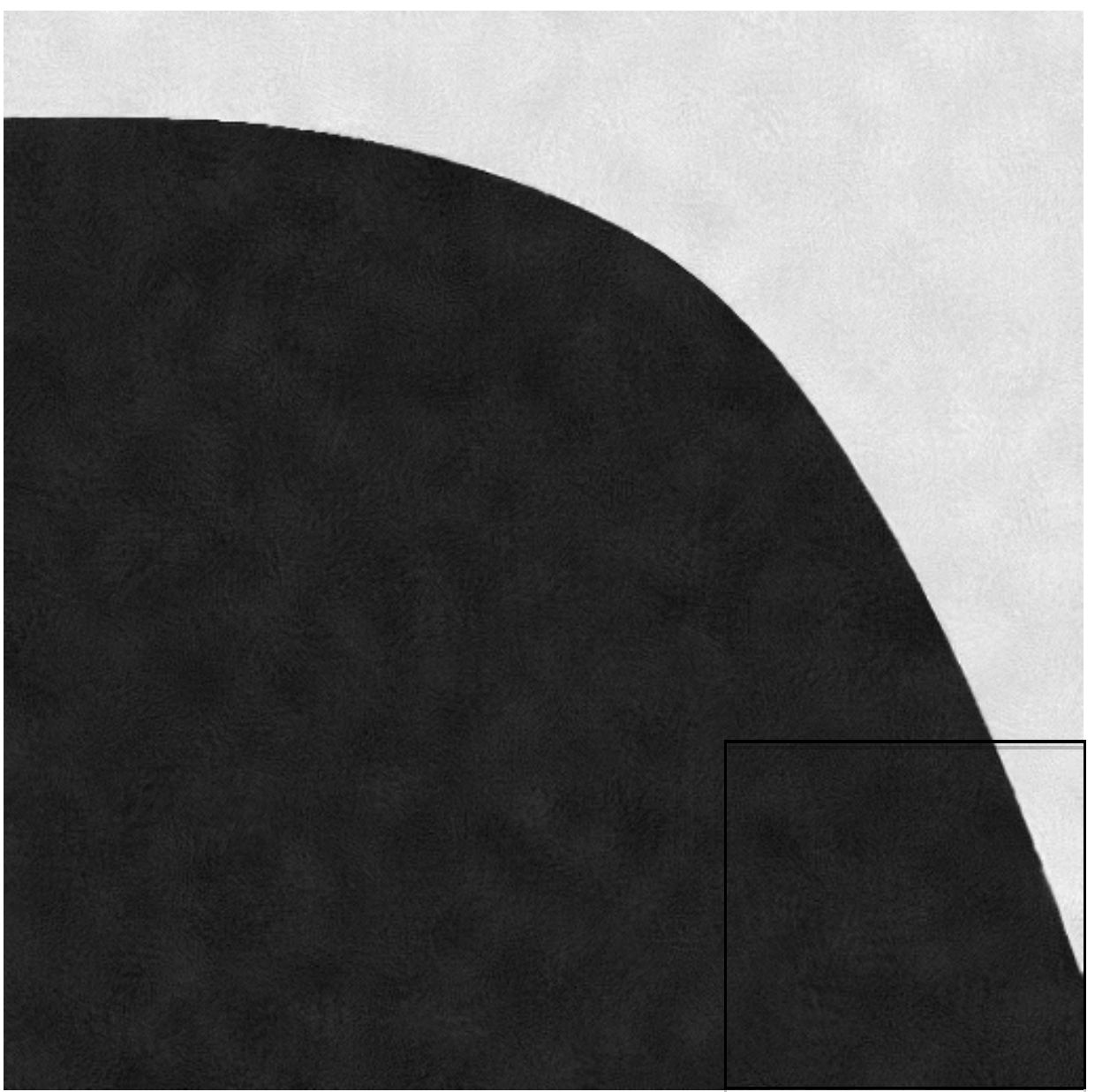}} \\
	\subfloat[DANLM, ${\rm PSNR}= 36.2$dB]{\label{fig:hor_danlm}
	\includegraphics*[width = .32\textwidth]{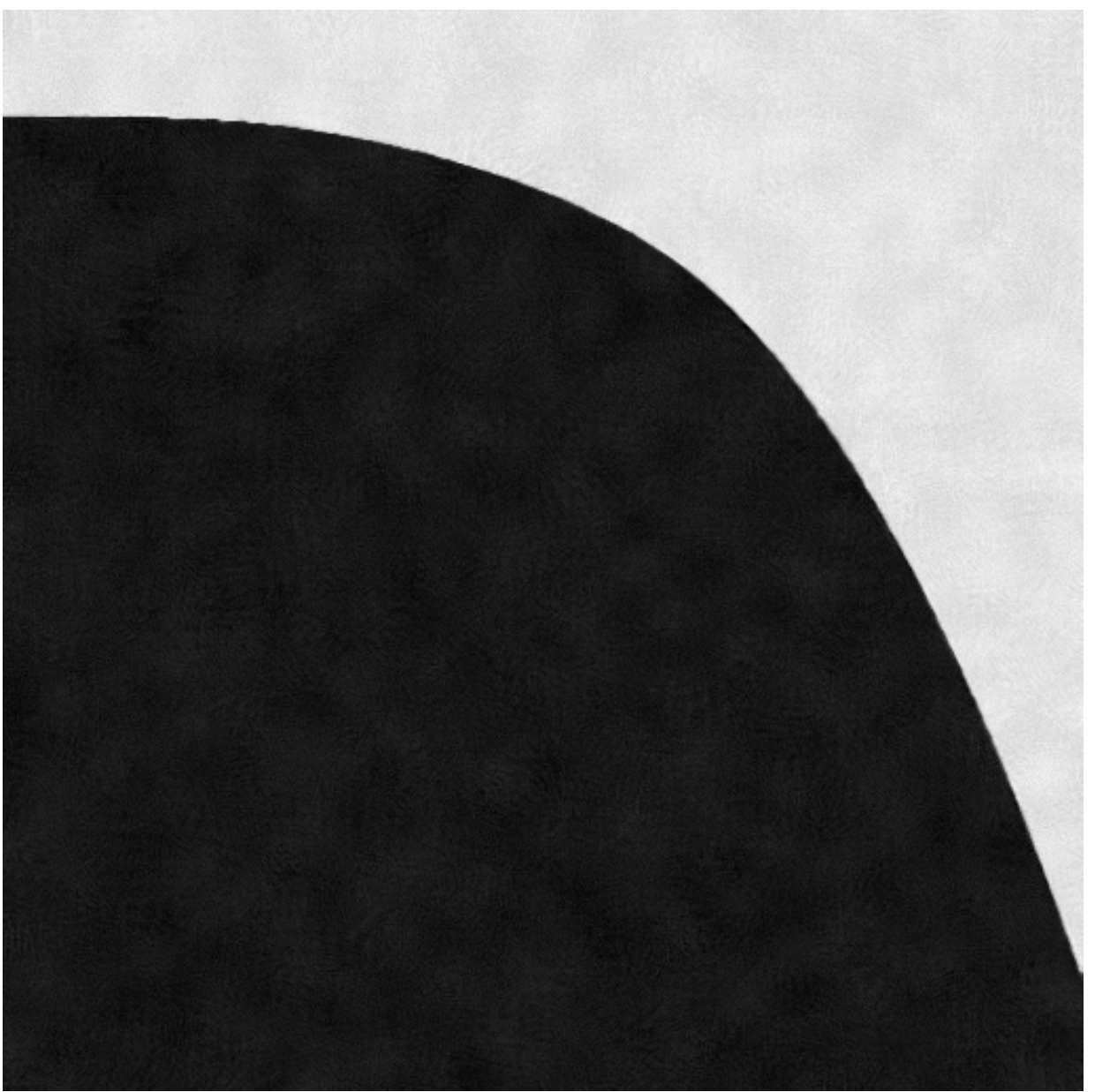}}
	\subfloat[Oracle GANLM, ${\rm PSNR}= 38$dB]{\label{fig:hor_ganlm}
	\includegraphics*[width = .32\textwidth]{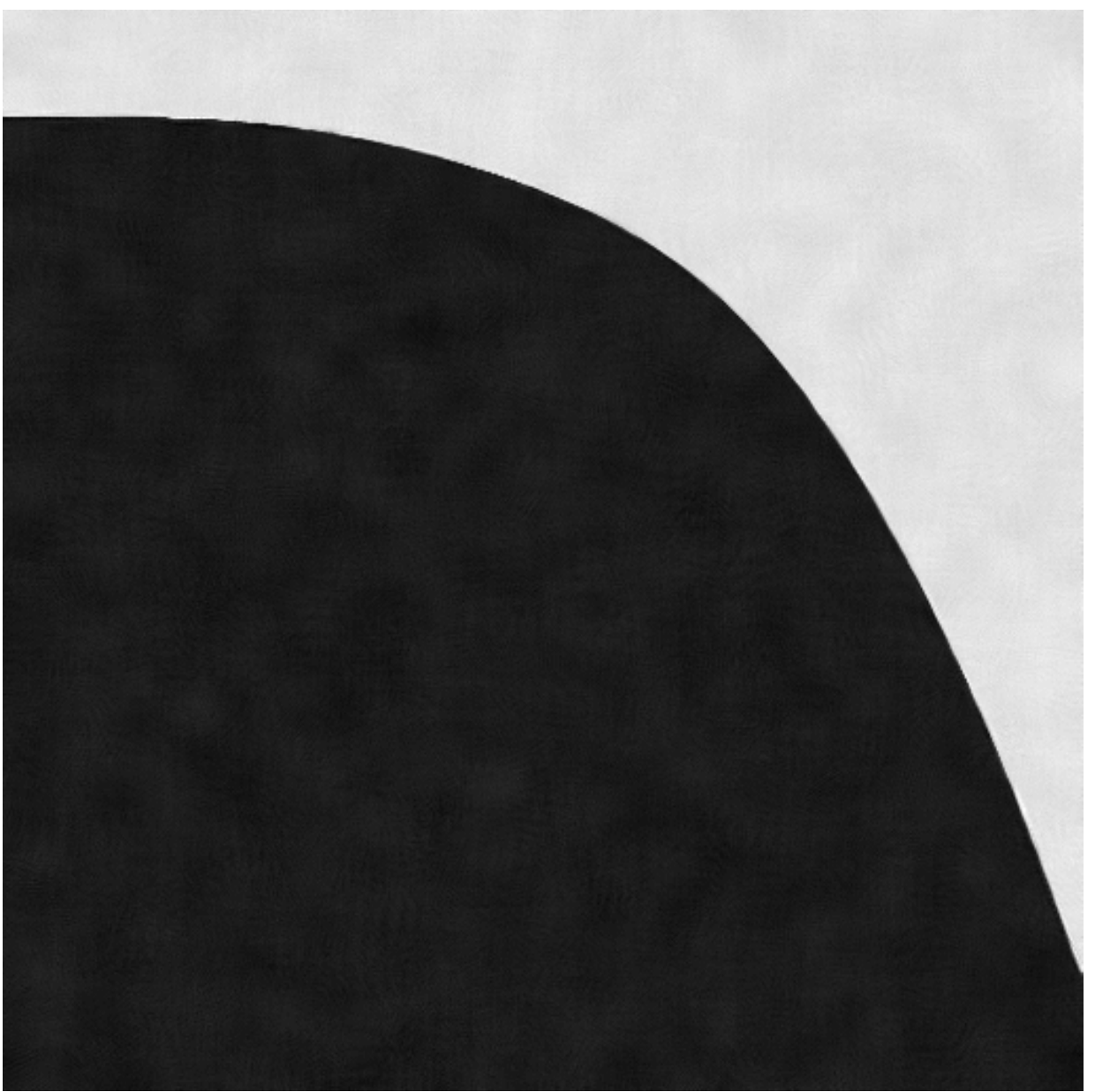}} 
	\subfloat[GANLM, ${\rm PSNR}= 38$dB]{\label{fig:hor_eganlm}
	\includegraphics*[width = .32\textwidth]{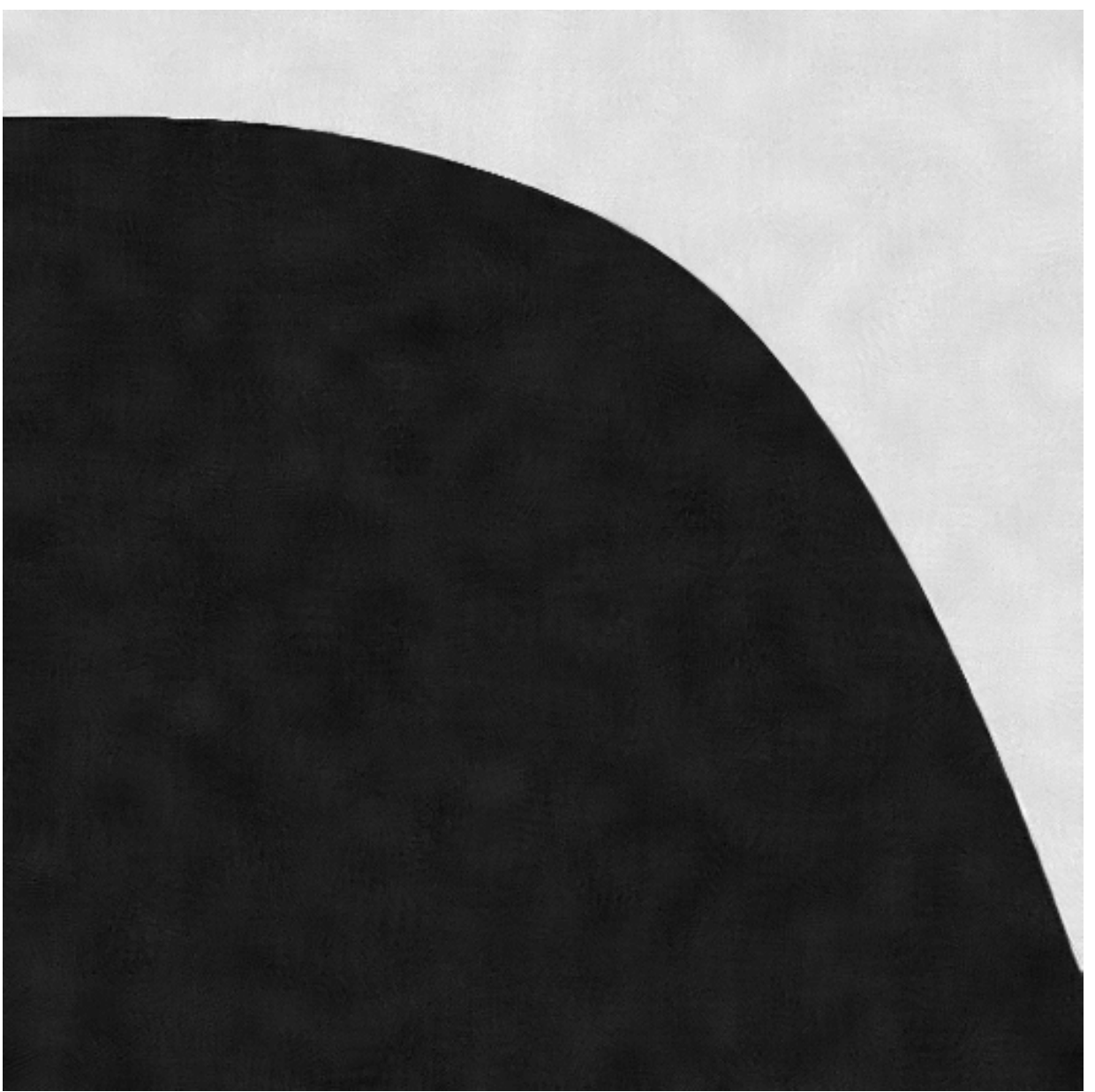}}
	\caption{Simulation experiment in the same vein as Figures \ref{fig:tuning} and \ref{fig:ONLMvsNLM} but with a slightly more realistic $C^2$ curved edge.  The images are of size $512\times 512$ pixels. In the ANLM algorithms $\delta_s$ and $\delta_{\ell}$ follow the quadratic scaling of Theorem \ref{thm:oracleANLM}. The neighborhood size of NLM is $\sqrt{\delta_s \times \delta_{\ell}}$. The standard deviation of the noise is $\xi = 0.15$. To show the differences we have zoomed on the area represented by the red block in Figure (a). } % {\bf RICHB: figure out how to center text for subfigs b e f.}}
\label{fig:hor_denoised}
\end{figure}

Table \ref{tab:t2} summarizes the performance of the algorithms introduced in this paper with that of NLM on the natural test images Barbara \cite{Barbara}, Boats  \cite{Boats}, and Wet Paint \cite{WetPaint} submerged in noise of two different levels.  The table enables us to draw two conclusions.  First, the performance of the practical GANLM algorithm is very close to the oracle GANLM algorithm.  Second these two algorithms outperform DANLM in all but one case and significantly outperform standard NLM in all cases. We use 4 discrete angles\footnote{Experiments with larger $q$ did not differ appreciably in PSNR at this image size.} in DANLM and attribute the superior performance of GANLM over DANLM to a more accurate selection of orientations.

\subsection{Computational complexity of ANLM}

Improving the risk performance of linear filters by nonlocal averaging is achieved at the price of higher computational cost.  If $|\mathcal{P}|$ denotes the neighborhood size, then the computational complexity of NLM is proportional to $\Theta(n^4 |\mathcal{P}|)$ as compared to $\Theta(n^2)$ for linear filters. Such high complexity is not acceptable in many applications. Therefore, several effective methods have been proposed to reduce the computational complexity of NLM. See \cite{Buades:2005p4221, MaSa05, OrEbWo08, BrKlCr08} and the references therein for more information on the research in this direction. These approaches have reduced the computational complexity of NLM to an acceptable limit for practice.  

The computational complexity of ANLM is in general higher than that of NLM. For instance, the computational complexity of the Gradient-based ANLM algorithm introduced in Section \ref{ssec:simugrad} is roughly two times that of NLM if we ignore the gradient calculation, which is $\Theta(n^2)$. The computational complexity of DANLM is $\Theta(q n^4 |\mathcal{P}|)$, where $q$ denotes the number of angles considered. Most of the approaches proposed for speeding up NLM are applicable to ANLM as well. Furthermore, due to specific structure of DANLM, in particular the overlap of the neighborhoods of each pixel, we expect to be able to reduce its computational complexity even further. But this question is left for future research.

\begin{table}[!t] \footnotesize
\caption{PSNR Performance of NLM and ANLM algorithms on three natural test images at two noise levels. }
\label{tab:t2}
\vspace{-0.5cm}
\begin{center}\label{tabl:mse}
\begin{tabular}{| c || c | c | c | }
       \hline
%	\multicolumn{3}{|c|}{$\text{Noise} = 17.24 \mathrm{dB} $} \\
%         \hline
%         \hline
%	Test Image &  Algorithm & PSNR \\ \hline
%	         \multirow{3}{*}{Barbara} & NLM & 26.4029 \\ 
%	          & Oracle GANLM & 27.2225 \\
%	          & Empirical GANLM & 27.2094 \\
%	          \hline 
%	         \multirow{3}{*}{Wet Paint} & NLM & 29.1668 \\ 
%	          & Oracle GANLM & 31.2571 \\
%	          & Empirical GANLM & 31.1723 \\
%	           \hline
%	          \multirow{3}{*}{Boats} & NLM & 26.3733 \\ 
%	          & Oracle GANLM & 26.8736 \\
%	          & Empirical GANLM & 26.7777 \\
%	           \hline
%	\hline
	%\multicolumn{4}{|c|} \\
        %\hline
        % \hline
	 Test image &  Algorithm & ${\xi = 0.25}$, 12.1dB  & ${\xi = 0.15}$, 16.5dB \\ \hline\hline
%	         \multirow{4}{*}{Straight edge} & NLM & 30.42 & 34.89 \\ 
%	          & Oracle GANLM & 34.48 & 39.03 \\
%	          & Empirical GANLM & 34.29 & 39.00 \\
%	          & DANLM & 32.15 & 36.85 \\
%	          %& BM3D (basic) & & \\
%	          \hline 
%		\multirow{4}{*}{$C^2$ edge} & NLM & 30.09 &  34.43\\ 
%	          & Oracle GANLM & 34.00 & 37.97\\
%	          & Empirical GANLM & 33.48 & 37.95\\
%	          & DANLM & 31.59 & 36.24 \\
%	          %& BM3D (basic) & 31.11 & f
%	          \hline 	        
	         \multirow{5}{*}{Barbara\vspace{25pt}} 
	         & NLM & 22.48 & 25.86\\ 
	          & DANLM & 23.15 & 26.27\\
	          (512 x 512) & Oracle GANLM & 23.51 & 26.63\\
	          & GANLM & 23.50 & 26.60\\
	          %& BM3D (basic) & 24.64 & \\
	          \hline 
	          \multirow{4}{*}{Boats \vspace{10pt}} 
	          & NLM & 22.75 & 25.83\\ 
	          & DANLM & 23.45 & {26.45}\\
	          (512 x 512) & Oracle GANLM &  {23.88} & 26.45\\
	          & GANLM & 23.75 & 26.35\\
	          %& BM3D (basic) & 24.80 & \\
	           \hline
	         \multirow{4}{*}{Wet Paint \vspace{10pt}} 
	         & NLM & 27.66 & 30.49 \\ 
	          & DANLM & 28.41 & 30.75 \\
	          (1024 x 1024) & Oracle GANLM & {29.02} & {31.18}\\
	          & GANLM & 28.86 & 31.07\\
	          % 28.38
	          %& BM3D (basic) & 29.39 & \\
	           \hline
\end{tabular}
\end{center}
\end{table}

%Barbara
%
%% 5.5206
%% 4.9552
%
%Boats
%
%% 5.0209
%% 4.998
%
%Paint 
%
%% 16.4815
%% 12.0363

\begin{figure}
	\captionsetup[subfigure]{justification=centering,labelformat=parens} %{labelformat=empty}
	\centering
	\subfloat[Barbara ]{\label{fig:barb_sig25} 
	\includegraphics*[width = .75\textwidth,height=.325\textwidth]{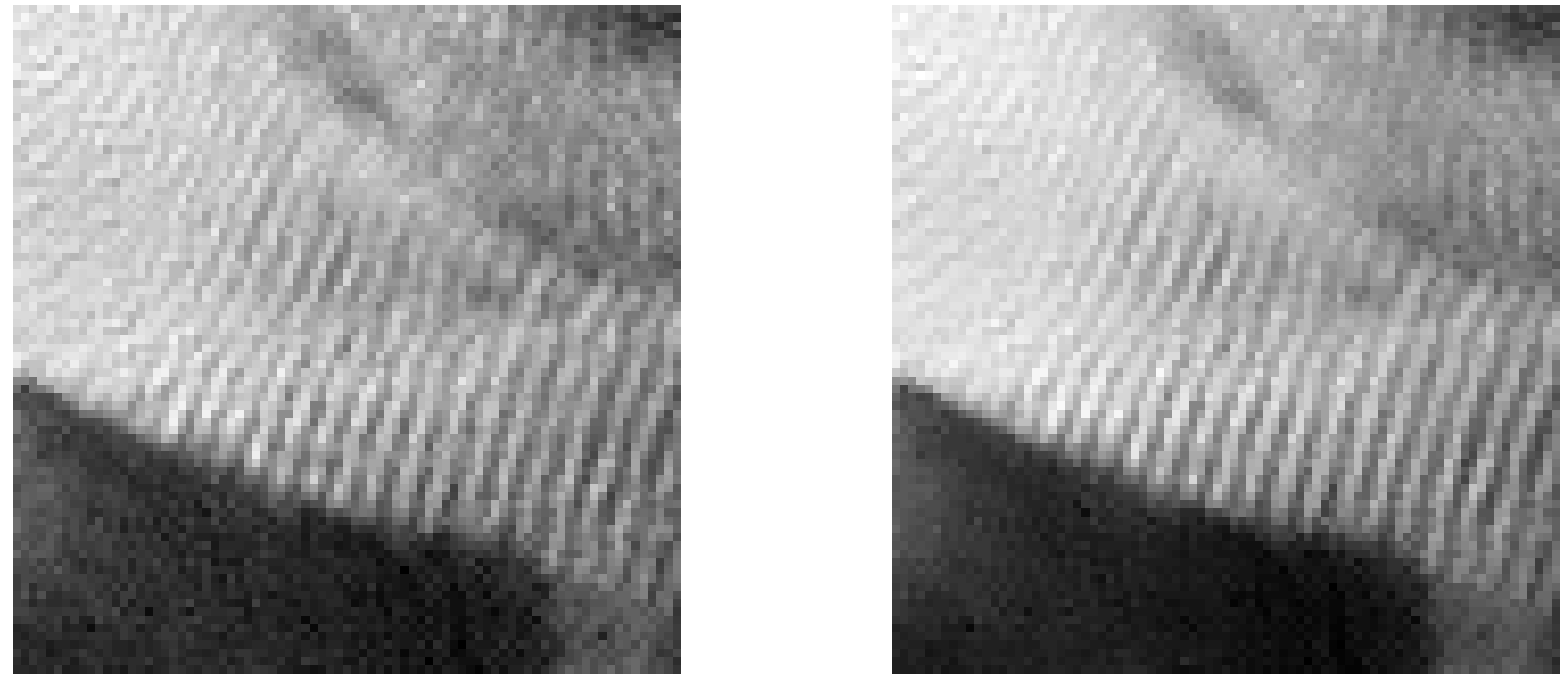}}	\\
	\subfloat[Boats ]{\label{fig:boat_sig25} 
	\includegraphics*[width = .75\textwidth,height=.325\textwidth]{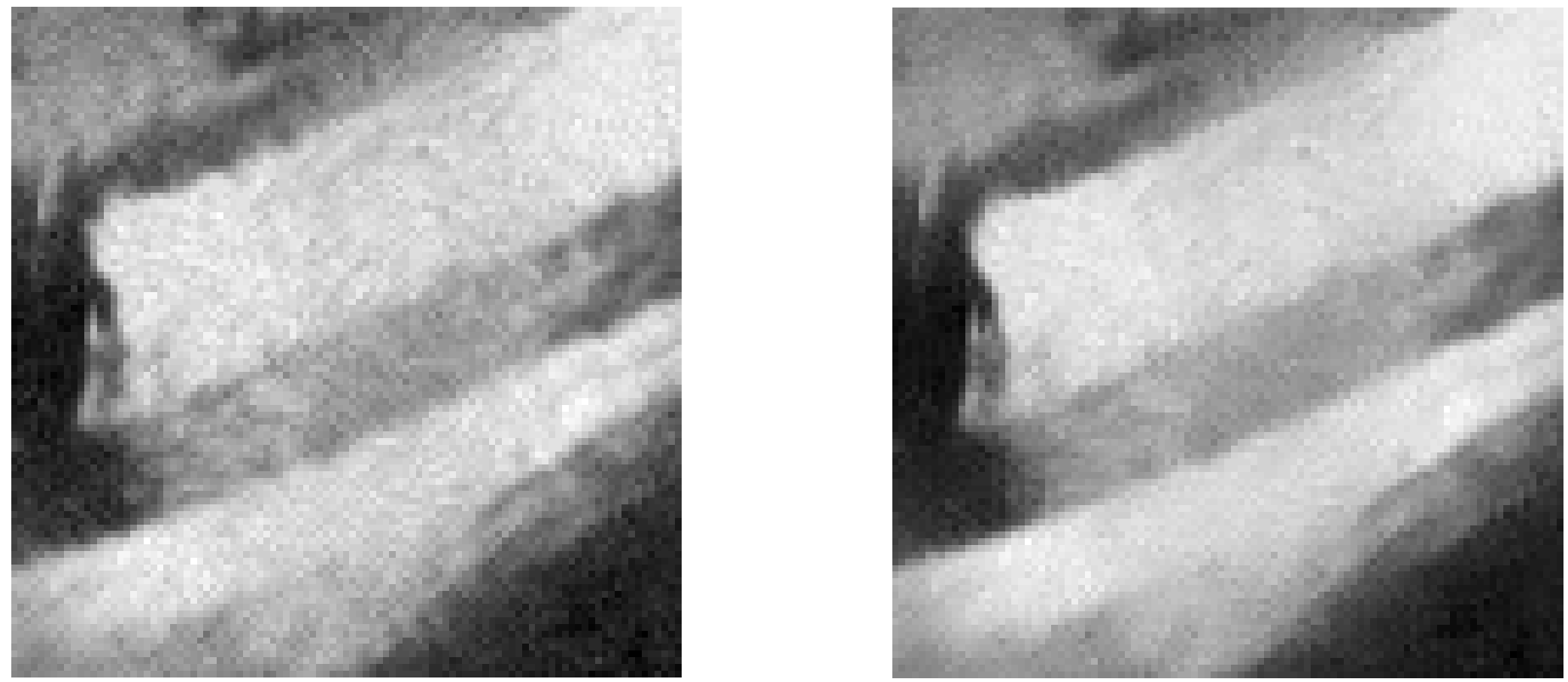}}	\\
	\subfloat[Wet Paint]{\label{fig:paint_sig25} 
	\includegraphics*[width = .75\textwidth,height=.325\textwidth]{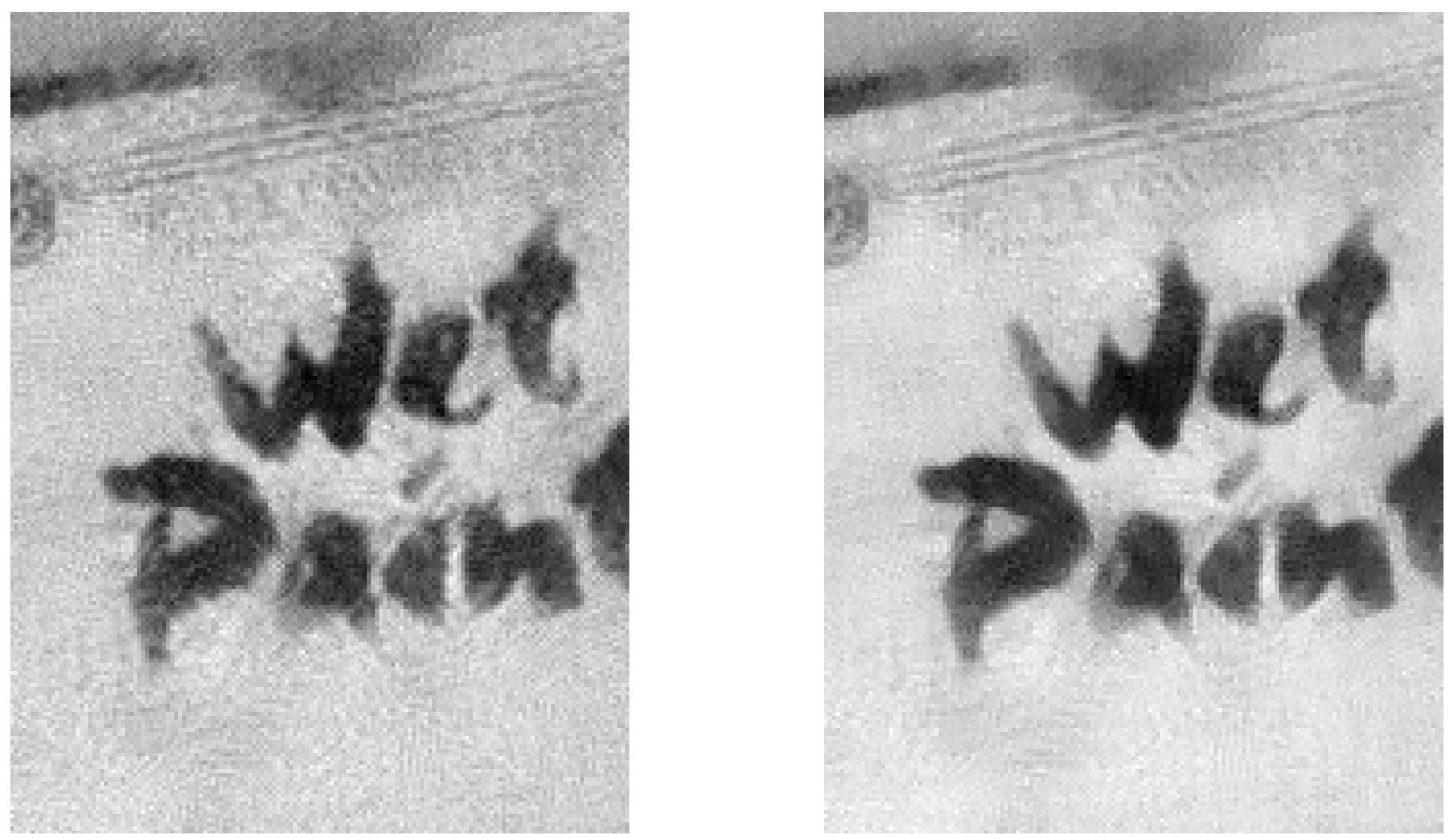}}	
%	\subfloat[Barbara, ANLM]{\label{fig:hor}	
%	%\framebox{ 
%	\includegraphics*[width = .32\textwidth]{imhor_sm.pdf}}
%	\subfloat[NLM, ${\rm PSNR}= 34.4$dB]{\label{fig:nlm} 
%	%\framebox{
%	\includegraphics*[width = .32\textwidth,height=.32\textwidth]{nl_hor_sm.pdf}} \\
%%
%	\subfloat[DANLM, ${\rm PSNR}= 36.2$dB]{\label{fig:hor_danlm}
%	\includegraphics*[width = .32\textwidth]{anihor3_sm.pdf}}
%%
%	\subfloat[Oracle GANLM, ${\rm PSNR}= 38$dB]{\label{fig:hor_ganlm}
%	\includegraphics*[width = .32\textwidth]{anihor_sm.pdf}} 
%%
%	\subfloat[GANLM, ${\rm PSNR}= 38$dB]{\label{fig:hor_eganlm}
%	\includegraphics*[width = .32\textwidth]{anihor2_sm.pdf}}
%
	\caption{Zooms of the natural test images denoised by NLM (left) and ANLM (right) studies in Table \ref{tab:t2}. The standard deviation of the noise is $\xi = 0.15$ that implies PSNR$ = 16.5$dB.} % {\bf RICHB: figure out how to center text for subfigs b e f.}}
\label{fig:hor_denoised}
\end{figure}

\section{Related work on anisotropic denoising}\label{sec:related}

Anisotropy is a fundamental feature of real-world images. As a result, anisotropic image processing tools can be traced back at least as far as the 1970s \cite{NaMa79}. Here, we briefly compare and contrast some of the relevant anisotrpic denoising schemes with ANLM.

{\em Anisotropic filtering} methods use a space-varying linear convolution filters to reduce blur along image edges. Toward this goal,  \cite{NaMa79} considers several different neighborhoods around a pixel and selects the most ``homogeneous'' neighborhood to smooth over.  A more advanced version of this idea can be found in \cite{Takeda06, TaFaMi07, EgKaAs01}.  There are major differences between such algorithms and NLM/ANLM; in particular, the estimators are local and do not exploit global similarities. See also \cite{MilanfarTour} for a more complete overview of such algorithms and their differences with NLM/ANLM.
 
 \emph{Anisotropic diffusion} methods smooth a noisy image with a Gaussian kernel. As the standard deviation of the kernel increases, the smoothing process introduces larger bias to the edges. In \cite{HuMo89,Koenderink84} the authors proved that the set of images derived by this approach can be viewed as the solution to the heat diffusion equation. Perona and Malik \cite{PeMa90} noted the isotropy in the heat equation and introduced anisotropy. Their anisotropic diffusion starts with the heat equation but at every iteration exploits the gradient information of the previous estimate to increase the conductance along the edges and decrease it across the edges.  See \cite{KiMaSo00, Tsch06, KiSo02} for more recent advances in this direction. Efforts to theoretically analyze the risk of such algorithms have left many open questions remaining \cite{CaDo09}. It is conjectured that there is a connection between such algorithms and iteratively applying nonlinear filters such as the median filter. It is worth noting that the idea of applying an image denoising algorithm iteratively and guiding it at every iteration based on previous estimates goes back to Tukey's twicing \cite{Tukey77}.

 \emph{Anisotropic transformations} enable simple threshold based denoising algorithms.  While the standard separable wavelet transform cannot exploit the smoothness of the edge contour, a menagerie of anisotropic transforms have been developed, including ridgelets \cite{Can99-1}, curvelets \cite{Donoho:2000p1676}, wedgelets \cite{Donoho:1999p1950}, platelets \cite{Wil03,Wil07}, shearlets \cite{KuLa07, GuLa07, EaLaLi08}, contourlets \cite{DoVe05}, bandelets \cite{1407972, PePeDoMa07}, and directionlets \cite{Vel06}. As mentioned in the Introduction, among the above algorithms only wedgelets can obtain the optimal minimax risk for the Horizon class; however wedgelets are not suited to denoising textures.  One promising avenue combing wavelets (for texture denoising) and wedgelets (for edge denoising) could follow the path of the image coder in \cite{WaRoChBa06}. See \cite{LaWeWeWi02, JaDuCuPe11} for an overview of anisotropic transformations and their applications in image processing. 
 
% In this paper we proposed new anisotropic algorithms motivated by the minimax optimality criterion of Korostelev and Tsybakov. To the best of our knowledge this paper and \cite{MaNaBa11} are the only two theoretical analyses of the NLM where demonstrate the suboptimality of NLM and introduce anisotropy to the algorithm to achieve the optimal minimax rate. 

Alternative NLM algorithms have been proposed to address the inefficiency of using a fixed neighborhood 
\cite{ GrZiWe11,  PiMrDiGr10, DeDuSa11, DaFoKatEgi08, SaSt10, Wong:2008p2002,Kervrann:2006p33, Dabov:2007p4629}. In \cite{DeDuSa11, SaSt10}, the authors aggregate several nonlocal estimates to obtain a final estimate. The nonlocal estimates are derived from NLM with different neighborhoods around each pixel. While these methods improve upon NLM on edges in practice, there is no theoretical result to support such empirical results. Similar approaches have been exploited in several state-of-the-art denoising algorithms \cite{DaFoKatEgi08, Dabov:2007p4629}. In \cite{Wong:2008p2002,Kervrann:2006p33, VaKo09}, the authors adapt the neighborhood size to the local image content. \cite{Kervrann:2006p33} considers different isotropic NLM neighborhood sizes depending on how smoothly the image content varies. \cite{Wong:2008p2002} uses image gradient information to increase the weights of the NLM along edges and decrease them across edges. This method is equivalent to modifying the threshold parameter $t_{n}$ to force NLM to assign higher weights to edge-like neighborhoods. \cite{VaKo09} sets the neighborhood size based on Stein's unbiased estimate of the risk (SURE).  Other generalizations include the rotationally invariant similarity measure \cite{GrZiWe11} and nonlocal variational approaches \cite{PiMrDiGr10, GiOs08, KiOsJo05}. Unfortunately, these techniques do not reduce the bias that renders NLM sub-optimal.

Finally, {\em data-driven optimality criteria} have been considered in \cite{Raphan:2010qy,LeNa11,ChMi10}, where the authors derive lower bounds for the performance of denoising algorithms. However, the analyses provided in these papers are not fully rigorous and do not cover the performance of NLM for images with sharp edges.

\section{Discussion and future directions}\label{sec:conc}

We have introduced and analyzed a framework for anisotropic nonlocal means (ANLM) denoising. Similar to NLM, ANLM exploits nonlocal information in estimating the pixel values. However, unlike NLM, ANLM uses anisotropic, oriented neighborhoods that can exploit the smoothness of edge contours.  This enables ANLM to outperform NLM both theoretically and empirically. In fact, the performance of ANLM is withtin a logarithmic factor of optimal as measured by the minimax rate on the Horizon class of images. 
 
%Although the optimal ANLM algorithm either uses oracle information on the edge contour or
%calculates many distances on a fine grid of directions (which is computationally very expensive), we proved that ANLM outperforms
%NLM even if it has access to an imperfect estimate of the edge direction. This theorem paved our way to propose practical ANLM algorithms, that first obtain an estimate
%of the edge direction and location and then use this estimate in the ANLM algorithm.
%This approach involved the application of the isotropic NLM on the noisy image and use the gradient information of the denoised image to estimate the edge directions. We showed 
%through simulations that even this simple approach, which is reminiscent of twicing algorithm of Tukey and anisotropic diffusion of Perona and Malik,
%leads to significant improvement over the original NLM algorithm. 

Numerous questions remain open for future research. From the theoretical perspective, the risk analysis of GANLM, the application to noise models beyond Gaussian, and the extension to three dimensions and beyond (for seismic, medical, and other data) pose interesting research challenges. From a practical perspective, the question of how to best tune ANLM to match the nuanced edges and textures of real-world images remains open, since we have considered only brutal binary images here.  Finally, while NLM is no longer the state-of-the-art denoising algorithm, it is a key building block in several top-performing algorithms.  It would be interesting to see whether anisotropy pays off as handily for those algorithms as it does for NLM.

\section{Proofs of the main results}\label{sec:proofs}

\subsection{Preamble}

We first introduce some notation. Define the following partitions of the set $S= [0,1]^2$:
\begin{eqnarray*}
S_1 &\triangleq& \{(v, u) \ | \ u > h(v) +(1+ C/2) \delta_s \}, \\
S_2 &\triangleq& \{(v, u) \ | \ h(v)-(1+ C/2)\delta_s < u < h(v)+ (1+C/2)\delta_s \}, \\
S_3 &\triangleq& \{(v,u)\ | \ u< h(v)-(1+C/2) \delta_s \}.
\end{eqnarray*}
It is important to note that, if $(t_1,t_2) \in S_1$ and $\tan(\theta) = h'(t_1)$, then $I_{\theta, \delta_s, \delta_{\ell}} (t_1,t_2)$ does not overlap with the edge
contour. In other words, the correctly aligned neighborhood of $(t_1,t_2) \in S_1$ is always above the edge. The points in $S_3$ satisfy a similar property.
This is clarified in Figure \ref{fig:ANLMDiscreteRegions}.
\begin{figure}
\begin{center}
  \includegraphics[width=6cm]{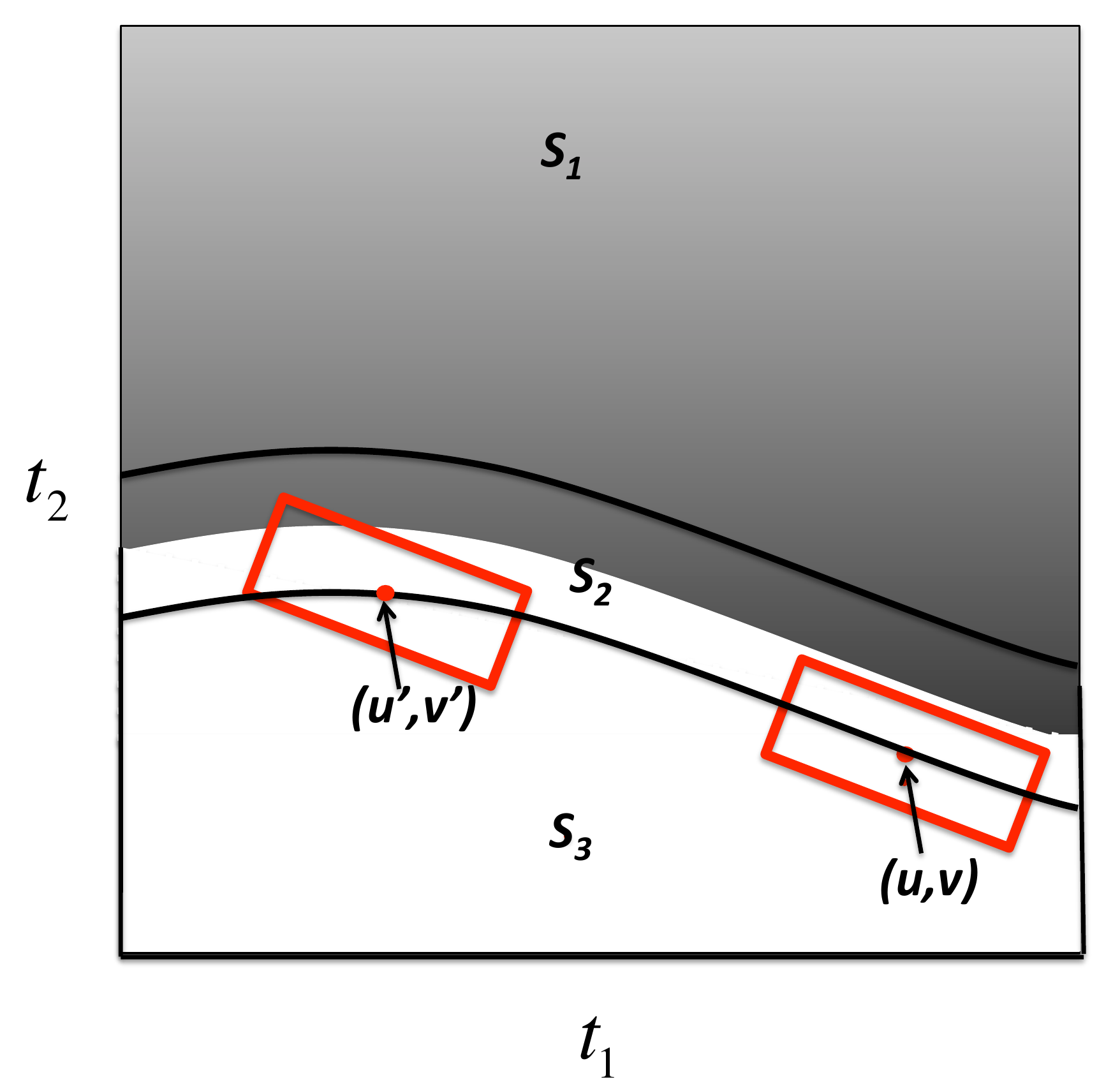}\\
  \caption{Regions $S_1$, $S_2$ and $S_3$. The neighborhood of $(u,v) \in {S}_3$ is aligned with the edge contour, and therefore it does not intersect the edge, while the neighborhood of $(u',v')$ is not aligned and therefore may intersect with the edge. The neighborhoods of the pixels in $S_2$ may intersect the edge even though they are correctly aligned.}
   \label{fig:ANLMDiscreteRegions}
  \end{center}
\end{figure}

\noindent We further partition $S_1$ into $P_1$ and $P_2$ and $S_3$ into $P_3$ and $P_4$ such that
\begin{eqnarray*}
P_1 &\triangleq& \{(v,u) \ | \ h(v) + 2\delta_{\ell} + (C/2) \delta_s \leq u \}, \\
P_2 &\triangleq&\{(v,u) \ | \ h(v) + (1+C/2)\delta_s \leq u \leq h(v)+ 2\delta_{\ell} + (C/2) \delta_s \},\\
P_3 &\triangleq&\{(v,u) \ | \ h(v) - (1+C/2)\delta_s \geq u \geq h(v)- 2\delta_{\ell} - (C/2) \delta_s \}, \\
P_4  &\triangleq&\{(v,u) \ | \ u \leq h(v) -2 \delta_{\ell} -(C/2) \delta_s \}.
\end{eqnarray*}
 
 \begin{figure}[!t]
\begin{center}
  \includegraphics[width=6.3cm]{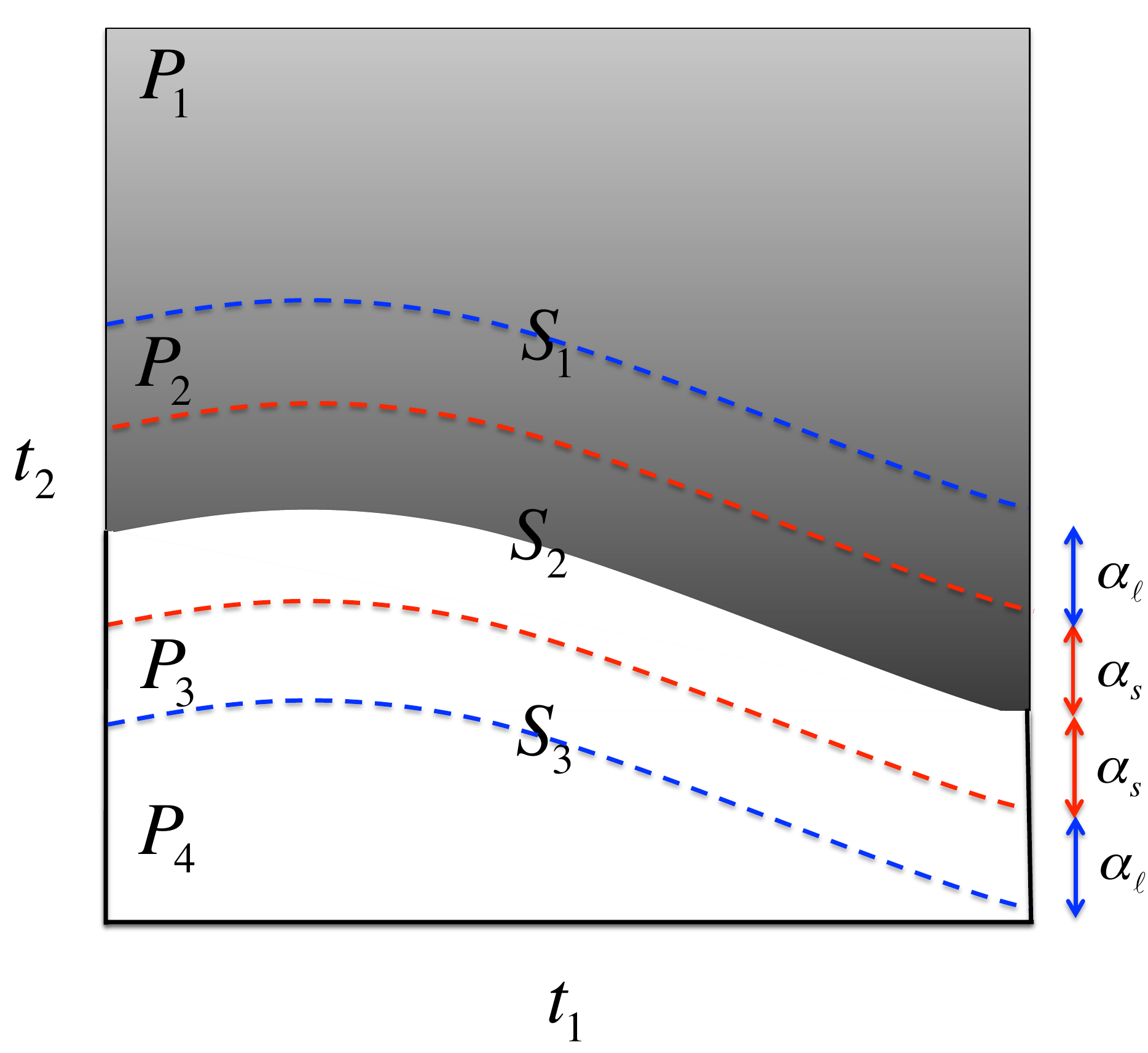}
  \caption{The regions $P_1$, $P_2$, $P_3$, $P_4$ with $\alpha_s =(1+ C/2)\delta_s$  and $\alpha_{\ell} = 2\delta_{\ell}-\delta_s$. Every neighborhood of $(t_1,t_2) \in P_1$ will lie completely above the edge contour. However, some of the neighborhoods of the pixels $(t_1,t_2) \in {P}_2$ may intersect the edge. A similar property holds for the regions $P_3$ and $P_4$. }
   \label{fig:RegionsP1P4}
  \end{center}
\end{figure}

\begin{lem}
Any neighborhood of pixel $(v, u) \in P_1$ will lie completely above the edge contour. 
\end{lem}
 \begin{proof}
 Let $(t_1,t_2) \in P_1$. If $(u,v) \in I_{\theta, \delta_s, \delta_{\ell}} (t_1,t_2)$, then 
 \begin{equation}\label{eq:neighborhood1}
 t_2-v<\delta_{\ell}.
 \end{equation}
  On the other hand, for $t_1' \in [t_1 - \delta_{\ell}, t_1+ \delta_{\ell}]$ we have,
 \[
 h(t^{\prime}_1) = h(t_1) + h'(t_1)(t'_1-t_1) + \frac{1}{2} h^{\prime\prime}(t^{\prime \prime}_1) (t^{\prime}_1-t_1)^2, 
 \]  
 where $t^{\prime \prime}_1$ is between $t_1$ and $t'_1$. According to the definition of the Horizon class, $h \in {\sl H\ddot{o}lder}^{\alpha}(C)$ and hence $|h'(t_1)| \leq 1$. Therefore,
 \begin{equation} \label{eq:edgecurve2}
 h(t^{\prime}_1) - h(t_1) < \delta_{\ell} + \frac{1}{2} C\delta_s, 
 \end{equation}
Comparing \eqref{eq:neighborhood1} and \eqref{eq:edgecurve2} completes the proof.
\end{proof}
In spite of $P_1$, some of the neighborhoods of the pixels $(v,u) \in P_2$
may intersect the edge. Similarly one can prove that any neighborhood of a pixel $(v,u) \in P_4$ lies completely below the edge, and some of the neighborhoods of the pixel $(v,u) \in P_3$ may intersect the edge. Figure \ref{fig:RegionsP1P4} displays these regions. 
The following lemma, proved in \cite{MaNaBa11}, will be used in the proofs of the main theorems. For the sake of completeness, we sketch the proof here. \medskip

\begin{lem}\label{thm:chisq_conc}
Let $Z_1, Z_2, \ldots, Z_r$ be iid $N(0,1)$ random variables. The $\chi^2_r$ random variable $\sum_{i=1}^r Z_i^2$ concentrates around its mean with high probability, i.e.,
\begin{align} 
&\P\left(\frac{1}{r}\sum_i Z_i^2 -1 >t \right) \leq {\rm e}^{-\frac{r}{2}(t- \ln (1+t))}, \label{eq1:concchis} \\
&\P\left( \frac{1}{r}\sum_i Z_i^2 -1 <-t \right) \leq {\rm e}^{-\frac{r}{2}(t+ \ln (1-t))}.  \label{eq2:concchis}
\end{align}
\end{lem}

\begin{proof}
Here we prove \eqref{eq1:concchis}; the proof of \eqref{eq2:concchis} follows along very similar lines. From Markov's Inequality, we have
\begin{eqnarray}\label{eq:prob:chisquare}
\lefteqn{\P\left (\left( \frac{1}{r} \sum_{i=1}^r Z_i^2 \right) -1 > t \right) \leq {\rm e}^{-\eta t-\eta} \E\left({\rm e}^{\frac{\eta}{r} \sum_{i=1}^{r}  Z_i^2} \right)}  \nonumber \\
&= &{\rm e}^{-\eta t-\eta} \left(\E\left({\rm e}^{\frac{\eta Z_1^2}{r}} \right)\right)^r =\frac{ {\rm e}^{-\eta t-\eta}}{\left( 1- \frac{2\eta}{r} \right)^{\frac{r}{2}}}.\hspace{2cm}
\end{eqnarray}
The last inequality follows from Lemma 3 in \cite{MaNaBa11}. The upper bound proved above holds for any $\eta < \frac{r}{2}$. To obtain the lowest upper bound we minimize $\frac{ {\rm e}^{-\eta t-\eta}}{\left( 1- \frac{2\eta}{r} \right)^{\frac{r}{2}}}$ over $\eta$. The optimal value is  $\eta^{\star} = \arg\min_{\eta}  \frac{ {\rm e}^{-\eta t-\eta}}{\left( 1- \frac{2\eta}{r} \right)^{\frac{r}{2}}} = \frac{rt}{2(t+1)}$. Plugging $\eta^\star$ into (\ref{eq:prob:chisquare}) proves the lemma.
\end{proof}

\subsection{Proof of Theorem \ref{thm:oracleANLM}}

In the set up above, we consider three different regions for the point $(t_1,t_2)$. As we will see in the proof,
the risk of all pixels in each region has the same upper bound. We calculate these upper bounds and then
combine them to obtain a master upper bound for the risk of OANLM.

\bigskip
\noindent Case I: $(t_1,t_2) \in S_1$. We know that if the anisotropic neighborhood of $(t_1,t_2)$, $I_{\theta, \delta_s, \delta_{\ell}}$ is aligned with the edge contour, i.e.,
$\tan(\theta) = h'(t_1)$, then it does not intersect the edge contour. To calculate the OANLM estimate we first calculate the weights. Define 
\begin{eqnarray}
\mathbf{z}^{\theta, \delta_s, \delta_{\ell}}_{t_1,t_2}(j_1,j_2) \triangleq  \frac{n_s n_{\ell}}{\delta_s \delta_{\ell}} \int_{(s_1,s_2) \in I^{j_1,j_2}_{\theta, \delta_s, \delta_{\ell}}} dW(s_1,s_2),
\end{eqnarray} 
and
\begin{eqnarray*}
Z(t_1,t_2) &\triangleq & \frac{n_s n_{\ell}}{\delta_s \delta_{\ell}} \int \int_{(s_1,s_2) \in I_{\theta,\frac{\delta_s}{n_s},\frac{\delta_{\ell}}{n_\ell} }(t_1, t_2)} dW(s_1,s_2), \\
F(t_1,t_2) &\triangleq& \frac{n_s n_{\ell}}{\delta_s \delta_{\ell}} \int \int_{(s_1,s_2) \in I_{\theta,\frac{\delta_s}{n_s},\frac{\delta_{\ell}}{n_\ell} }(t_1, t_2)} f(s_1,s_2)ds_1 ds_2.
\end{eqnarray*}
For notational simplicity we will use
$w({s_1,s_2})$ and $\mathbf{z}_{t_1,t_2}(j_1,j_2)$ instead of $w_{t_1,t_2}^{\theta, \delta_s, \delta_{\ell}}(s_1,s_2)$ and $\mathbf{z}^{\theta, \delta_s, \delta_{\ell}}_{t_1,t_2}(j_1,j_2)$, respectively. Define
\begin{eqnarray*}
A_1 &\triangleq& \{ (u,v)\in P_1 \ | \  w({u,v}) =1 \}, \\
A_2 &\triangleq& \{ (u,v)\in P_4 \ | \  w({u,v}) =0 \}.
\end{eqnarray*}
Here, let $\lambda(\cdot)$ denote the Lebesgue measure of a set over $\mathds{R}^2$. 
%
%
% NOT GREAT since $\lambda$ just above was 

\begin{lem}\label{lem:lebesgue}
Let $\delta_{\ell} = 2 \sigma^{2/3} |\log(\sigma)|^{2/3}$ and $\delta_{s} = 4 \sigma^{4/3} |\log(\sigma)|^{4/3}$,  $n_s =2|\log(\sigma)|^{2/3}$ $n_{\ell} = 4|\log(\sigma)|^{4/3}$, and $\tau_{\sigma} = \frac{2}{\sqrt{|\log \sigma|}}$.
Then,
\begin{eqnarray*}
\P(\lambda(P_1) - \lambda(A_1) > \epsilon) &=& O\left( \frac{\sigma^{8}}{\epsilon}\right), \\
\P(\lambda(P_4)-\lambda(A_2) > \epsilon) &=& O\left( \frac{\sigma^{8}}{\epsilon} \right).
\end{eqnarray*}
\end{lem}
\begin{proof}
Here we prove the first expression. The proof of the second expression follows along a similar route. Let $(u,v) \in P_1$. Since $\mathbf{z}_{t_1,t_2}(j_1,j_2)$ is the integral 
of the Wiener sheet, we have $\mathbf{z}_{t_1,t_2}(j_1,j_2) \sim N(0, \frac{n_s n_{\ell} \sigma^2}{\delta_s \delta_{\ell}})$. Similarly, $\mathbf{z}_{u,v}(j_1,j_2) \sim N(0, \frac{n_s n_{\ell} \sigma^2}{\delta_s \delta_{\ell}})$. Combining this with Lemma
\ref{thm:chisq_conc} we conclude that
\begin{eqnarray*}
\P\left(d^2_{\theta, \delta_s, \delta_{\ell}}(dY(t_1,t_2), dY(u,v)) - 2\frac{n_s n_{\ell}\sigma^2}{\delta_s \delta_{\ell}} \geq \tau_{\sigma} \right) \leq {\rm e}^{-\frac{n_sn_{\ell} t^2}{4}} 
  = O(\sigma^8).
\end{eqnarray*}
Therefore,
\begin{eqnarray}
\E(\lambda(A_1)) &=& \E \int_{(u,v) \in P_1} I ((u,v) \in A_1) = \int_{(u,v) \in P_1} \P ((u,v) \in A_1) \nonumber \\
 &\geq&\int_{(u,v) \in P_1} \P \left(d^2_{\theta, \delta_s, \delta_{\ell}}(dY(t_1,t_2), dY(u,v)) - 2\frac{n_s n_{\ell}\sigma^2}{\delta_s \delta_{\ell}} \leq \tau_\sigma\right) \nonumber \\
 &=& \lambda(P_1) - O(\sigma^8).
\end{eqnarray}
An upper bound for $\P(\lambda(P_1) - \lambda(A_1) > \epsilon)$ using the Markov inequality yields the result
\[
\P(\lambda(P_1) - \lambda(A_1) > \epsilon) \leq \frac{\E(\lambda(P_1) - \lambda(A_1))}{\epsilon} = O\left(\frac{\sigma^{8}}{\epsilon}\right).
\]
\end{proof}
This lemma indicates that most of the points in $P_1$ may contribute in the estimation of $(t_1,t_2) \in S_1$ and most of the points in $P_4$ do not contribute in the estimate. Define the event $\mathcal{E}$ as 
\[
\mathcal{E} \triangleq \{ \lambda(P_1) - \lambda(A_1) < \sigma^2 \} \cap \{ \lambda(P_4) - \lambda(A_2) < \sigma^2 \}. 
\]
Note that according to Lemma \ref{lem:lebesgue}, $\P(\mathcal{E}^c) = O(\sigma^6)$. As we will see later this lemma simplifies the calculation of the risk dramatically. For notational simplicity we also define the following notations:
\begin{eqnarray*}
wX du dv &\triangleq& w(u,v) X(u,v) dudv, \nonumber \\
X du dv &\triangleq& X(u,v) du dv, \nonumber \\
w du dv &\triangleq& w(u,v) du dv, \nonumber \\
Z du dv &\triangleq& Z(u,v) dudv, \\
wZ dudv &\triangleq& w(u,v) Z(u,v)du dv.
\end{eqnarray*}

The risk of the OANLM estimator at $(t_1,t_2) \in S_1$ is then given by
\[
\E\left( \frac{\int_S wX du dv}{\int_S wdudv}\right)^2.
\] 
Define $U = \left( \frac{\int_S wX du dv}{\int_S wdudv}\right)^2$. We have
\begin{eqnarray}\label{eq:pcase11}
\E(U) = \E( U \ | \ \mathcal{E}) \P(\mathcal{E}) + \E(U \ | \ \mathcal{E}^c) \P(\mathcal{E}^c) \leq \E(U \ | \ \mathcal{E}) \P(\mathcal{E}) + \P(\mathcal{E}^c).
\end{eqnarray}
Lemma \ref{lem:lebesgue} proves that $ \P(\mathcal{E}^c) = O(\sigma^6)$. Also, since for $(u,v) \in A_1$ $w(u,v)=1$, we have
\begin{eqnarray}\label{eq:lebmark1}
\lefteqn{\left|\int_{P_1} wX du dv - \int_{P_1} X du dv \right|} \nonumber \\
&=&\left| \int_{P_1} wX du dv - \int_{A_1}wXdudv + \int_{A_1}Xdudv- \int_{P_1} X du dv\right| \nonumber \\
&= &  \left| \int_{P_1} wX du dv - \int_{A_1}wXdudv \right|+ \left|\int_{A_1}Xdudv- \int_{P_1} X du dv\right| \nonumber \\
&=& O(  \lambda(P_1\backslash A_1)).
\end{eqnarray}
For the last inequality we have assumed that the estimate is bounded over the $P_1$ region. This assumption can be also justified
by calculating the probability that this condition does not hold and showing that the probability is negligible. Using arguments similar to \eqref{eq:lebmark1}
it is straightforward to prove that

\begin{eqnarray}\label{lem:lebmark2}
&&\left|\int_{P_4} wX du dv  \right| = O(  \lambda(P_4\backslash A_2)). \nonumber \\
&&\left|\int_{P_1} w du dv - \int_{P_1}  du dv \right| = O(  \lambda(P_1\backslash A_1)). \nonumber\\
&&\left|\int_{P_4} w du dv \right| = O(  \lambda(P_4\backslash A_2)). 
\end{eqnarray}

Define $P_{14} \triangleq P_1 \cup P_4$ and $B \triangleq S \backslash P_{14}$. We now calculate an upper bound for the first term of \eqref{eq:pcase11} % i.e., $\E(U \ | \ \mathcal{E}) \P(\mathcal{E})$.
\begin{eqnarray}\label{eq:upporac1}
\lefteqn{ \E( U \ | \ \mathcal{E}) \P(\mathcal{E}) } \nonumber \\
&\overset{\rm (a)}{=}& \E\left( \left( \frac{\int_{P_{14}} wXdu dv + \int_{B} wXdudv}{\int_{P_{14}} wdu dv + \int_{B} wdudv}\right)^2 \  \Big| \ \mathcal{E}   \right)\P(\mathcal{E}) \nonumber \\
&= & \E\left( \left( \frac{\int_{P_{1}} Xdu dv + \int_{B} wXdudv}{\int_{P_{1}} du dv + \int_{B} wdudv}\right)^2 \  \Big| \ \mathcal{E}   \right)\P(\mathcal{E}) +O(\sigma^2) \nonumber \\
 &\leq &\E\left( \frac{\int_{P_{1}} Xdu dv + \int_{B} wXdudv}{\int_{P_{1}} du dv + \int_{B} wdudv}\right)^2  +O(\sigma^2) \nonumber \\
 &\overset{\rm (b)}{\leq} & \E\left( \frac{ \int_{B} wFdudv}{\int_{P_{1}} du dv + \int_{B} wdudv}\right)^2 
  +  \E\left( \frac{\int_{P_{1}} Zdu dv + \int_{B} wZdudv}{\int_{P_{1}} du dv + \int_{B} wdudv}\right)^2 \nonumber \\
 && \! \! \! \!+ \, 2 \sqrt{\E\left( \frac{ \int_{B} wFdudv}{\int_{P_{1}} du dv + \int_{B} wdudv}\right)^2 } \sqrt{ \E\left( \frac{\int_{P_{1}} Zdu dv + \int_{B} wZdudv}{\int_{P_{1}} du dv + \int_{B} wdudv}\right)^2} \nonumber\\
 && \! \! \! \!+ \,O(\sigma^2).
 \end{eqnarray}
 Equality (a) is an application of \eqref{eq:lebmark1} and \eqref{lem:lebmark2}. To obtain (b) we plug in the value of $X = F+ Z$, where $F$ and $Z$ represent signal and noise respectively. In the next two lemmas, Lemma \ref{eq:upporac1} and Lemma \ref{lem:variance}, we obtain upper bounds for individual terms in \eqref{eq:upporac1}.

\begin{lem}\label{lem:bias}
Let $w({u,v})$ be the weights of OANLM with $\delta_{\ell}, \delta_s, n_s, n_{\ell}$, and $\tau_{\sigma}$ as specified
in Lemma \ref{lem:lebesgue}. Then
\begin{eqnarray*}
\E\left( \frac{ \int_{B} wFdudv}{\int_{P_{1}} du dv + \int_{B} wdudv}\right)^2  \leq O(\sigma^{4/3} |\log(\sigma)|^{4/3}).
\end{eqnarray*}
\end{lem}
\begin{proof}
We have  
\begin{eqnarray*}
\lefteqn{\E\left( \frac{ \int_{B} wFdudv}{\int_{P_{1}} du dv + \int_{B} wdudv}\right)^2 \overset{(a)}{\leq} \E\left( \frac{ \int_{B} dudv}{\int_{P_{1}} du dv }\right)^2} \\
 &= & \left(\frac{\lambda(B)}{\lambda(P_1)}\right)^2 = O(\sigma^{4/3} |\log(\sigma)|^{4/3}).
\end{eqnarray*}
To obtain Inequality $(a)$ we maximize the numerator and minimize the denominator independently. The last equality is due to the fact that $\lambda(B) = O(\delta_{\ell})$.
\end{proof}

\begin{lem}\label{lem:variance}
Let $w({u,v})$ be the weights of OANLM with $\delta_{\ell}, \delta_s, n_s, n_{\ell}$, and $\tau_{\sigma}$ as specified in Lemma \ref{lem:lebesgue}. Then
\begin{eqnarray*}
\E\left( \frac{\int_{P_1} Zdudv+ \int_{B} wZdudv}{\int_{P_{1}} du dv + \int_{B} wdudv}\right)^2  \leq O(\sigma^{4/3} |\log(\sigma)|^{4/3}).
\end{eqnarray*}
\end{lem}

\begin{proof}
Since $ \int_{B} wdudv \geq 0$, and we are interested in the upper bound of the risk, we can remove it from the denominator to obtain
\begin{eqnarray*}
\lefteqn{\E\left( \frac{\int_{P_1} Zdudv+ \int_{B} wZdudv}{\int_{P_{1}} du dv + \int_{B} wdudv}\right)^2  \leq \E\left( \frac{\int_{P_1} Zdudv+ \int_{B} wZdudv}{\int_{P_{1}} du dv}\right)^2} \\
&\leq& \underbrace{\E\left( \frac{\int_{P_1} Zdudv}{\int_{P_{1}} du dv}\right)^2} _{V_1}+ \underbrace{\E\left( \frac{\int_{B} wZdudv}{\int_{P_1} du dv }\right)^2}_{V_2} \\
&& + \underbrace{\ 2 \sqrt{\E\left( \frac{\int_{P_1} Zdudv}{\int_{P_{1}} du dv}\right)^2} \sqrt{ \E\left( \frac{\int_{B} wZdudv}{\left(\int_{P_{1}} du dv\right)^2 }\right)^2}}_{V_3}.
\end{eqnarray*}
Now we obtain upper bounds for $V_1$,$V_2$, and $V_3$ separately. First, we have
\begin{eqnarray}\label{eq:uppv1}
V_1 &=& \E\left( \frac{\int_{P_1} Zdudv}{\int_{P_{1}} du dv}\right)^2= \E\left( \frac{\int_{P_1}\int_{P_1} Z(u,v)Z(u',v')du'dv'dudv}{\left(\int_{P_{1}} du dv\right)^2}\right) \nonumber\\
&\overset{(b)}{=}& \left( \frac{\int_{P_1}\int_{I_{\theta, \frac{2\delta_s}{n_s}, \frac{2\delta_\ell}{n_{\ell}} }(u,v)}  \E(Z(u,v)Z(u',v'))du'dv'dudv}{\left(\int_{P_{1}} du dv\right)^2}\right) \nonumber \\
&\overset{(c)}{\leq} & \left( \frac{\int_{P_1}\int_{I_{\theta, \frac{2\delta_s}{n_s}, \frac{2\delta_\ell}{n_{\ell}} }(u,v)}  \frac{n_s n_{\ell} \sigma^2}{\delta_s \delta_{\ell}}du'dv'dudv}{\left(\int_{P_{1}} du dv\right)^2}\right) = O\left(\sigma^2 \right).
\end{eqnarray}
Equality (b) is due to the fact that if $(u',v') \notin I_{\theta, \frac{2\delta_s}{n_s}, \frac{2\delta_\ell}{n_{\ell}} }(u,v)$ then $Z(u,v)$ and $Z(u',v')$ will be independent and hence $\E(Z(u,v)Z(u',v')) = 0$. Inequality (c)  is due to the Cauchy-Schwartz inequality $|\E(Z(u,v)Z(u',v'))| \leq \sqrt{\E Z^2(u,v)}\sqrt{\E Z^2(u',v')}$.  \\

\noindent To obtain an upper bound for $V_2$, we first note that 
\begin{eqnarray}
\lefteqn{\E(w(u,v)Z(u,v)w(u',v')Z(u',v'))} \nonumber \\
&\leq& \sqrt{\E(w(u,v)Z(u,v))^2} \sqrt{\E(w(u',v')Z(u',v'))^2} = \frac{n_s n_{\ell} \sigma^2}{ \delta_s \delta_{\ell}},
\end{eqnarray}
and hence,
\begin{eqnarray}\label{eq:uppv2}
V_2 &=& \E\left( \frac{\int_{B} wZdudv}{\int_{P_1} du dv }\right)^2 \nonumber \\
&=&\E \left( \frac{\int_{B} \int_B w(u,v)Z(u,v) w(u',v')Z(u',v')dudvdu'dv'}{(\int_{P_1} du dv)^2 }\right) \nonumber \\
&\leq& \left( \frac{\int_{B} \int_B   \frac{n_s n_{\ell} \sigma^2}{ \delta_s \delta_{\ell}} dudvdu'dv'}{\left(\int_{P_1} du dv \right)^2 }\right) = O(\delta_{\ell}^2).
\end{eqnarray}
Note that the last inequality is due to the fact that $\lambda(B) = O(\delta_{\ell})$. Using \eqref{eq:uppv1} and \eqref{eq:uppv2} it is straightforward to obtain an upper bound for $V_3$ by using the Cauchy-Schwartz inequality. Combining the upper bounds for $V_1$, $V_2$, and
$V_3$ completes the proof.
\end{proof}

Combining \eqref{eq:upporac1} with Lemmas \ref{lem:bias} and \ref{lem:variance} establishes the following upper bound for Case I, where $(t_1,t_2) \in S_1$:
\begin{eqnarray}
\E \left(f(t_1,t_2)- \frac{\int_S w(u,v) X(u,v)dudv}{\int_S w(u,v)dudv} \right)^2 = O(\sigma^{4/3} |\log(\sigma)|^{4/3}).
\end{eqnarray}

%\bigskip

\noindent Case II: $(t_1,t_2) \in S_2$. In this case we assume that the risk is bounded by $1$, since the function $f$ is bounded between $0$ and $1$. If the estimate
is out of this range, then we will map it to either $0$ or $1$.

%\bigskip

\noindent Case III: $(t_1,t_2) \in S_3$. This case is exactly the same as Case I , and hence we skip the proof.

Finally, combining our results for Cases I, II, and III, we can calculate an upper bound for the risk of OANLM as 
\begin{eqnarray*}
R(f, \hat{f}^O) &=& \E(\|f- \hat{f}^O\|^2) = \int_{S_1 \cup S_3} (f-\hat{f}^O)^2 dt_1 dt_2 + \int_{S_2} (f-\hat{f}^O)^2 dt_1 dt_2 \nonumber \\
&=& \lambda(S_1 \cup S_3)O(\sigma^{4/3} |\log(\sigma)|^{4/3}) + \lambda(S_2) O(1) = O(\sigma^{4/3} |\log(\sigma)|^{4/3}). 
\end{eqnarray*}
This ends the proof of Theorem \ref{thm:oracleANLM}.
  
\subsection{Proof of Corollary \ref{cor:nlm}}\label{sec:proofupnlm}

The proof of this corollary follows exactly the same route as that of Theorem \ref{thm:oracleANLM}. We merely redefine the regions
$S_k$ and $P_k$ for $k \in \{1,2,3 \}$. The new definition of the $S_k$ regions for the NLM algorithm is given by
\begin{eqnarray*}
S_1^{n} &\triangleq& \{(t_1,t_2) \ | \ t_2> h(t_1)+ 2\delta  \}, \\
S_3^{n} &\triangleq& \{(t_1,t_2) \ | \ t_2< h(t_1)- 2\delta  \}, 
\end{eqnarray*}
and $S_2^{n} \triangleq S \backslash S^n_1 \cup S^n_3$. Since the neighborhoods in NLM are isotropic, we require a single parameter to describe the neighborhood size, $\delta = \delta_{s} = \delta_{\ell}$. For notational simplicity we have assumed that $\frac{C}{2} \delta^2 < \delta$. This is a valid assumption for small enough values of $\sigma$, since $\delta \rightarrow 0$ as $\sigma \rightarrow 0$. This assumption implies that the neighborhood of the pixels in $S_1^n$ and $S_3^n$ does not intersect with the edge contour. Since the neighborhoods no longer have different anisotropic lengths, the size of the intersection of the neighborhoods and the edge contour is identical in all directions and there is no need to define the regions, $P_1, \ldots, P_4$. Equivalently, we set $P_1 = S_1^n$, $P_4= S_3^n$ and $P_2 = P_3 = \emptyset$. With these definitions the rest of the proof follows along the exact same lines as the proof of Theorem \ref{thm:oracleANLM}. In particular one can show that the risk of a pixel $(t_1,t_2) \in S_1^n \cup S_3^n$ is $O(\sigma^2\log^2(\sigma))$, and the risk of a point $(t_1,t_2) \in S_2^n$ is $O(1)$. Furthermore, $\lambda(S_2) = O(2\delta)= O(\sigma \log(\sigma)|)$. Therefore the entire risk of NLM is equal to 
\begin{eqnarray*}
R(f, \hat{f}^N) &=&\int_{S_1 \cup S_3} (f-\hat{f}^N)^2 dt_1 dt_2 + \int_{S_2} (f-\hat{f}^N)^2 dt_1 dt_2 \nonumber \\
&=& \lambda(S_1 \cup S_3)O(\sigma^2 \log^2(\sigma)) + \lambda(S_2) O(1) = O(\sigma |\log(\sigma)|). 
\end{eqnarray*}

\subsection{Proof of Theorem \ref{thm:thetamismatch}}
The proof of this theorem is very similar to the proof of Theorem \ref{thm:oracleANLM}. The only difference is in the definitions of $S_k$ and $P_k$ for $k \in \{1,2,3 \}$.
Since there is a mismatch between the orientations of the neighborhood and the  edge contour, the neighborhood of a point in $S_1$ may intersect the edge contour. In order
to fix this, we define the new regions called $S_i^{\beta}$ and $P_i^{\beta}$. If the error in $\theta$ is upper bounded by $c_{\beta} \sigma^{\beta}$, then define
\begin{eqnarray*}
S_1^{\beta} &\triangleq& \{(t_1,t_2) \ | \ t_2> h(t_1)+c_{\beta}\sigma^{\beta} \delta_{\ell} + \delta_s + (C/2) \delta_{\ell}^2  \}, \\
S_3^{\beta} &\triangleq & \{(t_1,t_2) \ | \ t_2< h(t_1)-c_{\beta}\sigma^{\beta} \delta_{\ell} - \delta_s - (C/2) \delta_{\ell}^2  \}, 
\end{eqnarray*}
and $S_2^{\beta} \triangleq S \backslash S^{\beta}_1 \cup S^{\beta}_3$. Let $(t_1,t_2) \in S_1^\beta$ and set $\theta = \arctan(h'(t_1))$. Consider a directional neighborhood of $(t_1,t_2)$ with direction $\hat{\theta}$, where $|\hat{\theta}- \theta| \leq c_{\beta} \sigma^{\beta}$. Then it is straightforward to confirm that 
\[
I_{\hat{\theta}, \delta_s, \delta_{\ell}}(t_1,t_2) \subset S_1^{\beta}.
\]
Furthermore, define $P_1^{\beta}\triangleq P_1$, $P_4^{\beta}\triangleq P_4$, $P_2^{\beta} \triangleq S^{\beta}_1 \backslash P_1$, and
$P_3^{\beta} \triangleq S^{\beta}_3 \backslash P_4$. Note that the risk of a point $(t_1,t_2) \in S_1^{\beta}$ can be exactly calculated as in the proof of Theorem \ref{thm:oracleANLM} and therefore, it is upper bounded by $O(\sigma^{4/3} |\log(\sigma)|^{4/3})$. Also, the risk of the pixels in $S_2^{\beta}$ upper bounded by $O(1)$. Finally the Lebesgue measure of $S_2^{\beta}$ is $O(c_{\beta}\sigma^{\beta} \delta_{\ell} + \delta_s + (C/2) \delta_{\ell}^2)$. Combining these facts establishes the claimed upper bound.

\subsection{Proof of Theorem \ref{thm:ANLM}}

Since the proof is mostly similar to that of Theorem \ref{thm:oracleANLM}, we shall focus on
the major differences. The first difference is that we consider the regions $P_1$--$P_4$ instead of $S_{1}$--$S_{3}$. Previously the regions $P_{1}$ and $P_{2}$ were treated jointly under the region $S_{1}$. Instead, we shall now consider $P_1$ and $P_2$ separately, and their differences shall become progressively apparent.

\noindent Case I: $(t_1,t_2) \in {P}_1$. We start with the calculation of the weights $w({u,v})$ for $(u,v) \in {P}_1 \cup {P}_4$. Define 
\begin{eqnarray*}
R_1 \triangleq \{(u,v) \in P_1 \ | \ w^D({u,v}) = 1 \}, \\
R_2 \triangleq \{(u,v) \in P_4 \ | \ w^D({u,v}) = 0 \}.
\end{eqnarray*}

Intuitively speaking, we expect most of the pixels in $P_1$ to contribute to the estimate of the pixel at $(t_1,t_2)$, since all their directional neighborhoods are very similar to the directional neighborhoods of $(t_1,t_2)$. This implies that $R_1$ is ``close'' to $P_1$.  Also, we expect that most of the pixels in $P_4$ do not contribute in the estimate since all their directional neighborhoods are different from those of $(t_1,t_2) \in P_1$. Lemma \ref{lem:lebR_1} provides a formal statement of this intuition. 
%\medskip

\begin{lem} \label{lem:lebR_1}
Let $\delta_s, \delta_{\ell}, n_s, n_{\ell},t$ be as defined in Lemma \ref{lem:lebesgue} and let $q=\pi \sigma^{-2/3}$ in the DANLM algorithm. Then,
\begin{eqnarray*}
\P(\lambda(P_1) - \lambda(R_1) > \epsilon) &\leq& O\left(\frac{\sigma^8}{\epsilon}\right), \\
\P(\lambda(P_4) - \lambda(R_2) > \epsilon) &\leq& O\left(\frac{ \sigma^{22/3}}{\epsilon}\right).
\end{eqnarray*}
\end{lem}
\begin{proof}
We use the technique we developed in Lemma \ref{lem:lebesgue} for the proof. Consider a point $(u,v) \in P_1$ and let $\theta^*  \triangleq \arctan(h'(t_1))$ and $\hat{\theta}$ represent an angle from $\mathcal{O}$ that is closest to $\theta^*$. Then,
\begin{eqnarray}\label{eq:p1nooverlap}
\lefteqn{\P\left(d^2_{\mathcal{A}}(dY(t_1,t_2), dY(u,v))- \frac{2 n_s n_{\ell} \sigma^2}{\tau_s \tau_{\ell}} \leq \tau_{\sigma}\right) } \nonumber \\
&=&  \P \left( \min_{\theta' \in \mathcal{O}} d^2_{\theta', \delta_s, \delta_{\ell}}(dY(t_1,t_2), dY(u,v))- \frac{2 n_s n_{\ell} \sigma^2}{\tau_s \tau_{\ell}} \leq \tau_{\sigma}\right) \nonumber \\
&\geq&  \P \left( d^2_{\hat{\theta}, \delta_s, \delta_{\ell}}(dY(t_1,t_2), dY(u,v))- \frac{2 n_s n_{\ell} \sigma^2}{\tau_s \tau_{\ell}} \leq \tau_{\sigma}\right)  \nonumber \\
&\overset{(d)}{\geq}& 1- {\rm e}^{-\frac{n_s n_\ell \tau_{\sigma}^2}{4}} = 1- O(\sigma^8),
\end{eqnarray}
where Inequality (d) is a direct consequence of Lemma \ref{thm:chisq_conc}. Again as in the proof of Lemma \ref{lem:lebesgue} we conclude that
$\E(\lambda(P_1)- \lambda(R_1)) = O(\sigma^8)$. Employing this expression in Markov's Inequality leads us to 
\[
\P(\lambda(P_1) - \lambda(R_1) > \epsilon)\leq O\left(\frac{\sigma^8}{\epsilon}\right).
\]
We now prove the second part of this lemma. Consider a point $(u,v) \in P_4$. Then,

\begin{eqnarray*}
\lefteqn{\P\left(d^2_{\mathcal{A}}(dY(t_1,t_2), dY(u,v))- \frac{2 n_s n_{\ell} \sigma^2}{\tau_s \tau_{\ell}} \leq \tau_{\sigma}\right) } \nonumber \\
&=&  \P \left( \min_{\theta' \in \mathcal{O}} d^2_{\theta', \delta_s, \delta_{\ell}}(dY(t_1,t_2), dY(u,v))- \frac{2 n_s n_{\ell} \sigma^2}{\tau_s \tau_{\ell}} \leq \tau_{\sigma}\right) \nonumber \\
&\leq& \sum_{i=1}^q  \P \left( d^2_{\theta_i, \delta_s, \delta_{\ell}}(dY(t_1,t_2), dY(u,v))- \frac{2 n_s n_{\ell} \sigma^2}{\tau_s \tau_{\ell}} \leq \tau_{\sigma}\right)  \\
&{\leq}& O(\sigma^{22/3}).
\end{eqnarray*}
Therefore, we conclude that $\E(\lambda(P_4)- \lambda(R_2)) = O(\sigma^{22/3})$. Combining this with Markov's Inequality establishes the second part of  \ref{lem:lebR_1}. 
\end{proof}

 Define the event $\mathcal{F}$ as
\[
\mathcal{F} \triangleq \{ \lambda(P_1) - \lambda(R_1) < \sigma^2 \} \cap \{\lambda(P_4) - \lambda(R_2)< \sigma^2 \}.
\]
The risk of the OANLM estimator at $(t_1,t_2) \in P_1$ is then given by
\[
\E\left( \frac{\int_S w^{AN}X du dv}{\int_S wdudv}\right)^2.
\] 
Define $U = \left( \frac{\int_S w^{AN}X du dv}{\int_S wdudv}\right)^2$. We have
\begin{eqnarray}\label{eq:pcase11}
\E(U) = \E( U \ | \ \mathcal{F}) \P(\mathcal{F}) + \E(U \ | \ \mathcal{F}^c) \P(\mathcal{F}^c) \leq \E(U \ | \ \mathcal{F}) \P(\mathcal{F}) + \P(\mathcal{F}^c).
\end{eqnarray}
Lemma \ref{lem:lebR_1} confirms that $\P(\mathcal{F}^c) = O(\sigma^{16/3})$. Therefore, we should obtain an upper bound for the term $ \E(U \ | \ \mathcal{F}) \P(\mathcal{F})$.  This step of the proof is the same as the proof of Case I in Theorem \ref{thm:oracleANLM}, and therefore we do not repeat it here. 

The more challenging step of the proof is to obtain an upper bound for the risk of DANLM for pixels $(t_1,t_2) \in P_2$. 

%\bigskip
\noindent Case II: $(t_1,t_2) \in {P}_2$. In this case we again start with defining the following two sets:
\begin{eqnarray*}
L_1 &\triangleq& \{(u,v) \in P_1 \ | \ w^D_{u,v} = 1 \}, \\
L_2 &\triangleq& \{(u,v) \in P_4 \ | \ w^D_{u,v} = 0 \}.
\end{eqnarray*}

As before, our objective is to show that most of the pixels in $P_1$ contribute in the final estimate of $DANLM$ and most of the pixels in $P_4$ do not contribute. The following lemma formalizes this intuition. 

%\medskip

\begin{lem} \label{lem:lebL_1}
Let $\delta_s, \delta_{\ell}, n_s, n_{\ell}, \tau_{\sigma}$ be as defined in Lemma \ref{lem:lebesgue}, and  let $q=\pi \sigma^{-2/3}$. Then,
\begin{eqnarray*}
\P(\lambda(P_1) - \lambda(L_1) > \epsilon) &\leq& \frac{\pi \sigma^{8}}{\epsilon}.
\end{eqnarray*}
\end{lem}

\begin{proof}
We first prove that for some $\theta \in \mathcal{O}$, $I_{\theta, \delta_s, \delta_{\ell}}({t_1,t_2})$ does not intersect the edge.  Suppose that $\theta^*$ is such that $\tan(\theta^*) =h^{\prime}(t_1)$, and consider $\hat{\theta} = \arg \min_{\theta \in \mathcal{O}} |\theta - \theta^*|$.
To ensure that the neighborhood  $I_{\hat{\theta}, \delta_s, \delta_{\ell}}({t_1,t_2})$ does not intersect with the edge, we need to have 
\begin{eqnarray}\label{eq:regAP2_1}
t_2+\tan(\hat{\theta})\left(t'-t_1\right)-\delta_s ~>~ h\!\left(t_1\right)&+& h'\!\left(t_1\right)\left(t'-t_1\right) \nonumber \\
&+& \frac{1}{2}h^{\prime \prime}(t^{\prime \prime})\left(t'-\frac{i}{n}\right)^2
 \end{eqnarray}
 for all $ t^{\prime} \in [t_1,t_1+ \delta_{\ell})$. Also, $t^{\prime \prime} \in [t_1, t^{\prime}]$. According to the mean value theorem we have $\tan(\hat{\theta}) = \tan(\theta^*) + \frac{1}{1+ \tan^2(\theta')} \Delta \theta $, where $\Delta \theta = \hat{\theta}- \theta^*$, and $\theta' \in [- \pi/4, \pi/4]$, since in our Horizon model we have assumed $|h'(t_1)| \leq 1$.
Furthermore, $|h''(t'')| \leq C_2$ and $\Delta \theta \leq \sigma^{2/3} \leq \frac{\pi}{2} \delta_{\ell}$. Therefore, \eqref{eq:regAP2_1} will be satisfied if
\[
t_2> h\left(t_1\right)+ \frac{\pi \delta_{\ell}}{2} \delta_{\ell} + \delta_s + C_2 \delta_{\ell}^2.
\] 
This constraint holds for all pixels in the region $P_2$. Therefore, the neighborhood $I_{\hat{\theta}, \delta_s, \delta_{\ell}}({t_1,t_2})$ lies completely in the region above the $h$. Therefore,
\begin{eqnarray}\label{eq:p1nooverlap}
\lefteqn{\P\left(d^2_{\mathcal{A}}(dY(t_1,t_2), dY(u,v))- \frac{2 n_s n_{\ell} \sigma^2}{\tau_s \tau_{\ell}} \leq \tau_{\sigma}\right) } \nonumber \\
&=&  \P \left( \min_{\theta' \in \mathcal{O}} d^2_{\theta', \delta_s, \delta_{\ell}}(dY(t_1,t_2), dY(u,v))- \frac{2 n_s n_{\ell} \sigma^2}{\tau_s \tau_{\ell}} \leq \tau_{\sigma}\right) \nonumber \\
&\geq&  \P \left( d^2_{\hat{\theta}, \delta_s, \delta_{\ell}}(dY(t_1,t_2), dY(u,v))- \frac{2 n_s n_{\ell} \sigma^2}{\tau_s \tau_{\ell}} \leq \tau_{\sigma}\right)  \nonumber \\
&\overset{(e)}{\geq}& 1- {\rm e}^{-\frac{n_s n_\ell t^2}{4}} = 1- O(\sigma^8).
\end{eqnarray}
where the Inequality (e) is a consequence of \ref{thm:chisq_conc}. By repeating the same argument as the argument in the proof of Lemma \ref{lem:lebesgue} we conclude that
$\E(\lambda(P_1)- \lambda(L_1)) = O(\sigma^8)$. Employing this expression in Markov's Inequality leads us to 
\[
\P(\lambda(P_1) - \lambda(L_1) > \epsilon)\leq O\left(\frac{\sigma^8}{\epsilon}\right).
\]
\end{proof}
%\medskip

\begin{lem}\label{lem:AP4}
Let $\delta_s, \delta_{\ell}, n_s, n_{\ell}, \tau_{\sigma}$ be as defined in Lemma \ref{lem:lebesgue}, and let $q=\pi \sigma^{-2/3}$. Then,
\begin{eqnarray*}
\P(\lambda(P_4) - \lambda(L_2) > \epsilon) &\leq& O \left(\frac{ \sigma^{10/3}}{\epsilon}\right).
\end{eqnarray*}
\end{lem}

\begin{figure}
\begin{center}
  \includegraphics[width=5.6cm]{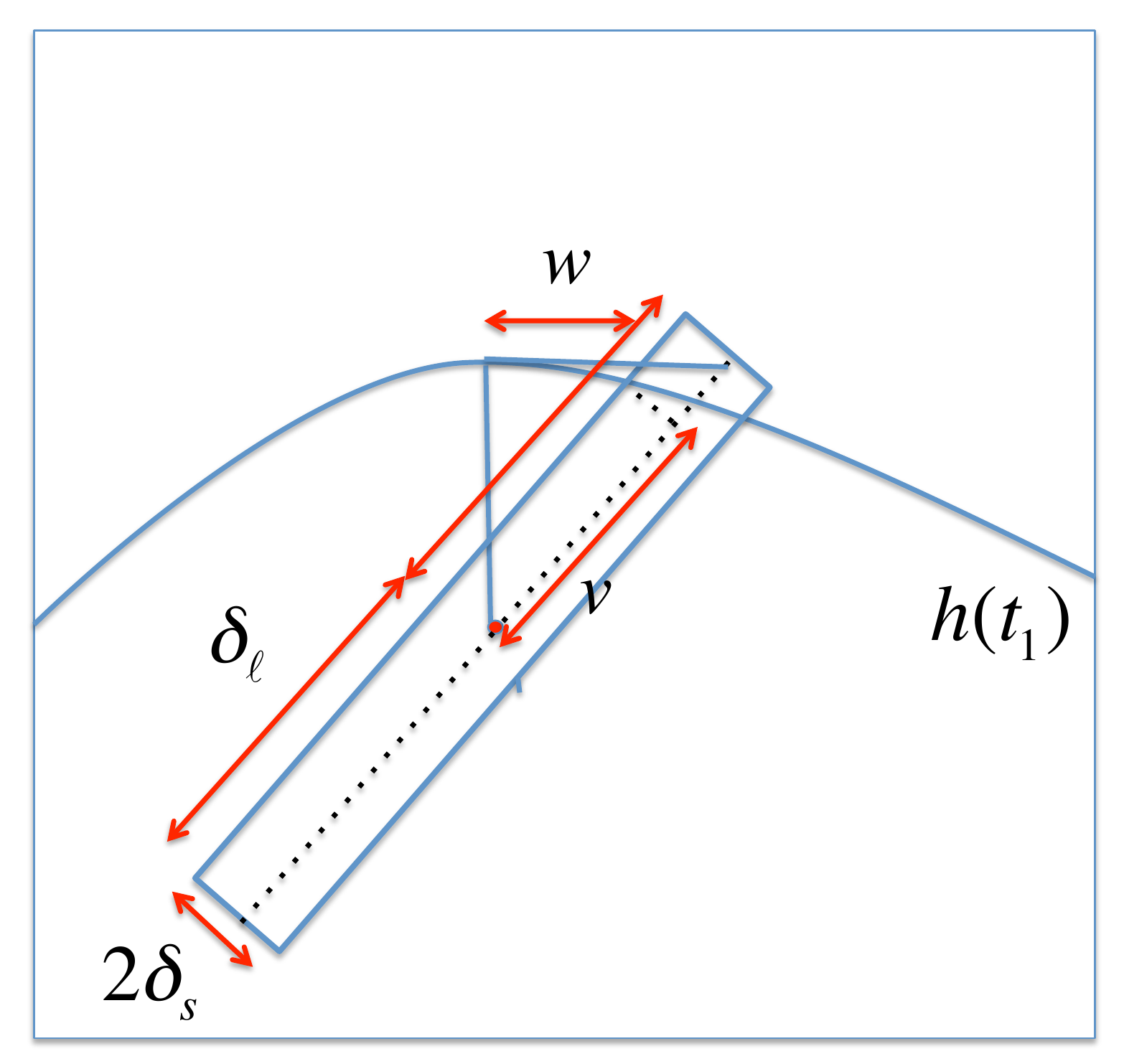}\\
  \caption{Explanation of the neighborhood rotation in Lemma \ref{lem:AP4}.}
   \label{fig:directialpatch_p3}
  \end{center}
\end{figure}

\begin{proof}
We first prove that at least half of the area of $I_{\theta, \delta_s, \delta_{\ell}}(u,v)$ is located below the edge. 
For notational simplicity we assume that $h^{\prime}(t_{1}) = 0$. Consider Figure \ref{fig:directialpatch_p3},  in which the neighborhood has an arbitrary orientation. Let $w,v$ be as defined in this figure.  
We then have that
\begin{equation}\label{eq:yestimate}
v = \frac{d}{\cos \theta}-\frac{Cw^2}{\cos \theta} - \delta_s \tan \theta = \frac{d- Cw^2 - \delta_s \sin\theta}{\cos \theta} \geq \frac{d- Cw^2 - \delta_s}{\cos \theta},
\end{equation}
where $d = t_2- h(t_1)$.
Suppose $d$ is such that, at the correct angle $d- Cw^2 - \delta_s >0$, i.e., $d > (C+1) \sigma^{4/3}$. According to \eqref{eq:yestimate}, $v>0$, and hence
more than half of the area of the neighborhood is in the white region (below the edge contour). Therefore, the number of pixels in this region is $4\log^2\sigma  \pm o(\log^2 \sigma)$. Consider a point $(u,v) \in P_4$. Then,
\begin{eqnarray}\label{eq:p4measure}
\lefteqn{\P\left(d^2_{\mathcal{A}}(dY(t_1,t_2), dY(u,v))- \frac{2 n_s n_{\ell} \sigma^2}{\tau_s \tau_{\ell}} \leq \tau_{\sigma}\right) } \nonumber \\
&=&  \P \left( \min_{\theta' \in \mathcal{O}} d^2_{\theta', \delta_s, \delta_{\ell}}(dY(t_1,t_2), dY(u,v))- \frac{2 n_s n_{\ell} \sigma^2}{\tau_s \tau_{\ell}} \leq \tau_{\sigma}\right) \nonumber \\
&\leq& \sum_{i=1}^q  \P \left( d^2_{\theta_i, \delta_s, \delta_{\ell}}(dY(t_1,t_2), dY(u,v))- \frac{2 n_s n_{\ell} \sigma^2}{\tau_s \tau_{\ell}} \leq \tau_{\sigma}\right) \nonumber \\
&\overset{(g)}{=}&  O(\sigma^{10/3}).
\end{eqnarray}
To derive Inequality (g) we have used the fact that $\Theta(2\delta_{\ell} \delta_s)$ of the the pixels in each neighborhood are equal to $0$. To consider the worst case we assumed the rest of the values are equal to $1$. Using the same technique as in the proof of Lemma \ref{lem:lebR_1} in combination with \eqref{eq:p4measure} we obtain
\begin{eqnarray*}
\P(\lambda(P_4) - \lambda(L_2) > \epsilon) &\leq& O \left(\frac{ \sigma^{10/3}}{\epsilon}\right).
\end{eqnarray*}
\end{proof}
\noindent Based on Lemmas \ref{lem:lebL_1} and  \ref{lem:AP4} we can calculate the risk of DANLM for $(t_1, t_2) \in P_2$.  Define the event $\mathcal{G}$ as
\[
\mathcal{G} \triangleq \{ \lambda(P_1) - \lambda(L_1) < \sigma^2 \} \cap \{\lambda(P_4) - \lambda(L_2)< \sigma^2 \}.
\]
The risk of the OANLM estimator at $(t_1,t_2) \in P_2$ is then given by
\[
\E\left( \frac{\int_S w^{AN}X du dv}{\int_S wdudv}\right)^2.
\] 
Define $U = \left( \frac{\int_S w^{AN}X du dv}{\int_S wdudv}\right)^2$. We have
\begin{eqnarray}\label{eq:pcase11}
\E(U) = \E( U \ | \ \mathcal{G}) \P(\mathcal{G}) + \E(U \ | \ \mathcal{G}^c) \P(\mathcal{G}^c) \leq \E(U \ | \ \mathcal{G}) \P(\mathcal{G}) + \P(\mathcal{G}^c).
\end{eqnarray}
Lemmas \ref{lem:lebL_1} and  \ref{lem:AP4} imply that $\P(\mathcal{G}^c = O(\sigma^{4/3})$. Furthermore, following along the same lines of case I, we can prove $\E( U \ | \ \mathcal{G}) \P(\mathcal{G}) = O(\sigma^{4/3} |\log(\sigma)|^{4/3}|)$.

\noindent Case III: $(t_{1},t_{2}) \in {S}_2$. The proof of this case is identical to that of Case II in the proof of Theorem \ref{thm:oracleANLM}. 

\noindent Case IV: $(t_{1},t_{2}) \in {P}_3 \cup {P}_4$. The proof of this case is similar to that for the regions ${P}_1$ and ${P}_2$ in Cases I and II and therefore is skipped here.

Finally, combining the upper bounds obtained in Cases I, II, III, and IV completes the proof of the theorem.

\section{Acknowledgements}
This
work was supported by the grants NSF CCF-0431150, CCF-0926127,
and CCF-1117939; DARPA/ONR N66001-11-C-4092 and N66001-11-1-4090; 
ONR N00014-08-1-1112, N00014-10-1-0989, and N00014-11-1-0714;
AFOSR FA9550-09-1-0432; ARO MURI W911NF-07-1-0185 and W911NF-09-1-0383; 
and the TI Leadership University Program.

\appendix

\section{Minimax risk of the mean filter}\label{app:meanfilter}

In this appendix, we obtain the decay rate of the risk of the mean filter for the Horizon class of images. Our proof is similar to the proof
in \cite{CaDo09}. Nevertheless, we repeat the proof here for the sake of completeness and since our continuous framework is slightly different from the discrete framework in \cite{CaDo09}. See \cite{MaNaBa11} for the generalization of this result. 
 
The classical mean filter estimates the image via
\begin{eqnarray*}
\hat{f}^{\rm{MF}}(t_1,t_2) = \frac{1}{\Delta^2} \int_{\tau_1 = - \frac{\Delta}{2}}^{\frac{\Delta}{2}} \int_{\tau_2 = - \frac{\Delta}{2}}^{\frac{\Delta}{2}} dY(t_1- \tau_1, t_2- \tau_2),
\end{eqnarray*}
where $\Delta$ specifies the size of the window on which averaging takes place.
\begin{thm}[\rm \cite{CaDo09}]
If $\hat{f}^{\rm{MF}}$ is the estimate of the mean filter, then
\[
\inf_{\Delta} \sup_{f \in \mathcal{H}^{\alpha}(C)} R(f, \hat{f}^{MF})= \Theta(\sigma^{2/3}).
\]
\end{thm}
\begin{proof}
We first derive a lower bound for $R(f, \hat{f}^{\rm{MF}})$. Consider a function $f_{h}(t_1,t_2)$ with $h(t_1) = 1/2$ (recall Figure \ref{fig:compareanisotropic}) and define
\[
Q_{\Delta} = \{(t_1,t_2) \ | \ 1/2- \Delta<t_2< 1/2+ \Delta \}.
\]
It is straightforward to confirm that if $(t_1,t_2) \in Q_{{\Delta}/{4}}$, then $|f_h(t_1,t_2)-\E \hat{f}^{\rm{MF}}(t_1,t_2)|> \frac{1}{4}$. Therefore, if ${\rm Bias}(\hat{f}^{\rm{MF}})$ 
denotes the bias of the mean filter estimator, then we have
\begin{eqnarray}\label{eq:biasmf}
{\rm Bias^2}(\hat{f}^{\rm{MF}}) &=& \int_S |f_h(t_1,t_2) - \E \hat{f}^{\rm{MF}}(t_1,t_2)|^2 dt_1dt_2 \nonumber \\
 & \geq&  \int_{Q_{\Delta/4}} |f_h(t_1,t_2) - \E \hat{f}^{\rm{MF}}(t_1,t_2)|^2 dt_1dt_2 \geq \frac{1}{32} \Delta. 
\end{eqnarray}
Now consider a point $(t_1, t_2) \in S \backslash Q_{\Delta} $. We have
\begin{eqnarray}
\lefteqn{\E (\hat{f}^{\rm{MF}}(t_1,t_1) - \E \hat{f}^{\rm{MF}}(t_1,t_2))^2} \nonumber \\
\! \! \! \! \! \!&=&  \! \! \! \! \!  \int_{\tau_1 = - \frac{\Delta}{2}}^{\frac{\Delta}{2}} \int_{\tau_2 = - \frac{\Delta}{2}}^{\frac{\Delta}{2}}  \int_{\tau'_1 = - \frac{\Delta}{2}}^{\frac{\Delta}{2}} \int_{\tau'_2 = - \frac{\Delta}{2}}^{\frac{\Delta}{2}} \! \! \!  \frac{\E (dW(t_1- \tau_1,t_2- \tau_2) dW(t_1- \tau'_1, t_2- \tau'_2))}{\Delta^4} \nonumber \\
 \! \! \! \! \! \! &=&\! \! \! \!  \frac{\sigma^2}{\Delta^2} \nonumber
 \end{eqnarray}
Therefore, if ${\rm Var}(\hat{f}^{\rm{MF}})$ is the variance of the mean filter estimator, then we have
\begin{eqnarray}\label{eq:varMF}
{\rm Var}(\hat{f}^{\rm{MF}}) \geq \int_{S\backslash Q_{4\Delta}}  \E (\hat{f}^{\rm{MF}}(t_1,t_1) - \E \hat{f}^{\rm{MF}}(t_1,t_2))^2  = \frac{\sigma^2}{\Delta^2}(1 - 4\Delta).
\end{eqnarray}
By combining \eqref{eq:biasmf} and \eqref{eq:varMF} we obtain the lower bound of $\frac{\sigma^2}{\Delta^2}+ \frac{1}{32} \Delta - 4 \frac{\sigma^2}{\Delta}$ for the risk of the mean filter estimator. If we minimize this lower bound over $\Delta$ then we obtain the lower bound of $\Theta(\sigma^{2/3})$ for the risk. 

Now, we derive an upper bound for the risk. Define the region
\[
R_{\Delta} = \{(t_1,t_2) \ | \  h(t_1) - \Delta \leq t_2 \leq  h(t_1) + \Delta\}. 
\]
It is straightforward to confirm that, if $(t_1, t_2) \in S\backslash R_{\Delta}$, then the $\Delta$-neighborhood of this point does not intersect the edge contour. Therefore, the bias of the estimator over this region is zero and the variance is $\frac{\sigma^2}{\Delta^2}$. If we bound the risk of the points in the $R_{\Delta}$ region by $1$, we obtain the following upper bound for the risk,
\[
R(f, \hat{f}^{\rm{MF}}) \leq \frac{\sigma^2}{\Delta^2} + 2\Delta.
\]
Again by optimizing the upper bound over $\Delta$ we obtain the upper bound of $\Theta(\sigma^{2/3})$. This completes the proof.
\end{proof}

\section{Minimax risk of NLM}\label{app:nlm}

Corollary \ref{cor:nlm} states the following upper bound for the risk of NLM:
\[
\sup_{f \in H^{\alpha}(C)}R(f, \hat{f}^N) = O(\sigma |\log \sigma|).
\]
In this section we prove that the risk of NLM is lower bounded by $\Omega(\sigma)$. The proof
is similar to the proof of Theorem 5 in \cite{MaNaBa11};  we consider the function $f_h(t_1,t_2)$ for $h(t_1) = 1/2$ and prove that the bias of the NLM on $(t_1,t_2) \in Q_{\frac{\delta}{n}}$ is $\Theta(1)$ for any choice of the threshold parameter. However, in the continuous setting considered here, the steps are more challenging. Following \cite{MaNaBa11}
we consider the semi-oracle NLM algorithm (SNLM). The semi-oracle $\delta$-distance is defined as
\begin{eqnarray}
 \lefteqn{\tilde{d}^2_{\delta}(dY(t_1,t_2), dY(s_1,s_2))} \nonumber \\
  \! \!&=& \! \! \! \frac{1}{n^2-1} \left( \| \mathbf{x}_{t_1,t_2}^{\delta}- \mathbf{y}_{s_1,s_2}^{\delta}  \|_2^2 
 - |\mathbf{x}_{t_1,t_2}^{\delta}(0,0) - \mathbf{y}_{s_1,s_2}^{\delta}(0,0)|^2\right), \nonumber
\end{eqnarray}
where 
\[
\mathbf{x}_{t_1,t_2}^{\delta}(j_1,j_2) = \int_{(s_1,s_2) \in I^{j_1,j_2}_{\frac{\delta}{n}} } f(s_1,s_2)ds_1ds_2.
\]
SNLM then estimates the weights according to 
\begin{eqnarray}
w^{S}_{t_1,t_2}(s_1,s_2) =\! \! \left\{\begin{array}{rl}
 1 &  \mbox{ if $\tilde{d}^2_{\delta} (dY(t_1,t_2),dY(s_1,s_2)) \leq \frac{n^2\sigma^2}{\delta^2}+ \tau_{\sigma} $,} \\ 
 0 &   \mbox{ otherwise.}
\end{array}\right. \nonumber
\end{eqnarray}
It is clear that the distance estimates of the SNLM algorithm are more accurate than those of NLM algorithm. Hence it outperforms NLM, and a lower bound
that holds for SNLM will hold for NLM as well. 

\noindent We make the following mild assumptions on the parameters of SNLM.
 \begin{itemize}
 \item [A1:] The window size $\delta \rightarrow 0$ as $\sigma \rightarrow 0$. 
 
 \item [A2:] $\frac{\delta^2}{n^2} = \Omega(\sigma^2)$. Otherwise $\E \tilde{d}_{\delta}^2 (dY(t_1,t_2), dY(s_1,s_2)) \rightarrow \infty$ as $\sigma \rightarrow 0$.
 
 \item [A3:] $n \rightarrow \infty$ as $\sigma \rightarrow 0$. This ensures that, if two points $(t_1, t_2)$ and $(s_1, s_2)$ have the same neighborhoods, then
 $w_{t_1,t_2}(u_1,u_2) = 1$ with high probability.
 
 \item[A4:] If $\tilde{d}(f(t_1,t_2), f(s_1,s_2))> 1/4$, then $\P(w^S_{t_1,t_2}(s_1,s_2)=1) = o(\sigma^3)$. 
 \end{itemize}
 
Again, consider the function $f_h(t_1,t_2)$ for $h(t_1) = 1/2$ (recall Figure \ref{fig:compareanisotropic}). Let $(t_1,t_2) \in Q_{\delta/n}$. For notational simplicity we use $w({s_1,s_2})$ instead of $w^S_{t_1,t_2}(s_1,s_2)$. 
 
\begin{lem}\label{lem:sym1}
If $|s_1-t_1| > \delta/2$ and $|s'_1-t_1|> \delta/2$, then
\[
\P(w^{S}_{t_1,t_2}(s_1,s_2)=1) = \P(w^{S}_{t_1,t_2}(s'_1,s_2)=1)
\]
for any $t_1,t_2, s_1,s_2$ and $s'_1$.
\end{lem}
\noindent The proof is straightforward and hence skipped here. Note that we we do not consider the boundary points whose neighborhoods are not subset of $S = [0,1]^2$. 

\noindent Consider a point $(t_1,t_2) \in G_{\delta/n}$. 

\begin{lem}\label{lem:sym2}
If $|s-t_1|> \delta/2$, for $u<\delta/4$, then we have
\[
\P(w^{S}_{t_1,t_2}(s,t_2-u)=1) = \P(w^{S}_{t_1,t_2}(s,t_2+u)=1)
\]
\end{lem}
\noindent This lemma is a straightforward application of symmetry, and hence we skip the proof.

 \begin{lem}\label{lem:appBoutsideweights}
 Suppose that $\delta$ and $\tau_{\sigma}$ satisfy A1--A4. Then we have
 \[
 \P\left(\int_{S \backslash Q_{\delta/2}} w(s_1,s_2)ds_1 ds_2 > \sigma^2 \right)= o(\sigma).
 \]
 \end{lem}
 \begin{proof}
 Define
 \[
 B = \{(s_1,s_2) \in S \backslash Q_{\delta/2} \ | \ w(s_1,s_2)=1 \}
 \]
and the event
 \begin{eqnarray}
 \mathcal{G} = \{ \lambda(B) \geq \sigma^2\}. \nonumber
 \end{eqnarray}
We have
\begin{eqnarray}\label{eq:appbexpec}
\E(\lambda(B)) &=& \E \int \int_{S \backslash Q_{\delta/2}} I((u,v) \in B) du dv \nonumber \\
&=& \int_{S \backslash Q_{\delta/2}} \P((u,v) \in B) du dv \overset{(i)}{=} o(\sigma^3)
\end{eqnarray}
Equality (i) is due to Assumption (A4); at least quarter of the pixel values in the isotropic neighborhood of any $(u,v) \in S \backslash Q_{\delta/2}$ are different from the corresponding pixel values of the neighborhood around $(t_1,t_2) \in Q_{\delta/n}$. Using the Markov Inequality it is straightforward to show that
 \begin{eqnarray}
 \P(\mathcal{G}) &=& \P\left(\int_{S \backslash Q_{\delta/2}} w(u,v) du dv > \sigma^2\right) \leq \frac{\E \left(\int_{S \backslash Q_{\delta/2}} w(u,v) du dv\right)}{\sigma^2}  \nonumber \\
 &=& \frac{\E \left(\int_{S \backslash Q_{\delta/2}} I((u,v) \in B) du dv\right)}{\sigma^2} = o(\sigma),
 \end{eqnarray}
 where the last equality is the result of \eqref{eq:appbexpec}.
Finally, if $\lambda(B)< \sigma^2$, then
 \[
 \int_{S \backslash Q_{\delta/2}} w(u,v) dudv = O(\sigma^2).
 \]
 This completes the proof.\end{proof}

 Our next step toward the proof of a lower bound for the risk of NLM algorithm is to show that the bias of the pixels that are in $Q_{\delta/n}$ is $\Theta(1)$. To achieve this goal we show that if $(t_1,t_2) \ Q_{\delta/n}$ is above the edge and $(s_1,s_2)$ is below the edge then there is a non-vanishing chance (not converging to zero as $\sigma \rightarrow 0$)  that $(s_1,s_2)$ may contribute to the estimate of $(t_1,t_2)$. The following theorem is a formal presentation of this claim. 
 
 \begin{prop}\label{prop:weightprob}
 Let $(t_1,t_2), (s_1,s_2) \in Q_{\delta/n}$. Then there exists $p_0 >0$ independent of $\sigma$ such that for every $t_n, \delta$ we have 
 \[
 \P(w^n_{t_1,t_2}(u,v) = 1) > p_0.
 \]
 \end{prop}
\begin{proof} We need to demonstrate that regions above and below the edge contour are included in the NLM algorithm with a constant non-zero probability $p_{0}$. First, we use the definition of the weights of the NLM algorithm to determine the probability that $w(u,v) = 1$
\begin{eqnarray}
\lefteqn{ \P(w^n_{t_1,t_2}(s_1,s_2) = 1)  = \P \left(\tilde{d}^2_{\delta}(dY(t_1,t_2), dY(s_1,s_2) < \frac{n^2 \sigma^2}{\delta^2} + \tau_{\sigma}\right)} \nonumber \\
&=& \P\left(\frac{1}{n^2-1} \sum\left(s_{\ell,p}^2 - \frac{n^2\sigma^2}{\delta^2}\right) - \frac{1}{n^2-1} \sum s_{\ell,0} \leq -\frac{1}{n}+ \tau_{\sigma}\right) \nonumber\\
&\geq& \P\left(\frac{1}{n^2-1} \sum\left(s_{\ell,p}^2 - \frac{n^2\sigma^2}{\delta^2}\right) - \frac{1}{n^2-1} \sum s_{\ell,0} \leq -\frac{1}{n}\right), \nonumber
\end{eqnarray}
where $s_{\ell,p} = \mathbf{z}_{s_1,s_2}^{0, \delta, \delta}(\ell, p)$. Using the Berry-Esseen Central Limit Theorem for independent non-identically
distributed random variables \cite{Stein86}, we can easily bound this probability away from zero. For more details, see Proposition 1 in \cite{MaNaBa11}. Note that according to Assumption A2 we have $\frac{n^2 \sigma^2}{\delta^2} = O(1)$. 
\end{proof}

We now consider the weights for the regions above and below the edge contour separately. Therefore, we define
\begin{eqnarray}
Q^1_{\Delta} &\triangleq& \{(t_1,t_2) \ | \ 1/2<t_2< 1/2+ \Delta/4 \}, \nonumber \\
Q^2_{\Delta} &\triangleq& \{(t_1,t_2) \ | \ 1/2- \Delta/4<t_2< 1/2 \}, \nonumber \\
\Omega^1_{\Delta,\delta} &\triangleq&  \{(t_1,t_2) \ | \ 1/2<t_2< 1/2+ \Delta/4,\  0<t_1< 2\delta \}, \nonumber \\
\Omega^2_{\Delta,\delta} &\triangleq& \{(t_1,t_2) \ | \ 1/2- \Delta/4<t_2< 1/2,\  0<t_1< 2\delta \}. \nonumber 
\end{eqnarray}
Let
\[
p_{u, \sigma} \triangleq \P(w^n_{(t_1,t_2)}(u,v) = 1). 
\]
Note that according to Lemma \ref{lem:sym1} this probability does not depend on $v$ and hence we have used the notation $p_{u, \sigma}$. 

%If $(t_1,t_2) \in Q^1_{\Delta}$ then 
%\[
%\int_{Q^2_{\delta}} w(u,v)du dv = \Theta(\sigma).
%\]
%\begin{proof}
%We proved in Proposition \ref{prop:weightprob} that  $\P(w^n_{t_1,t_2}(u,v) = 1) > p_0$.
%We have
%\begin{eqnarray*}
%\int_{Q^2_{\delta}} w^S(u,v)du dv = \int_{\Omega^2_{\Delta}} \sum_{k= 1}^{\frac{1}{2 \delta}} w^n(u+ 2k\delta,v)du dv.
%\end{eqnarray*}
%Using Hoeffding inequality we have
%\[
%\P \left\{ \sum_{k= 1}^{\frac{1}{2 \delta}} w^n(u+ 2k\delta,v) - \frac{p_0}{2\delta} < -t \right\} \leq 2 {\rm e}^{- 2 \delta t^2}.
%\]
%Note that due to the construction of the summations the random variables $w^n(u+ 2k\delta,v)$ are independent and therefore we can use Hoeffding inequality.
%Now define $\Gamma$ as
%\[
%\Gamma = \left\{(u,v) \in \Omega^2_{\Delta} \ | \   \sum_{k= 1}^{\frac{1}{2 \delta}} w^n(u+ 2k\delta,v) - \frac{p_0}{2\delta} < -t \right\}
%\]
%Now we have
%\begin{eqnarray*}
%\P(\lambda(\Omega^2_{\Delta}) - \lambda(\Gamma)> \epsilon) &=& \P \left(\int_{(u,v) \in \Omega^2_{\Delta}} (1- \mathds{I}((u,v) \in \Gamma)dudv > \epsilon \right) \\
%&\leq& \frac{2 {\rm e}^{- 2 \delta t^2}}{\epsilon}
%\end{eqnarray*}
%
%\end{proof}

\begin{lem}\label{lem:appBinsidepixels}
If $w(u,v)$ denotes the weights used in the NLM algorithm, then we have
\begin{eqnarray}
\P\left( \left|\int_{Q^1_\Delta} w(u,v) du dv - \int_{Q^1_\Delta}p_{u, \sigma} du\right| >2\frac{\epsilon}{\delta}+ \sqrt{\frac{{|\log(\sigma)|}}{\delta}} \Delta \delta \right)&=& O\left(\frac{\sigma^4}{\epsilon}\right), \nonumber \\
\P\left( \left|\int_{Q^2_\Delta} w(u,v) du dv - \int_{Q^2_\Delta}p_{u, \sigma} du\right| > 2\frac{\epsilon}{\delta}+ \sqrt{\frac{{|\log(\sigma)|}}{\delta}} \Delta \delta \right)&=& O\left(\frac{\sigma^4}{\epsilon}\right). \nonumber
\end{eqnarray}
\end{lem}

\begin{proof} The weights $w(u,v)$ over the pixelated neighborhoods can be rewritten in the following way
\begin{eqnarray}
\int_{Q^1_\Delta} w(u,v) du dv = \int_{\Omega^1_\Delta} \sum_{k=1}^{\frac{1}{2\delta}} w(u+ 2k \delta,v) du dv. \nonumber
\end{eqnarray}
Applying the Hoeffding inequality, it is straightforward to confirm that
\begin{eqnarray}\label{eq:appbhoeff}
\P\left( \left|\sum_{k=1}^{\frac{1}{2\delta}} w(u+ 2k \delta,v)- \frac{p_{u,\sigma}}{2 \delta}\right| > t  \right) \leq 2 {\rm e}^{- 2\delta t^2}.
\end{eqnarray}
Now, defining 
\[
\Gamma^1_\Delta \triangleq  \left\{(u,v)\in \Omega^1_{\Delta} \ \left| \  \Big| \sum_{k=1}^{\frac{1}{2\delta}} w(u+ 2k \delta,v)- \frac{p_{u,\sigma}}{2 \delta}\Big|  < t  \right. \right\},
\]
we see that
\begin{eqnarray*}
\P(\lambda(\Omega^1_{\Delta}) - \lambda(\Gamma^1_{\Delta})> \epsilon) &=& \P \left(\int_{(u,v) \in \Omega^1_{\Delta}} (1- \mathds{I}((u,v) \in \Gamma^1_\Delta)dudv > \epsilon \right) \\
&\overset{(ii)}{\leq}& \frac{ \left( \E \int_{(u,v) \in \Omega^1_{\Delta}} (1- \mathds{I}((u,v) \in \Gamma^1_\Delta)dudv \right)}{\epsilon} \\
&= &  \frac{ \left(  \int_{(u,v) \in \Omega^1_{\Delta}} (1- \P((u,v) \in \Gamma^1_\Delta)dudv \right)}{\epsilon} \\
&\overset{(iii)}{\leq}& \frac{2 {\rm e}^{- 2 \delta t^2}}{\epsilon},
\end{eqnarray*}
where (ii) is due to Markov Inequality and Inequality (iii) is due to \eqref{eq:appbhoeff}. Furthermore,
\begin{eqnarray}\label{eq:appboneterm}
 \lefteqn{\left| \int_{\Gamma^1_\Delta} \sum_{k=1}^{\frac{1}{2\delta}} w(u+ 2k \delta,v) du dv - \int _{\Gamma^1_{\Delta}} \frac{p(u, \sigma)}{2 \delta} du dv \right| } \nonumber \\
&\leq& \int_{\Gamma^1_\Delta} \left| \sum_{k=1}^{\frac{1}{2\delta}} w(u+ 2k \delta,v) - \frac{p(u, \sigma)}{2 \delta}\right| du dv \nonumber \\
& \overset{\rm (iv)}{\leq}& t \lambda(\Gamma^1_\Delta) \leq  t \lambda(\Omega^1_\Delta) \leq t \Delta \delta, 
\end{eqnarray}
where Inequality (iv) is due to the definition of $\Gamma^1_{\Delta}$. Define the event 
\[
\mathcal{F} \triangleq \{\lambda(\Omega^2_{\Delta}) - \lambda(\Gamma^2_\Delta)< \epsilon\}.
\]
Given $\mathcal{F}$ holds, we conclude $\lambda(\Omega^2_{\Delta}) - \lambda(\Gamma^2_\Delta)< \epsilon$ and hence
\begin{eqnarray}
\lefteqn{\left|\int_{Q^1_\Delta} w(u,v) du dv- \int_{Q^1_\Delta} p(u,\sigma) du dv \right|} \nonumber \\
&\leq&\left| \int_{Q^1_\Delta} \sum_{k=1}^{\frac{1}{2\delta}} w(u+ 2k \delta,v) du dv - \int _{Q^1_{\Delta}} \frac{p(u, \sigma)}{2 \delta} du dv \right|  \nonumber \\
&\leq & \left|  \int_{Q^1_\Delta} \sum_{k=1}^{\frac{1}{2\delta}} w(u+ 2k \delta,v) du dv-  \int_{\Gamma^1_\Delta} \sum_{k=1}^{\frac{1}{2\delta}} w(u+ 2k \delta,v) du dv \right|     \nonumber \\
&& + \left|\int_{\Gamma^1_\Delta} \sum_{k=1}^{\frac{1}{2\delta}} w(u+2k \delta,v) du dv -  \int_{\Gamma^1_\Delta} \frac{p(u,\sigma)}{2 \delta} du dv \right|  \nonumber \\
&&+ \left| \int_{\Gamma^1_\Delta} \frac{p(u,\sigma)}{\delta} du dv - \int_{Q^1_\Delta} \frac{p(u,\sigma)}{\delta} du dv \right| \nonumber \\
&\overset {(\rm v)}{\leq}& 2 \frac{\epsilon}{\delta} + t\Delta\delta,  \nonumber
\end{eqnarray}
where Inequality (v) is due to \eqref{eq:appboneterm} and the fact that we have assumed $\mathcal{F}$ holds.
Setting $t = \sqrt{\frac{|\log \sigma|}{\delta}}$ completes the proof.
\end{proof}

\begin{thm}
Suppose that $\delta$, $\tau_{\sigma}$ and $n$ satisfy Assumptions A1--A4. Then the risk of SNLM is
\[
\inf_{\delta_n, t_n} \sup_{f \in H^{\alpha}(C)} R(f, \hat{f}^S) = \Omega(\sigma).
\]
\end{thm}

\begin{proof}
Let $(t_1,t_2) \in \Omega^2_{\delta/n}$. Consider the following two definitions from the proofs of Lemma \ref{lem:appBoutsideweights} and Lemma \ref{lem:appBinsidepixels}:

 \begin{eqnarray}
 \mathcal{G} \triangleq \{ \lambda(B) \geq \sigma^2\}, \nonumber
\end{eqnarray}
where
\[
 B = \{(s_1,s_2) \in S \backslash Q_{\delta/2} \ | \ w(s_1,s_2)=1 \},
 \]
 and 
 \[
\mathcal{F} \triangleq \{\lambda(\Omega^2_{\delta/2}) - \lambda(\Gamma^2_{\delta/2})< \sigma^2\}.
\]
Note that according to Lemmas \ref{lem:appBoutsideweights} and \ref{lem:appBinsidepixels}
\[
\P(\mathcal{F}^c \cup \mathcal{G}^c) = o(\sigma).
\]
We will calculate a lower bound for the risk of NLM. We have
\[
\E \left(f(t_1,t_2) - \frac{\int_S w(u,v) X(u,v)dudv}{\int_Sw(u,v)dudv}\right)^2 \geq \left( \E\left(\frac{\int_S w(u,v) X(u,v)dudv}{\int_Sw(u,v)dudv}\right)\right)^2.
\]
which leads us to towards the lower bound
\begin{eqnarray} \label{eq:lastone}
\lefteqn{\E\left(\frac{\int_S w(u,v) X(u,v)dudv}{\int_Sw(u,v)dudv}\right)  \geq  \E\left(\frac{\int_S w(u,v) X(u,v)dudv}{\int_{Q_{\delta/2}} w(u,v)dudv}\right)} \nonumber \\
& \geq&  \E\left(\frac{\int_S w(u,v) X(u,v)dudv}{\int_{Q_{\delta/2}}w(u,v)dudv} \  | \  \mathcal{F} \cap \mathcal{G}\right) \P(\mathcal{F} \cap \mathcal{G}) \nonumber \\
&\overset{\rm (vi)}{\geq}&  \E\left(\frac{\int_{S} w(u,v) X(u,v)dudv}{\int_{Q_{\delta/2}} p_{u, \sigma}dudv- 2 \frac{\sigma^2}{\delta} -|\log(\sigma)|\delta^{3/2}} \  | \  \mathcal{F} \cap \mathcal{G} \right) \P(\mathcal{F}\cap \mathcal{G}) \nonumber \\
&\geq& \E\left(\frac{\int_{S} w(u,v) X(u,v)dudv}{\int_{Q_{\delta/2}} p_{u, \sigma}dudv- 2 \frac{\sigma^2}{\delta} -|\log(\sigma)|\delta^{3/2}} \right)-  \P(\mathcal{F}^c \cup \mathcal{G}^c) \hspace{5cm} \nonumber \\
&= & \left(\frac{\int_{S} w(u,v) f(u,v)dudv}{\int_{Q_{\delta/2}} p_{u, \sigma}dudv- 2 \frac{\sigma^2}{\delta} -|\log(\sigma)|\delta^{3/2}} \right)-  \P(\mathcal{F}^c \cup \mathcal{G}^c) \nonumber \\
&\overset{\rm (vii)}{\leq} &\left(\frac{\int_{Q^1_{\delta/2}} w(u,v) dudv + \sigma^2}{\int_{Q_{\delta/2}} p_{u, \sigma}dudv- 2 \frac{\sigma^2}{\delta} -t\delta^2} \right)-  \P(\mathcal{F}^c \cup \mathcal{G}^c) \nonumber \\
&\overset{\rm (viii)}{\leq} &\left(\frac{\int_{Q^1_{\delta/2}} p_{u, \sigma} dudv+ 2 \frac{\sigma^2}{\delta} +|\log(\sigma)|\delta^{3/2} + \sigma^2}{\int_{Q_{\delta/2}} p_{u, \sigma}dudv- 2 \frac{\sigma^2}{\delta} -t\delta^2} \right)-  \P(\mathcal{F}^c \cup \mathcal{G}^c) \nonumber \\
&\geq& \E\left(\frac{\int_{Q^1_{\delta/2}} p_{u, \sigma} dudv+ 2 \frac{\sigma^2}{\delta} +|\log(\sigma)|\delta^{3/2} + \sigma^2}{\sigma+2\int_{Q^1_{\delta/2}} p_{u, \sigma}dudv- 2 \frac{\sigma^2}{\delta} -|\log(\sigma)|\delta^{3/2}} \right)-  \P(\mathcal{F}^c \cup \mathcal{G}^c). 
\end{eqnarray}
Inequality (vi) is an immediate consequence of Lemma \ref{lem:appBinsidepixels}. Inequality (vii) is due to Lemma \ref{lem:appBoutsideweights}. Finally Inequality (viii) is the result of Lemma \ref{lem:appBinsidepixels}.
The minimum of the last line is achieved when $\int_{Q^1_{\delta/4}} p_{u, \sigma} dudv$ is minimized. However, we have proved in Proposition \ref{prop:weightprob} that $p(u)> p_0$ for $(u,v) \in Q^1_{\delta/n}$.  Therefore, $\int_{Q^1_{\delta/4}} p_{u, \sigma} dudv$ is lower bounded by $\Theta(\frac{\delta}{n})$ which is equal to $\Omega(\sigma)$ according to Assumption $A2$. When we substitute this optimum value in the lower bound (\ref{eq:lastone}), we see that the risk over this region is $\Theta(1)$. Therefore the bias of the NLM over the entire image is $\Omega(\frac{\delta}{n})$ or, equivalently, $\Omega(\sigma)$ according to Assumption A2.
\end{proof}

\section{Proof of Theorem \ref{thm:minimax}}\label{app:proofminimax}
 Here we focus on the case of $\alpha =2$ and use the standard technique of hypercube construction to establish the lower bound \cite{CaDo02Illposed}. Let $\phi:[0,1] \rightarrow \mathds{R}^+ \cup \{0\}$ be a two-times differentiable function with $\|\frac{d^2 \phi}{dt^2} \|_{\infty}=1$, $\phi(0) = \phi(1) =0$, $\left. \frac{d\phi}{dt} \right|_0 = \left. \frac{d\phi}{dt} \right|_{1}=0$, and $\left. \frac{d^2\phi}{dt^2} \right|_0 = \left. \frac{d^2\phi}{dt^2} \right|_{1}=0$. For $ i = 0, 1, \ldots, m-1$ define
\begin{eqnarray*}
\phi_{i,m} \triangleq \! \! \left\{\begin{array}{rl}
C m^{-2} \phi(mt- i)  & \frac{i}{m} \leq t \leq \frac{i+1}{m} , \\ 
0  \hspace{2.5cm}           & {\rm otherwise} 
\end{array}\right.
\end{eqnarray*}

Consider the function $f_0: [0,1]^2 \rightarrow [0,1]$, $f_0 \triangleq \textbf{1}_{\{t_2 < 0.5\}}$, and define $\psi_{i,m} \triangleq \textbf{1}_{\{t_2 \leq \phi_{i,m}(t_1)+ 0.5 \}} - f_0$. Finally, define
\[
\mathcal{F}_m \triangleq \{f_0 + \sum_{i=1}^m \zeta_i \psi_{i,m}(t_1,t_2), \ \ \zeta_i \in \{0,1\}  \}.
\]
\begin{figure}
\centering{
  \includegraphics[width=6.2cm]{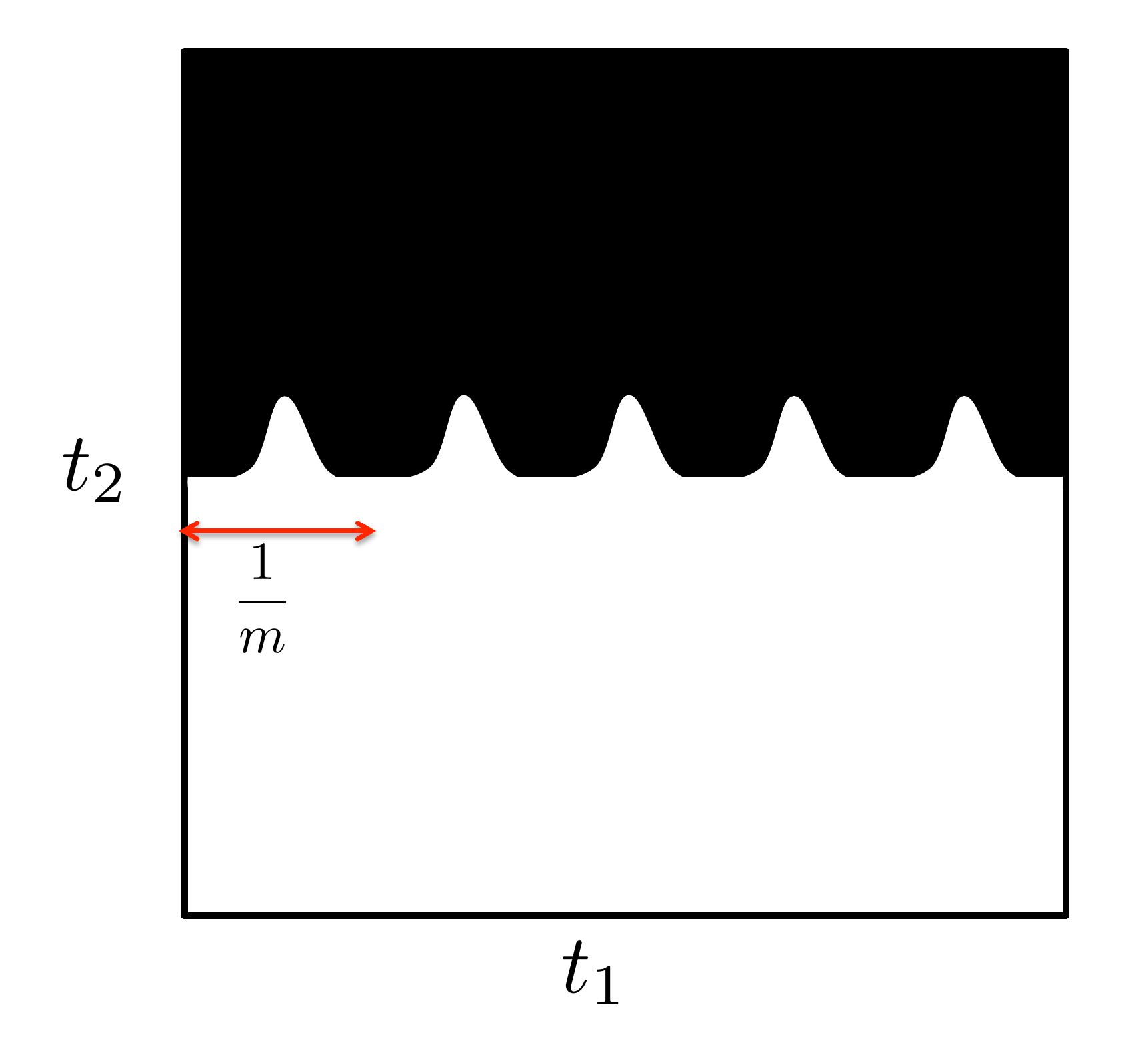} 
  \caption{A sample function from $\mathcal{F}_m$. In this figure, $\zeta_i= 1$ for $i =1, 2, \ldots, m-1$. }
   \label{fig:hypercube}
  }
\end{figure}
A sample function of this class is displayed in Figure \ref{fig:hypercube}. Since $\mathcal{F}_m \subset H^2(C)$, we have
\begin{equation} \label{eq:lower_minimax}
\inf_{\hat{f}} \sup_{H^2(C)} \E \|f - \hat{f}\|_2^2 \geq \inf_{\hat{f}} \sup_{\mathcal{F}_m} \E \|f - \hat{f}\|_2^2.
\end{equation}
The right hand side of \eqref{eq:lower_minimax} can be calculated more easily, since we can restrict our attention to the estimators of the form $f_0 + \sum_{i=1}^m \hat{\zeta}_i \psi_{i,m}$. This is due to the fact that if $P_{\mathcal{F}_m}$ is the projection onto $\mathcal{F}_m$, then for every $f \in \mathcal{F}_m$ we have
\[
\|P_{\mathcal{F}_m}(\hat{f}) - f\|_2^2 \leq \|P_{\mathcal{F}_m}(\hat{f}) - P_{\mathcal{F}_m}(f) \|_2^2 \leq \| \hat{f} - f\|_2^2. 
\] 
Furthermore for $f = f_0 + \sum_{i=1}^m \zeta_i \psi_{i,m}(t_1,t_2)$ and   $\hat{f} = f_0 + \sum_{i=1}^m \hat{\zeta}_i \psi_{i,m}(t_1,t_2)$ we have
\begin{equation}\label{eq:equivalence}
\|f- \hat{f}\|_2^2 = \kappa_m \|\zeta - \hat{\zeta}\|_2^2,
\end{equation}
where $\kappa_m = \|\psi_{i,m}(t_1,t_2)\|^2$ and satisfies
\[
\kappa_m = \|\psi_{i,m}(t_1,t_2)\|^2 = \int_{t_1} \int_{t_2}  \psi^2_{i,m}(t_1,t_2) =  \int_{t_1} \int_{t_2}  \psi_{i,m}(t_1,t_2) = \Theta(m^{-3}).
\]
According to \eqref{eq:equivalence} the original problem reduces to one of estimating $\zeta$. Therefore, we reduce the problem to the problem of estimating $\zeta_1, \ldots, \zeta_m$ from the observations $y^{\mathcal{F}}_1, \ldots, y^{\mathcal{F}}_m$ given by
\[
y^{\mathcal{F}}_i = \int \psi_{i,m}(t_1,t_2) dY(t_1,t_2) = \|\psi_{i,m}\|_2^2\zeta_i + \int \sum_i \psi_{i,m} dW(t_1,t_2).
\] 
We have
\[
y^{\mathcal{F}}_i = \kappa_m \zeta_i + w^{\mathcal{F}}_i,
\]
where $w^{\mathcal{F}}_i \overset{iid}{\sim} N(0, \sigma^2 \kappa_m)$. Consider the problem of estimating $\zeta_1$ from the observation $\zeta_1 + N(0, \sigma^2/\kappa_m)$ and set $m$ to the smallest integer for which $\sigma^2/\kappa_m \leq 1$. For this choice of $m$ the risk of any estimator is lower bounded by a constant $B$. Combining this with \eqref{eq:equivalence} we conclude that minimax risk defined in \eqref{eq:lower_minimax} is lower bounded by $B m \kappa_m $. It is straightforward to confirm that $m=\sigma^{-2/3}$ satisfies the condition $\sigma^2/\kappa_m \leq 1$. Therefore we finish the proof by setting $m=\sigma^{-2/3}$ and  $\kappa = \sigma^2$.

\bibliographystyle{model1b-num-names}
\bibliography{nonlocal,estimation,wedgelet}






% \end{thebibliography}

\end{document}